\documentclass[11pt,letterpaper]{amsart}
\newif\ifextraappendix
\newif\ifextraextraappendix

\usepackage{hyperref}
\hypersetup{hypertexnames = false, bookmarksdepth = 2, bookmarksopen = true, colorlinks, linkcolor = black, citecolor = black, urlcolor = black, pdfstartview={XYZ null null 1}}
\usepackage{enumitem}
\usepackage{amsmath}
\usepackage{amsthm}

\usepackage{amsfonts}
\usepackage{amssymb}
\usepackage{mathrsfs}
\usepackage{wasysym}

\usepackage{graphicx}
\usepackage{multicol}
\usepackage{leftidx}
\usepackage{bm}
\usepackage{faktor}

\usepackage[left=3.5cm,right=3.5cm, top=2.2cm,bottom=2.3cm]{geometry}

\usepackage{tikz-cd} 
\usepackage[all,cmtip]{xy}	
\xyoption{arrow}
\usepackage{dynkin-diagrams}

\usepackage{xcolor,etoolbox}
\usepackage{enumitem}
\usepackage{adjustbox}

\usepackage{empheq}
\usepackage[most]{tcolorbox}

\newtcbox{\mymath}[1][]{%
nobeforeafter, math upper, tcbox raise base, enhanced, colframe=blue!30!black, colback=blue!30, boxrule=1pt,
    #1}

\DeclareFontFamily{U}{rcjhbltx}{}
\DeclareFontShape{U}{rcjhbltx}{m}{n}{<->s*[1.3]rcjhbltx}{}
\DeclareSymbolFont{hebrewletters}{U}{rcjhbltx}{m}{n}

\let\aleph\relax\let\beth\relax
\let\gimel\relax\let\daleth\relax

\DeclareMathSymbol{\aleph}{\mathord}{hebrewletters}{39}
\DeclareMathSymbol{\beth}{\mathord}{hebrewletters}{98}
\DeclareMathSymbol{\gimel}{\mathord}{hebrewletters}{103}
\DeclareMathSymbol{\daleth}{\mathord}{hebrewletters}{100}

\DeclareMathSymbol{\lamed}{\mathord}{hebrewletters}{108}
\DeclareMathSymbol{\mem}{\mathord}{hebrewletters}{109}
\DeclareMathSymbol{\ayin}{\mathord}{hebrewletters}{96}
\DeclareMathSymbol{\tsadi}{\mathord}{hebrewletters}{118}
\DeclareMathSymbol{\qof}{\mathord}{hebrewletters}{113}
\DeclareMathSymbol{\shin}{\mathord}{hebrewletters}{152}

\usepackage{array}
\usepackage{multirow}

\usepackage{yhmath} 

\usepackage[capitalise]{cleveref}

\theoremstyle{plain}

\newtheorem{lem}[subsubsection]{Lemma}
\newtheorem{prop}[subsubsection]{Proposition}
\newtheorem{corollary}[subsubsection]{Corollary}
\newtheorem{thm}[subsubsection]{Theorem}
\newtheorem{cor}[subsubsection]{Corollary}
\newtheorem*{thmSewing}{Theorem~(\ref{thm:Sewing})}
\newtheorem*{corB}{Corollary~(\ref{cor:SewingAndFactorizationCorollary})}
\newtheorem*{thmBig}{Theorem~(\ref{thm:Big})}

\newtheorem*{thm*}{Theorem}
\newtheorem*{lem*}{Lemma}
\newtheorem*{cor*}{Corollary}
\newtheorem*{prop*}{Proposition}

\theoremstyle{definition}
\newtheorem{defn}[subsubsection]{Definition}
\newtheorem{claim}[subsubsection]{Claim}
\newtheorem{deflem}[subsubsection]{Definition/Lemma}
\newtheorem{notation}[subsubsection]{Notation}
\newtheorem{question}[subsubsection]{Question}
\newtheorem*{remark*}{Remark}
\newtheorem{remark}[subsubsection]{Remark}
\newtheorem{warning}[subsubsection]{Warning}
\newtheorem{ex}[subsubsection]{Example}

\newtheorem*{question*}{Question}

\newcommand{\ZZ}{\mathbb{Z}}

\newcommand{\RR}{\mathbb{R}}
\newcommand{\NN}{\mathbb{N}}
\newcommand{\CC}{\mathbb{C}}

\newcommand{\VV}{\mathbb{V}}

\newcommand{\UV}{\mathscr{U}}

\newcommand{\UVR}{\mathscr{U}^{\mathsf{R}}}
\newcommand{\UVL}{\mathscr{U}^{\mathsf{L}}}

\newcommand{\UL}{\mathscr{U}^{\mathsf{L}}}

\newcommand{\UuR}{\Uu^{\mathsf{R}}}
\newcommand{\UuL}{\Uu^{\mathsf{L}}}

\newcommand{\N}{\operatorname{N}\!}
\newcommand{\cN}{\mbox{}^c\!\operatorname{N}\!}

\newcommand{\NL}[1]{\N_\mathsf{L}^{#1}}
\newcommand{\NR}[1]{\N_\mathsf{R}^{#1}}
\newcommand{\Nn}[1]{\N^{#1}}

\newcommand{\hUu}{\widehat{\Uu}}
\newcommand{\hUuR}{\hUu^{\mathsf{R}}}
\newcommand{\hUuL}{\hUu^{\mathsf{L}}}

\newcommand{\LV}{{\mathfrak{L}}(V)}
\newcommand{\LVf}{{\mathfrak{L}}(V)^{\mathsf{f}}}

\newcommand{\LVR}{\LV^{\mathsf{R}}}
\newcommand{\LVL}{\LV^{\mathsf{L}}}

\newcommand{\PhiL}{\Phi^\mathsf{L}}
\newcommand{\PhiR}{\Phi^\mathsf{R}}

\newcommand{\coker}{\operatorname{coker}}

\newcommand{\dds}{d/ds}

\newcommand{\bMgn}[1][1]{{\overline{\mathcal{M}}_{g,#1}}}

\newcommand{\tMgn}[1][1]{{\widetriangle{\mathcal{M}}_{g,#1}}}

\makeatletter
\newcommand{\ostar}{\mathbin{\mathpalette\make@circled\star}}

\newcommand{\make@circled}[2]{%
  \ooalign{$\m@th#1\smallbigcirc{#1}$\cr\hidewidth$\m@th#1#2$\hidewidth\cr}%
}

\newcommand{\smallbigcirc}[1]{%
  \vcenter{\hbox{\scalebox{0.7}{$\m@th#1\bigcirc$}}}%
}
\makeatother

\newcommand{\Hom}{\mathrm{Hom}}

\newcommand{\Vir}{\mathrm{Vir}}

\newcommand{\Spec}{\operatorname{\mathrm{Spec}}}

\newcommand{\val}{\mathsf{val}}
\newcommand{\id}{\mathrm{id}}

\newcommand{\Ac}{\mathfrak{A}}

\newcommand{\Oc}{{\mathcal{O}}}
\newcommand{\Lc}{\mathcal{L}}

\newcommand{\Aa}{\mathsf{A}}
\newcommand{\xx}{\mathfrak{a}}
\newcommand{\yy}{\mathfrak{b}}
\newcommand{\uu}{\mathfrak{u}}

\newcommand{\Os}{\mathscr{O}}
\newcommand{\Cs}{\mathscr{C}}
\newcommand{\Ws}{\mathscr{W}}
\newcommand{\Ls}{\mathscr{L}}

\newcommand{\Uu}{\mathit{U}}

\newcommand{\Is}{\mathscr{I}}

\newcommand{\mf}[1]{\mathfrak #1}
\newcommand{\mc}[1]{\mathcal #1}
\newcommand{\ms}[1]{\mathscr #1}

\newcommand{\ov}[1]{\overline{#1}}
\newcommand{\til}[1]{\widetilde{#1}}
\newcommand{\wh}{\widehat}

\newcommand{\End}{\operatorname{End}}

\newcommand{\im}{\operatorname{im}}

\newcommand{\innerstar}{\ostar}

\newcommand{\grcomp}[1]{
	{
		{\wh{#1}}^{\textsf{g}}
	}
}
\newcommand{\filcomp}[1]{
	{
		{\wh{#1}}^{\textsf{f}}
	}
}

\newcommand{\fctensor}{\filcomp{\otimes}}
\newcommand{\gctensor}{\grcomp{\otimes}}
\newcommand{\gr}{\operatorname{gr}}
\newcommand{\Cl}[1]{\CC(\!({#1})\!)}

\makeatletter
\newcommand\opteq[1]{\mathrel{\mathpalette\opt@eq{#1}}}
\newcommand{\opt@eq}[2]{%
  \begingroup
  \sbox\z@{$#1#2$}%
  \sbox\tw@{\resizebox{!}{.5\ht\z@}{$\m@th#1($}}%
  \nonscript\hskip-\wd\tw@
  \mkern1mu
  \raisebox{-.35\ht\z@}[0pt][0pt]{\resizebox{!}{.5\ht\z@}{$\m@th#1($}}%
  \mkern-1mu
  {#2}%
  \mkern-1mu
  \raisebox{-.35\ht\z@}[0pt][0pt]{\resizebox{!}{.5\ht\z@}{$\m@th#1)$}}%
  \mkern1mu
  \nonscript\hskip-\wd\tw@
  \endgroup
}
\makeatother

\usepackage[textsize=footnotesize, textwidth=25mm, color=green!40]{todonotes}

\begin{document}
\extraextraappendixfalse
\extraappendixtrue

\title[Smoothings and mode transition algebras]{Conformal blocks on smoothings\\ via mode transition algebras
}

\subjclass[2020]{14H10, 17B69 (primary), 81R10, 81T40,14D21 (secondary)} \keywords{Vertex algebras, factorization, sewing, conformal blocks, vector bundles on moduli of curves, logarithmic conformal field theory}

\begin{abstract}

Here we introduce a series of associative algebras attached to a vertex operator algebra $V$, called mode transition algebras, and show they reflect both algebraic properties of $V$  and geometric constructions on moduli of curves.   Pointed and coordinatized curves, labeled by modules over $V$, give rise to sheaves of coinvariants.
We show that if the mode transition algebras admit multiplicative identities satisfying certain natural properties (called strong identity elements), then these sheaves deform as wanted on families of curves with nodes. This provides new contexts in which coherent sheaves of coinvariants  form vector bundles. We also show that mode transition algebras carry information about higher level Zhu algebras and generalized Verma modules. To illustrate, we explicitly describe the higher level Zhu algebras of the Heisenberg vertex operator algebra, proving a conjecture of Addabbo--Barron.
\end{abstract}

{
\author[C.~Damiolini]{Chiara Damiolini}
\address{Chiara Damiolini \newline \indent  Department of Mathematics, University of Texas at Austin,  Austin, TX  78712}
\email{chiara.damiolini@austin.utexas.edu}}
{
\author[A.~Gibney]{Angela Gibney}
\address{Angela Gibney \newline  \indent  Department of Mathematics, University of Pennsylvania,  Phil, PA 19104}
\email{agibney@math.upenn.edu}}
{
\author[D.~Krashen]{Daniel Krashen}
\address{Daniel Krashen \newline  \indent  Department of Mathematics, University of Pennsylvania,  Phil, PA 19104}
\email{dkrashen@math.upenn.edu}}

\maketitle

Mode transition algebras, introduced here, are a series of associative algebras that give insight into algebraic structures on moduli of stable pointed curves, and representations of the vertex operator algebras from which they are derived.  

Modules over vertex operator algebras (VOAs for short) give rise to vector bundles of coinvariants on moduli of smooth, pointed, coordinatized curves \cite{bzf}. To extend these to singular curves, coinvariants must deform as expected on smoothings of nodes, maintaining the same rank for singular curves as for smooth ones.  By \cref{thm:Sewing}, this holds when coinvariants form coherent sheaves and the mode transition algebras (defined below) admit multiplicative identities with certain properties.  Consequently, by \cref{cor:SewingAndFactorizationCorollary} one obtains a potentially rich source of vector bundles, including as in \cref{rmk:rational}~(\ref{it:rat}), the well-known class given by rational and $C_2$-cofinite VOAs \cite{tuy, bfm, NT,  DGT2}, and by \cref{ex:HeisVB}, a new family on moduli of stable pointed rational curves from modules over the Heisenberg VOA, which is neither $C_2$-cofinite nor rational.  Vector bundles are valuable---their characteristic classes, degeneracy loci, and section rings have been instrumental in the understanding of  moduli of curves (e.g.~\cite{HM, Mumford, HarrisEisenbud,  ELSV, FarkasVB, bchm}).

Known as essential to the study of the representation theory of VOAs, basic questions about the structure of higher level Zhu algebras remain open. Via \cref{thm:Big}, the mode transition algebras also give a new perspective on these higher level Zhu algebras.  As an application, we prove \cite[Conjecture 8.1]{addabbo.barrow:level2Zhu}, thereby giving an explicit description of the higher level Zhu algebras for the Heisenberg VOA. This is done in \cref{sec:Heisenberg} by analyzing the mode transition algebras associated to this VOA.

To describe our results more precisely,  we set a small amount of notation, with more details given below.  We assume that $V$ is a vertex operator algebra of CFT type.  While they have applications in both VOA theory and algebraic geometry, we begin by describing the geometric problem, which motivated the definition of the mode transition algebras.  By \cite{DGK}, the sheaves of coinvariants are coherent when defined by modules over a $C_2$-cofinite VOA. By \cite{DG}, coherence is also known to hold for some sheaves given by representations of VOAs that are $C_1$-cofinite and not $C_2$-cofinite.  It is natural then to ask when such coherent sheaves are vector bundles, as they were shown to be if $V$ is both $C_2$-cofinite and rational  \cite{tuy, bfm, NT, DGT2}. 

One may check coherent sheaves are locally free by proving they are flat.  This may be achieved using Grothendieck's valuative criteria of flatness.  The standard procedure first carried out by \cite[Theorem 6.2.1]{tuy}, and then followed in \cite[Theorem 8.4.5]{NT}, and \cite[VB Corollary]{DGT2}), is to argue inductively, using the factorization property (\cite[Theorem 6.2.6]{tuy}, \cite[Theorem 8.4.3]{NT}, \cite[Theorem 7.0.1]{DGT2}, \cite[Theorem 3.2]{DGK}).  However, here factorization is not available, since we do not assume that Zhu's algebra $A(V)$ is finite and semi-simple (\cite[Proposition 7.1]{DGK}). 

Instead, as explained in the proof of \cref{cor:SewingAndFactorizationCorollary}, the geometric insight here, is that in place of factorization, one may show that ranks of coinvariants are constant as nodes are smoothed in families.  This relies on the mode transition algebras admitting multiplicative identities that also act as identity elements on modules (we call these {\em{strong identity elements}}). The base case then follows from the assumption of coherence and that sheaves of coinvariants support a projectively flat connection on moduli of smooth curves \cite{bzf, dgt1}.  We refer to this process as smoothing of coinvariants.
Let us now summarize the strategy. 

For simplicity, let $\Cs_0$ be a projective curve over $\CC$ with a single node $Q$, $n$ smooth points  $P_\bullet=(P_1, \ldots, P_n)$, and formal coordinates $t_\bullet=(t_1,\ldots, t_n)$ at $P_\bullet$. Let $W^1, \ldots ,W^n$ be an $n$-tuple of $V$-modules, or equivalently of smooth $\mathscr{U}$-modules, where $\mathscr{U}$ is the universal enveloping algebra associated to $V$ (defined in detail in \cref{sec:UnivEnv}).

We assume each $V$-module $W^i$ is generated by a module $W^i_0$ over (the zero level) Zhu algebra $\Aa_0(V):= \Aa$. The vector space of coinvariants $[W^\bullet]_{(\Cs_0, P_\bullet, t_\bullet)}$ is the largest quotient of $W^\bullet= W^1\otimes \cdots \otimes W^n$ on which the Chiral Lie algebra $\Lc_{\Cs_0\setminus P_\bullet}(V)$ acts trivially (described here in \cref{sec:LimAndCoinv}).  Coinvariants at $\Cs_0$ are related to those on the normalization  $\eta:\widetilde{\Cs}_0\rightarrow \Cs_0$ of $\Cs_0$ at $Q$. Namely, by \cite{DGT2} the map
\begin{equation*}\label{eq:DGK0}
\alpha_0 \colon  W^\bullet_0{\overset{}{\rightarrow}} W^\bullet_0 \otimes \Aa, \qquad u \mapsto u \otimes 1^{\Aa},
\end{equation*}
gives rise to an $\Lc_{\Cs_0\setminus P_\bullet}(V)$-module map, inducing a map between spaces of coinvariants
\begin{equation*}\label{eq:DGK}
[\alpha_0] \colon [W^\bullet]_{\left(\Cs_0, P_\bullet, t_\bullet\right)}{\overset{}{\rightarrow}} [W^\bullet \otimes \Phi(\Aa)]_{\left(\widetilde{\Cs}_0, P_\bullet \sqcup Q_\pm, t_\bullet \sqcup s_\pm\right)}.
\end{equation*}
Here $\Phi(A)$ is a $\ms{U}$-bimodule assigned at points $Q_\pm$ lying over $\eta^{-1}(Q)$, and $s_\pm$  are formal coordinates at $Q_\pm$.  By \cite{DGK}, the map $[\alpha_0]$ is an isomorphism if $V$ is $C_1$-cofinite.

One may extend the nodal curve $\Cs_0$  to a smoothing family $(\Cs, P_\bullet, t_\bullet)$ over the scheme $S={\Spec}(\CC[\![q]\!])$, with special fiber  $(\Cs_0, P_\bullet, t_\bullet)$, and  smooth generic fiber, while one may trivially extend $\widetilde{\Cs}_0$  to a family $(\widetilde{\Cs}, P_\bullet \sqcup Q_\pm, t_\bullet \sqcup s_\pm)$ over $S$.  While the central fibers of these two families of curves are related by normalization, there is no map between $\widetilde{\Cs}$ and $\Cs$. However, for $V$ rational and $C_2$-cofinite, sheaves of coinvariants on $S$ are naturally isomorphic, an essential ingredient in the proof that such sheaves are locally free under these assumptions \cite{DGT2}.  

To obtain an analogous isomorphism of coinvariants under less restrictive conditions, our main idea is to generalize the algebra structure of $X_d\otimes X_d^\vee\subset \Phi(\Aa)$, which exists for simple admissible $V$-modules $X=\bigoplus_{d} X_d$, and for all $d\in \NN$ (see \cref{rmk:rational}).  Namely, we show that $\Phi(\Aa)$ has the structure of a bi-graded algebra, which we call the {\em{mode transition algebra}} and denote $\Ac=\bigoplus_{d_1,d_2\in \mathbb{Z}}\Ac_{d_1,d_2}$. 
We show that $\Ac$ acts on generalized Verma modules $\PhiL(W_0)=\bigoplus_{d}W_d$, such that the subalgebras $\Ac_d:=\Ac_{d,-d}$, which we refer to as the {\em{$d$th mode transition algebras}}, act on the degree $d$ components $W_d$ (these terms are defined in \cref{sec:PhiL} and \cref{sec:ModeAlgebra}). 

We say that $V$ satisfies {\text{smoothing}}  (\cref{def:Sewing2}), if for every pair $(W^\bullet,\Cs_0)$, 
consisting of $n$ admissible $V$-modules $W^\bullet$, not all trivial, a stable $n$-pointed curve  $\Cs_0$  with a node, there exist an element $\Is=\sum_{d \geq 0} \Is_d q^d \in \Ac[\![q]\!]$, such that the map 
\begin{equation*}\label{eq:sewIntro}
\alpha \colon W^{\bullet}[\![q]\!] \longrightarrow   (W^{\bullet} \otimes \Ac)[\![q]\!],   \qquad
u  \mapsto  \, u \otimes \Is , \end{equation*}
is an $\Lc_{\Cs\setminus P_{\bullet}}(V)$-module homomorphism which extends $\alpha_0$.

In  \cref{thm:Sewing}, we equate smoothing for $V$ with a property of multiplicative identity elements in the $d$th mode transition algebras $\Ac_d$, when they exist.  Specifically, if $\Ac_d$ admit identity elements $\Is_d \in \Ac_d$ for all $d \in \NN$, satisfying any of the equivalent properties of \cref{lem:StrongUnitConditions}, then we say that $\Is_d \in \Ac_d$ is a \textit{strong identity element}. 

\begin{thmSewing}
Let $V$ be a VOA of CFT-type.  The algebras $\Ac_d=\Ac_d(V)$ admit strong identity elements for all $d\in \NN$ if and only if $V$ satisfies smoothing. 
\end{thmSewing}

We remark that the analogue to $\alpha$ is called the sewing map in \cite{tuy, NT, DGT2}. 
As an application of \cref{thm:Sewing}, we obtain geometric consequences stated as \cref{cor:sewing} and \cref{cor:SewingAndFactorizationCorollary}.  A particular case of which is  as follows:


\begin{corB}If $V$ is $C_2$-cofinite and satisfies smoothing, then $\VV(V;W^\bullet)$ is a vector bundle on $\bMgn[n]$ for simple $V$-modules $W^1, \dotsm, W^n$.
\end{corB}
By \cref{rmk:rational}, rational VOAs satisfy smoothing,  so \cref{cor:SewingAndFactorizationCorollary} specializes to \cite[VB Corollary]{DGT2}.  As is shown in \cref{ex:HeisVB}, one can apply the full statement of \cref{cor:SewingAndFactorizationCorollary} to show that modules over VOAs derived from  Heisenberg lie algebras of  dimension one (which are $C_1$-cofinite, but neither $C_2$-cofinite nor rational), define vector bundles on moduli of stable pointed rational curves (see \cref{ex:HeisVB}).

 \cref{thm:Big}, described next, gives further tools for investigating other VOAs which may or may not satisfy smoothing by providing information about the relationship between mode transition algebras and higher level Zhu algebras and their representations.

 Recall that in \cite{ZhuMod} Zhu defines a two step induction functor, which in the first part takes $\Aa=\Aa(V)$-modules to a $V$-modules through a Verma module construction, and then in the second step takes a quotient.   In \cref{{def:PhiL}} we describe this first step with a different, although naturally isomorphic functor $\PhiL$, a crucial ingredient to this work.  Through this functor (naturally isomorphic to) $\PhiL$, Zhu shows that there is a bijection between simple $\Aa$-modules and simple $V$-modules, so that if $\Aa$ is finite dimensional, and semi-simple, $\PhiL$ describes the category of admissible $V$-modules.  However, if $\Aa$ is either not finite-dimensional or is not semi-simple, then there are indecomposable, but non-simple $V$-modules not induced from simple indecomposable modules over $\Aa$ via $\PhiL$. To describe such modules,  \cite{DongLiMason:HigherZhu} defined the higher level Zhu algebras $\Aa_d$ for $d\in \NN$ , further studied in \cite{BVY}.

The mode transition algebras $\Ac_d$ are related to the higher level Zhu algebras $\Aa_d$.
For instance,   $\Ac_0=\Aa_0=\Aa$ (\cref{rmkWhyEqualA}), and by \cref{lem:right exact seq}, there is an exact sequence 
\begin{equation}\label{eq:SES}
\xymatrix{\Ac_d \ar[r]^-{\mu_d} & \Aa_d \ar[r]^-{\pi_d} & \Aa_{d-1}\ar[r]&  0.} \end{equation}
When $\Ac_d$ admits an identity element (and not necessarily a strong identity element), $\mu_d$ injective and the sequence splits by Part (a) of \cref{thm:Big}. In particular, if  $V$ is $C_2$-cofinite, as observed in  \cite{BuhlSpanning, GN,  MiyamotoC2, He},  the $d$th Zhu algebras $\Aa_d$ are finite dimensional, hence if $\Ac_d$ has an identity element, it is finite dimensional as well.

We note that \eqref{eq:SES} may be exact when $\Ac_d$ does not admit an identity element. For instance, in \cref{eg:virasoro} we show exactness of \eqref{eq:SES}  when $d=1$ for the Virasoro VOA $\Vir_c$, for any values of $c$, and use it to show that $\Ac_1$ does not admit an identity element. In particular, by  \cref{thm:Sewing}, one finds that $\Vir_c$ never satisfies smoothing. When $c$ is in the discrete series, the maximal simple quotient $L_c$ of $\Vir_c$ is rational and $C_2$-cofinite, hence by  \cref{rmk:rational}~\ref{it:rat} (or by \cite{DGT2}) it will satisfy smoothing.


 \cref{thm:Big} allows one to  use the mode transition algebras to obtain other valuable structural information about the higher level Zhu algebras.  

\begin{thmBig}
\begin{enumerate}
\item If the $d$th mode transition algebra $\Ac_d$ admits an identity element, then \eqref{eq:SES} is split exact, and $\Aa_{d} \cong \Ac_d \times \Aa_{d-1}$ as rings. In particular, if $\Ac_d$ admits an identity element for every $d \in \mathbb Z_{\geq 0}$ then $\Aa_d \cong \Ac_d \oplus \Ac_{d-1} \oplus \cdots \oplus \Ac_0.$
\item For $\Ac=\Ac(V)$, if   $\Ac_d$ admits a strong identity for all $d\in \NN$, so that smoothing holds for $V$, then  given any generalized Verma module $W=\PhiL(W_0)=\bigoplus_{d\in \NN} W_d$ where $L_0$ acts on $W_0$ as a scalar with eigenvalue $c_W \in \CC$, there is no proper submodule $Z \subset W$ with $c_Z - c_W\in \ZZ_{>0}$ for every eigenvalue $c_Z$ of $L_0$ on $Z$. 
\end{enumerate}
\end{thmBig}
We refer to \cref{sec:PhiL} for a discussion about generalized Verma modules. 

 We note that by \cref{lem:right exact seq} and  \cref{abstract splitting}, the exact sequence \eqref{eq:SES} as well as Part \textit{(a)} of \cref{thm:Big} hold for generalized higher Zhu algebras and generalized $d$th mode transition algebras (see \cref{def:GHZ} and \cref{def:GoodTripleMTA notriple}).  For further discussion  see \cref{sec:GoodTriples}.

We now describe some further consequences of \cref{thm:Sewing} and \cref{thm:Big}.

In \cref{sec:Heisenberg} we describe the $d$th mode transition algebras $\Ac_d$  for the Heisenberg vertex algebra $M_a(1)$ defined by a one dimensional Heisenberg Lie algebra (denoted $\pi$ in \cite{bzf}), and show the $\Ac_d$ admit strong identity elements for all $d\in \NN$. In particular, \cref{thm:Big} and \cref{claim:HigherHeis} imply that the conjecture of Addabbo and Barron \cite[Conj 8.1]{addabbo.barrow:level2Zhu} holds, and one can write
\begin{equation}\label{eq:ABCon}\Aa_d(\pi)= \Aa_d(M_a(1))\cong  \prod_{j=0}^d\mathrm{Mat}_{p(j)}(\CC[x]),\end{equation}
where $p(j)$ is the number of ways to decompose $j$ into a sum of positive integers, with $p(0)=1$. The level one Zhu algebra $A_1(M_a(1))$ was first constructed in the paper \cite{BVYFirst}, and then later announced in \cite{BVY}. In \cite{ABCon} the authors determine $A_2(M_a(1))$ using the infrastructure for finding generators and relations for higher level Zhu algebras they had developed in \cite{addabbo.barrow:level2Zhu}.

To illustrate their computability, we note that in \cite{DGK3}  the $d$-th mode transition algebras for VOAs derived from Heisenberg Lie algebras of arbitrary rank are shown to admit strong identities for all $d$, and further results developed there allow us to compute higher Zhu algebras for these and other examples.

In \cref{sec:EvidenceTriplet}, we use Part (b) of \cref{thm:Big} to show that the family of triplet vertex operator algebras $\mathcal{W}(p)$ does not satisfy smoothing.  We do this by giving an explicit pair of modules $W\subset Z$ with $W=\PhiL(W_0)$ and such that $c_Z > c_W$.  The actual pair of modules used was suggested to us by Thomas Creutzig (with some details filled in by Simon Wood). Dra\v{z}en Adamovi\'c had also sketched for us the existence of such an example.  The importance of this example is that it establishes that smoothing is not guaranteed to hold for a $C_2$-cofinite VOA if rationality is not assumed.  In particular, while sheaves of coinvariants defined by the representations of $C_2$-cofinite VOAs are coherent, this can be seen as an indication  that they may not necessarily be locally free. Taken together with the family of Heisenberg vector bundles from \cref{ex:HeisVB}, this example illustrates the subtlety of the problem of determining which sheaves of coinvariants define vector bundles.

We also expect that smoothing will not hold for the family of symplectic fermion algebras ${\mathrm{SF}^+_d}$ which are $C_2$-cofinite and not rational, since ${\mathrm{SF}^+_1}=\mathcal{W}(2)$.  It is natural to ask whether there is an example of a vertex operator algebra that is $C_2$-cofinite, is not rational, and satisfies smoothing (see \cref{sec:SecondQuestion}).   

Finally, we emphasize that our procedure to use smoothing to show that sheaves of coinvariants are locally free is just one approach to this problem (see \cref{sec:BundleQuestions}).

\subsection*{Plan of the paper} \ In \cref{sec:Background},  we set the terminology used here for vertex operator algebras and their representations.  In \cref{sec:UnivEnv}, we provide detailed descriptions of the universal enveloping algebra $\UV$ associated to a vertex operator algebra $V$. Technical details are given in \cref{sec:Appendix}, where an axiomatic treatment of the constructions of the graded and filtered enveloping algebras as topological or semi-normed algebras is given.  The concepts discussed involving filtered and graded completions can be found throughout the VOA literature (for instance in \cite{tuy, FrenkelZhu, bzf,FrenkelLanglands, NT, MNT}), but little is said about how they relate to one another.  We discuss these relations in \cref{sec:UnivEnv}. 
In \cref{sec:ZhuFunctor} we give an alternative construction of the generalized Verma module functor $\PhiL$ (and the right-analogue $\PhiR$) from the category of $\Aa$-modules, to the category of smooth left (and right) $\UV$-modules. We use a combination of $\PhiL$ and $\PhiR$ to define the mode transition algebras $\Ac_d \subset \Ac$.  More general versions of these constructions are defined in \cref{sec:Generalized}, where their analogous properties are proved. In \cref{sec:LimAndCoinv}, smoothing is formally defined, and we describe sheaves of coinvariants on families of pointed and coordinatized curves in general terms, and cite the relevant references.  In  \cref{sec:SewingThmProof} we prove \cref{thm:Sewing}, \cref{cor:sewing}, and \cref{cor:SewingAndFactorizationCorollary}.
In \cref{sec:BigThmProof} we  prove \cref{thm:Big} Part \textit{(b)}, while Part \textit{(a)} is detailed in \cref{sec:Generalized}.  In \cref{sec:Heisenberg} we compute the mode transition algebras $\Ac_d$ for the Heisenberg algebra for all $d$. In \cref{eg:virasoro} we compute the 1st mode transition algebras for the non-discrete series Virasoro VOAs. We ask a number of questions in \cref{sec:OtherQuestions}.  In \cref{sec:SecondQuestion} and in \cref{sec:BundleQuestions}  questions are discussed about  $C_2$-cofinite and non-rational VOAs that may not satisfy smoothing, and whether their induced sheaves of coinvariants may still  define vector bundles.  Finally, as noted, many of the results here are stated and proved for generalizations of higher level Zhu algebras and of mode transition algebras, and in \cref{sec:GoodTriples} we raise the question of finding other examples and applications of such algebraic structures, beyond those naturally associated to a vertex operator algebra.

\subsection*{Acknowledgements} We would like to thank Dra\v{z}en Adamovi\'c, Katrina Barron, Thomas Creutzig,  Haisheng Li, Jianqi Liu, Antun Milas, Rahul Pandharipande, Dennis Sullivan, and Simon Wood for helpful discussions and for answering our questions. Thomas Creutzig and Simon Wood gave crucial assistance with the details of \cref{sec:EvidenceTriplet}. We thank Katrina Barron for her valuable feedback at various points. Most of all, we thank  Yi-Zhi Huang who encouraged us from the beginning to consider questions regarding sewing, when rationality is not assumed.  Gibney was supported by NSF DMS -- 2200862, and Krashen was supported by NSF DMS--1902237.

\section{Background on VOAs and their modules}
\label{sec:Background}
In \cref{sec:VOAsAndModules} we state the conventions we follow for vertex operator algebras and their representations.  Throughout this paper, by an algebra we mean an associative algebra which is not necessarily commutative and by a ring we mean an algebra over $\ZZ$.   We refer to \cite{FrenkelZhu,ZhuMod,bfm,NT} for more details about vertex operator algebras and their modules.

\subsection{VOAs and their representations}\label{sec:VOAsAndModules}
We recall here the definition of a vertex operator algebra of CFT type, which in the paper will be simply denoted \textit{VOA}. 
\begin{defn} A {\textit{vertex operator algebra of CFT-type}} is four-tuple $(V, {\textbf{1}}, \omega, Y(\cdot,z))$:
\begin{enumerate}
\item $V=\bigoplus_{i\in \NN} V_i$ is a vector space with $\dim V_i<\infty$, and $\dim V_0 = 1$;
\item ${\textbf{1}}$ is an element in $V_0$, called the \textit{vacuum vector};
\item $\omega$ is an element in $V_2$, called the \textit{conformal vector};
\item $Y(\cdot,z)\colon V \rightarrow  \textrm{End}(V)[\![ z,z^{-1} ]\!]$ is a linear map
  $a \mapsto  Y(a,z) :=\sum_{m\in\mathbb{Z}} a_{(m)}z^{-m-1}$.   The series $Y(a,z)$ is called the \textit{vertex operator}  assigned to $a\in V$,
\end{enumerate}
satisfying the following axioms:
\begin{enumerate}
\item \textit{(vertex operators are fields)} for all $a$, $b \in V$,  $a_{(m)}b=0$, for $m \gg 0$;
\item \textit{(vertex operators of the vacuum)}  $Y({\bm{1}}, z)=\mathrm{id}_V$, that is
\[
{\bm{1}}_{(-1)} = \mathrm{id}_V \qquad \mbox{and} \qquad {\bm{1}}_{(m)} = 0, \quad\mbox{for $m\neq -1$,}
\]
and for all $a\in V$, \ $Y(a,z){\bm{1}} \in a+ zV[\![ z ]\!]$, that is
\[
a_{(-1)}{\bm{1}} = a \qquad \mbox{and} \qquad a_{(m)}{\bm{1}} =0, \qquad \mbox{for $m\geq 0$};
\]
\item \textit{(weak commutativity)} for all $a$, $b\in V$, there exists an $N\in\NN$ such~that
\[
(z_1-z_2)^N \, [Y(a,z_1), Y(b,z_2)]=0 \quad \mbox{in }\textrm{End}(V)[\![ z_1^{\pm 1}, z_2^{\pm 1} ]\!];
\]
\item \textit{(conformal structure)} for 
$Y(\omega,z)=\sum_{m\in\mathbb{Z}} \omega_{(m)}z^{-m-1}$,
\[
\left[\omega_{(p+1)}, \omega_{(q+1)} \right] = (p-q) \,\omega_{(p+q+1)} + \frac{c}{12} \,\delta_{p+q,0} \,(p^3-p) \,\textrm{id}_V. 
\]
Here $c\in\mathbb{C}$ is  the \textit{central charge} of $V$.  Moreover:
\[
\omega_{(1)}|_{V_m}=m\cdot\textrm{id}_V, \quad \text{for all } m, \qquad \mbox{and} \qquad Y\left(\omega_{(0)}a,z \right) = \frac{d}{dz} Y(a,z).
\]
\end{enumerate}
\end{defn}


\begin{defn}
An admissible $V$-module is an $\mathbb{C}$-vector space $W$ together with a linear map
\[Y^W(\cdot,z) \colon  V \rightarrow  \operatorname{End}(W)[\![z,z^{-1}]\!], \ a \in V \mapsto Y^W(a,z) :=\sum_{m\in\mathbb{Z}} a^W_{(m)}z^{-m-1},\] which 
satisfies the following axioms:
\begin{enumerate}
    \item \textit{(vertex operators are fields)} if $a \in V$ and $u \in W$, then $a^W_{(m)}u=0$, 
    for $m \gg 0$;
    \item \textit{(vertex operators of the vacuum)} $Y^W\left({\bm{1}},z\right)=\mathrm{id}_W$;
    \item \textit{(weak commutativity)} for all $a$, $b \in V$, there exists an $N\in \NN$ such that for all $u \in W$ 
    \[
(z_1-z_2)^N \left[ Y^W(a,z_1), Y^W(b,z_2) \right]u=0;
\]
    \item \textit{(weak associativity)} for all $a \in V$ and $u \in W$,  there exists an $N\in\NN$, such that for all $b \in V$, one has
     \[
(z_1+z_2)^N \left( Y^W(Y(a,z_1)b,z_2)- Y^W(a,z_1+z_2)Y^W(b,z_2)\right)u =0;
\]
    \item \textit{(conformal structure)}\label{VirAct} for   $Y^W(\omega,z) = \sum_{m \in \mathbb{Z}} \omega_{(m)}^W z^{-m-1}$, one has
    \[
\left[\omega^W_{(p+1)}, \omega^W_{(q+1)} \right] = (p-q) \,\omega^W_{(p+q+1)} + \frac{c}{12} \,\delta_{p+q,0} \,(p^3-p) \,\textrm{id}_W. 
\]
	where $c \in \mathbb{C}$ is the central charge of $V$.  Moreover $Y^W\left(L_{-1}a,z \right) = \frac{d}{dz} Y^W(a,z)$;
    \item \textit{($\mathbb N$-gradability)} $W$ admits a grading $W = \bigoplus\limits_{n \in \NN} W_n$ with $a^W_{(m)}W_n \subset W_{n + \deg(a) - m - 1}$.
\end{enumerate}\end{defn} 
 
As one can see in the literature,  e.g. by \cite{dl, fhl,  lepli, lvertexsuperalg},  weak associativity and  weak commutativity together are equivalent to the \textit{Jacobi identity}: for $\ell$, $m$, $n \in \mathbb{Z}$, and $a$, $b\in V$ 
\begin{equation*}
    \sum_{i\ge 0}(-1)^i{\binom{\ell}{i}} \left( a^W_{(m+\ell-i)}b^W_{(n+i)}-(-1)^{\ell}b^W_{(n+\ell-i)}a^W_{(m+i)} \right)
=\sum_{i\ge 0}{\binom{m}{i}}(a_{(\ell+i)}(b))^W_{(m+n-i)}.\end{equation*}  Moreover, by \cite[Lemma 2.2]{DongLiMasonRegular}, axiom (\ref{VirAct}) is redundant. 


\section{The universal enveloping algebra of a VOA}\label{sec:UnivEnv}
Here we describe constructions of the universal enveloping algebra associated to a VOA $V$, as quotients of certain graded, as well as (left and right) filtered completions of the universal enveloping algebra of the Lie algebra associated to $V$. Filtered completions are essential to our constructions, as they are compatible with crucial restriction maps from the Chiral Lie algebra to certain ancillary Lie algebras, allowing for the definition of the action of the Chiral Lie algebra on (tensor products of) $V$-modules.  The graded completion, on the other hand, allows both for ease in computation, and simpler descriptions of induced modules, and bimodules, and in \cref{sec:ModeAlgebra} of the mode transition algebras.  While these concepts are treated in one way or another throughout the VOA literature, for instance in \cite{FrenkelZhu, bzf,FrenkelLanglands, NT,  MNT},  we provide here and  in \cref{sec:Appendix}, a uniform description, where many details are given, clarifications are made, and the different constructions are compared to one another. We further remark that, although this section assumes that $V$ is a vertex operator algebra, however all the arguments and construction here in \cref{sec:UnivEnv} hold assuming only that $V$ is a graded vertex algebra since the conformal structure does not play a role  (see also \cref{sec:GoodTriples}).

\subsection{Graded and filtered completions} We recall the constructions of the universal enveloping algebra  \cite{FrenkelZhu} and the current algebra  \cite{NT} associated to a VOA $V$. 

\subsection{Split filtrations} The underlying vector spaces of the objects we will need to consider will either be graded or filtered (sometimes both), and these filtrations and gradings will be related to each other. The basic example of this is the space of Laurent polynomials $\CC[t, t^{-1}]$ which we choose to grade by $\CC[t, t^{-1}]_n = \CC t^{-n-1}$, and the space of Laurent series $\CC(\!(t)\!)$, which admits an increasing filtration by setting $\CC(\!(t)\!)_{\leq n} := t^{-n-1}\CC[\![t]\!]$. We will refer to this filtration as a \textit{left filtration} of $\CC(\!(t)\!)$ (see \cref{filtration def}). In this situation, when we have a graded subspace of a filtered space $\CC[t, t^{-1}] \subset \CC(\!(t)\!)$ which identifies the degree $n$ part of the graded subspace with the degree $n$ part of the associated graded space, we say that this pair gives a \textit{split filtration} (see \cref{split filter def}). Similarly, we refer to the filtration of $\CC(\!(t^{-1})\!)$ given by $\CC(\!(t^{-1})\!)_{\geq n} = t^{-n-1}\CC[\![t^{-1}]\!]$ as a \textit{right filtration}, and the pair $\CC[t, t^{-1}] \subset \CC(\!(t^{-1})\!)$ is also a split filtration.

\subsection{Lie algebras} Let $V$ be a VOA (or a graded vertex algebra). Since it is a graded vector space, it admits a trivial left (respectively right) split-filtration, given by $V_{\leq n} = \bigoplus_{d \leq n} V_{d}$ (respectively $V_{\geq n} = \bigoplus_{d \geq n} V_{d}$). In view of \cref{split tensor}, tensor products of split-filtered modules are naturally split-filtered, and consequently
\[V \otimes_{\CC} \CC[t, t^{-1}] \subset V \otimes_{\CC} \CC(\!(t)\!) \quad \text{ and } \quad V \otimes_{\CC} \CC[t, t^{-1}] \subset V \otimes_{\CC} \CC(\!(t^{-1})\!)\] define splittings of their left and right filtrations.

\begin{remark} Concretely, we can define the map $\val \colon V \otimes \CC(\!(t)\!) \to \ZZ$ by
\[ \val(a \otimes f(t)) = \deg(a)-N-1.
\]for a homogeneous element $a \in V$ and $f(t) \in t^{N}\CC[\![t]\!] \setminus t^{N-1}\CC[\![t]\!]$. The natural left filtration on $V \otimes_\CC \CC(\!(t)\!)$ is then given by $(V \otimes \CC(\!(t)\!))_{\leq n} := \val^{-1}(-\infty, n]$. \end{remark}

The linear map $\nabla = L_{-1} \otimes \id + \id \otimes \frac{d}{dt}$ is a linear endomorphism of each of these spaces of degree $1$ (see \cref{split maps def}). We define
\[ \LVL = \dfrac{V \otimes_\CC \CC(\!(t)\!)}{\mathrm{Im}(\nabla)}, \quad
\LVR = \dfrac{V \otimes_\CC \CC(\!(t^{-1})\!)}{\mathrm{Im}(\nabla)},
\ \mbox{ and } \  \LVf = \dfrac{V \otimes_\CC \CC[t,t^{-1}]}{\mathrm{Im}(\nabla)}. 
\]
These have induced split-filtrations via $\LVf \subset \LVL, \LVR$ by \cref{cokernel splitting}.
These filtered and graded vector spaces admit (filtered and graded) Lie algebra structures, with Lie brackets defined by:
\[ \left[a \otimes f(t),b \otimes g(t)\right]:= \sum_{k \geq 0} \frac{1}{k!} \left(a_{(k)}(b)\right) \otimes g(t)\dfrac{d^k(f(t))}{dt^k},
\]for all $a,b \in V$ and $f(t),g(t) \in \CC(\!(t)\!)$ or $f(t),g(t) \in \CC(\!(t^{-1})\!)$. 

More concretely in the case of $\LVf$, for $a \in V$ and $i \in \ZZ$ we denote by $a_{[i]}$ the class of the element $a \otimes t^{i}$ in $\LVf \subset \LVL, \LVR$. 
The restriction of the Lie bracket on $\LVf$ then is given by the following formula
\[ \left[a_{[i]},b_{[j]}\right]:= \sum_{k \geq 0} \binom{i}{k} \left(a_{(k)}(b)\right)_{[i+j-k]},
\] for all $a,b \in V$ and $i,j \in \ZZ$. 
Extending the notation introduced in \cite{DGT2}, we call $\LVL$ and $\LVR$ the left and right \textit{ancillary Lie algebras} and $\LVf$ the \textit{finite ancillary Lie algebra}. Note that $\LVL$ is isomorphic to the \textit{current Lie algebra} $\mathfrak{g}(V)$ from \cite{NT}.

\subsection{Enveloping algebras} We now let $\UuL, \UuR, \Uu$ be the universal enveloping algebras of $\LVL, \LVR, \LVf$ respectively. These enveloping algebras are left filtered, right filtered, and graded respectively, and we note that $\Uu \subset \UuL, \UuR$ again give splittings to the respective filtrations (see \cref{universal split filtration}). In the language of \cref{good definition} we will say that $(\UuL, \Uu, \UuR)$ forms a good triple of associative algebras. 

\begin{ex} These induced filtrations can be explicitly described. For instance we have that $\Uu_d$ is linearly spanned by elements $\ell^1 \cdots \ell^k$  such that $\ell^i \in \LVf_{d_i}$ and  $\sum_{i}d_i = d$ (and $k$ possibly zero if $d=0$). Analogously $(\UuL)_{\leq d}$ is linearly spanned by elements $\ell^1 \cdots \ell^k$   such that $\ell^i \in \LV_{\leq d_i}$ and $\sum_{i}d_i = d$ (and $k$ possibly zero if $d \geq 0$).
\end{ex}

These enveloping algebras have an additional structure of a topology induced by seminorms, which can be described in terms of systems of neighborhoods of $0$ (as in \cref{seminorm def} and \cref{seminorm remark}). These neighborhoods of the identity are given by left ideals $\NL{n}\UuL$ of $\UuL$ and $\NL{n}\Uu$ of $\Uu$ and right ideals $\NR{n}\UuR$ of $\UuR$ and $\NR n \Uu$ of $\Uu$ defined by:
\[
\NL{n}\UuL= \UuL \UuL_{\leq -n}, \quad
\NL{n}\Uu = \Uu \Uu_{\leq -n},  \quad
\NR{n}\UuR = \UuR_{\geq n} \UuR,  \quad
\NR{n}\Uu= \Uu_{\geq n} \Uu. 
\]
This definition coincides with that of a canonical seminorm on a (split-)filtered algebra, as described in \cref{canonical def}, and in particular gives a good seminorm on the triple (see \cref{good seminorm def} and \cref{good seminorm remark}). Most useful for us is that the category of good triples of algebras with good seminorms is closed under quotients and completions (\cref{good quotients} and \cref{good completions}). 

These neighborhoods can also be described alternately as follows:
\begin{lem}\label{lem:Ann}
One has an identification of left ideals $\NL{n}\Uu = \Uu \Uu_{\leq -n} = \Uu \LVf_{\leq -n}$. Similarly, one has $\NL{n}\UuL= \UuL \LVL_{\leq -n}$, $\NR{n}\UuR = \LVR_{\geq n} \UuR$ and $\NR{n}\Uu = \LVf_{\geq n} \Uu$.
\end{lem}
We note that in \cref{lem:Ann}, we identify $\LVf_{\leq -n}$ with its image in $\Uu$.
\begin{proof} 
We will prove only the first equality, the others following by similar methods. Clearly $\Uu \Uu_{\leq -n}$ contains $\Uu \LVf_{\leq -n}$, so we need only prove the reverse inclusion, which amounts to showing $\Uu_{-n} \subset \Uu \LVf_{\leq -n}$. For this, suppose we have an element $\alpha = a^1_{n_1} a^2_{n_2} \cdots a^r_{n_r} \in \Uu_{-n}$. By inductively using \cite[Lemma~A.2.1]{He}, for $M >> 0$, we may write  
\(\alpha = a^1_{n_1} a^2_{n_2} \cdots a^r_{n_r} = \beta + \gamma\),
with $\beta \in \LVf_{-n}$ and $\gamma \in \Uu \LVf_{\leq -M}$. In particular, choosing $M \geq n$, we find $\gamma \in \Uu \LVf_{\leq -n}$, showing $\alpha \in \Uu \LVf_{\leq -n}$ as desired.
\end{proof}

We can restrict these seminorms to various filtered and graded parts of these algebras as in \cref{seminorm restriction def}. For example, we write $\NL{n}\UuL_{\leq p}$ to denote $\NL{n}(\UuL) \cap \UuL_{\leq p}$. Concretely we obtain systems of neighborhoods as follows (see \cref{split the products} for more details):
\[
    \NL{n}\UuL_{\leq p} = (\UuL \UuL_{\leq -n})_{\leq p} = \sum_{j\leq -n} \UuL_{\leq p-j}\UuL_{\leq j}, \qquad \NR{n}\UuR_{\geq p} = (\UuR_{\geq n} \UuR)_{\geq p} = \sum_{i\geq n} \UuR_{\geq i}\UuR_{\geq p-i}\]
\[
    \NL{n}\Uu_{p} = (\Uu \Uu_{\leq -n})_{p} = \sum_{j\leq -n} \Uu_{p-j}\Uu_{j}, \ \mbox{ and } 
    \qquad \NR{n}\Uu_{p} = (\Uu_{\geq n} \UuR)_{p} = \sum_{i\geq n} \Uu_{i}\Uu_{p-i}.
\]

We note that in particular, we have $\NR{n+p}\Uu_p = \NL{n} \Uu_p$. Through the restriction of the seminorm to these subspaces, we then define a filtered completions of $\UuL$ and $\UuR$, both containing a graded completion of $\Uu$ (see \cref{fg completion def}).
Specifically, we set 
\[\hUuL_d := \varprojlim_n \dfrac{\UuL_{\leq d}}{\NL{n}\UuL_{\leq d}}, 
\qquad \hUuR_d := \varprojlim_n \dfrac{\UuR_{\geq d}}{\NL{n}\UuR_{\geq d}},
\qquad \hUu_d := \varprojlim_n \dfrac{\Uu_{ d}}{\NL{n}\Uu_{d}} = \varprojlim_n \dfrac{\Uu_{ d}}{\NR{n+d}\Uu_{d}},
\]
And set
\[  
\hUuL := \bigcup_d \hUuL_d,
  \qquad \hUuR := \bigcup_d \hUuR_d, \qquad \hUu := \bigoplus_d \hUu_d.
 \]
As previously mentioned, it follows from \cref{good completions} that this will result in a good triple $(\hUuL, \hUu, \hUuR)$ of associative algebras with good seminorms.

Finally, one may construct a graded ideal $J$ of $\hUu$ generated by the Jacobi relations, namely $J$ is generated by for $\ell$, $m$, $n \in \mathbb{Z}$, and $a$, $b\in V$, by
\begin{equation*}
    \sum_{i\ge 0}(-1)^i{\binom{\ell}{i}}\left(a_{[m+\ell-i]}b_{[n+i]}-(-1)^{\ell}b_{[n+\ell-i]}a_{[m+i]}\right)
=\sum_{i\ge 0}{\binom{m}{i}}(a_{(\ell+i)}(b))_{[m+n-i]}.
\end{equation*}

If we let $J^{\mathsf{R}}$ and $J^{\mathsf{L}}$ be the ideals of $\hUuR$ and $\hUuL$ generated by $J$, and we let $\ov J, \ov J^{\mathsf{R}}, \ov J^{\mathsf{L}}$ be the respective closures (see \cref{closure split}), then we find that $(\ov J^{\mathsf{L}}, \ov J, \ov J^{\mathsf{R}})$ form a good triple (of nonunital algebras) 
by \cref{good ideals from homogeneous} and \cref{good ideal closures}. Finally, by \cref{good quotients}, 
we find that the resulting quotient algebras 
\[
\UVL = \hUuL/\ov J^{\mathsf{L}}, \quad
\UV = \hUu/\ov J, \quad
\UVR = \hUuR / \ov J^{\mathsf{R}},
\]
form a good triple of associative algebras with good seminorms (actually norms).
\begin{defn} \label{enveloping algebra defs}
We call $\UVL, \UVR, \UV$ the left, right and finite universal enveloping algebras of $V$, respectively.	
\end{defn}

We note that a $V$-module corresponds, in this language, to a $\UV$-module $W$ (or a $\UVL$-module) such that the action of this normed (and hence topological) algebra is continuous. That is, such that the multiplication map
\[\UV \times W \to W \quad \text{ or equivalently, } \quad \UVL \times W \to W\]
is continuous, where $W$ is given the discrete topology, and $\UV$ and $\UVL$ are topologized according to their norms.

\subsubsection{Relation to the literature} We note that $\UV$ coincides with the {universal enveloping algebra} of $V$ introduced in \cite{FrenkelZhu}, while we can identify $\UVL$ with the {current algebra} introduced in \cite{NT} or with the {universal enveloping algebra} $\widetilde{U}(V)$ introduced in \cite{bzf}, with a minor modification (see \cite[footnote on p.74]{FrenkelLanglands}).

\subsection{Subalgebras and subquotient algebras}

We describe here some algebras built from the algebras $\UV, \UVL, \UVR$ which will play a special role. By definition, $\UVL_{\leq -n} \triangleleft \UVL_{\leq 0}$ and $\UVR_{\geq n} \triangleleft \UVR_{\geq 0}$ are two-sided ideals when $n > 0$ with
\[\UVL_{\leq 0}/\UVL_{\leq -1} \cong \UV_0 \cong \UVR_{\geq 0}/\UVR_{\geq 1},\] 
by the fact that these algebras are part of a good triple. We now look more closely at $\UV_0$, which forms a subring of $\UV$. As our triple of algebras is good, the seminorms on our algebras are almost canonical (\cref{alm can def}), and in particular by \cref{alm can def}\eqref{alm can 3}, it follows that $\NL{n}\UV_0=\NR{n}\UV_0$ for every $n$, so that there is no ambiguity in denoting these neighborhoods by $\Nn{r}\UV_0$. We also see in the same way that $\Nn r \UV_0$ is a two-sided ideal of $\UV_0$. 

\begin{defn}\label{def:GZA}The \textit{$d$th higher level Zhu algebra} of $V$ is  the quotient 
\[\Aa_d(V)=\UV_0/\Nn{d+1}\UV_0.\] For an element $\alpha \in \UV_0$, we will write $[\alpha]_d$ for its image in $\Aa_d(V)$.  When $V$ is understood, we will denote $\Aa_d(V)$ simply by $\Aa_d$.
\end{defn}

\subsubsection{Relation to the literature} In \cite{ZhuMod} the author defines an associative algebra, now referred to as the (zeroth) Zhu algebra as the quotient of $V$ by an appropriate subspace $O(V)$. In \cite{FrenkelZhu,NT} it is shown that this algebra is isomorphic to an appropriate quotient of the degree zero piece of the universal enveloping algebra of $V$ (or of the current algebra of $V$). As mentioned in the introduction, higher level Zhu algebras $\Aa_d$ have been introduced in \cite{DongLiMason:HigherZhu} as quotients of $V$ by subspaces $O_d(V)$, and proved to be realized as quotients of the degree zero piece of the universal enveloping algebra of $V$ in \cite{He}. We notice further that the map that realizes the isomorphism between $V/O_d(V)$ and $\Aa_d$ is explicitly realized by identifying $[a] \in V/O_d(V)$ with the class of the element $a_{\deg(a)-1}$ in $\Uu_0/\Nn{d+1}\Uu_0$ for every homogeneous element $a \in V$.

\section{Induced modules
 and the mode transition algebra}\label{sec:ZhuFunctor}

The constructions and results discussed here are true in greater generality, as detailed in \cref{sec:GHZ} and in \cref{sec:GMTA}. For instance, as  in \cref{sec:UnivEnv}, the constructions mentioned in \cref{sec:PhiL} and in \cref{sec:ModeAlgebra} hold for  a graded vertex algebra. The conformal structure is however used in \cref{sec:BimodFun}.

\subsection{Induced modules}\label{sec:PhiL}
As is the convention, throughout we denote $\Aa_0$  by $\Aa$.

\begin{defn}\label{def:PhiL}
For a left module $W_0$ over $\Aa$, we define the \textit{left generalized Verma module} $\PhiL(W_0)$ as
\begin{equation}\label{eq:PhiL}\PhiL(W_0) = \bigoplus_{p = 0}^{\infty} \PhiL(W_0)_p = \left(\UV/\NL 1\UV\right) \otimes_{\UV_{0}} W_0 \cong \left(\UVL/\NL 1\UVL\right) \otimes_{\UV_{0}} W_0. \end{equation}
For $Z_0$ a right module over $\Aa$, we define the \textit{right generalized Verma module} $\PhiL(W_0)$ as
\begin{equation}\label{eq:PhiR}\PhiR(Z_0) =  \bigoplus_{p = 0}^{\infty} \PhiL(Z_0)_{-p} = Z_0 \otimes_{\UV_{0}} \left(\UV/\NR 1\UV\right) \cong  Z_0 \otimes_{\UV_{0}} \left(\UVR/\NR 1\UVR\right). 
\end{equation}
\end{defn}

We note that this is well defined. In fact, the claimed isomorphisms of \eqref{eq:PhiL} and \eqref{eq:PhiR} follow from \cref{discrete lemma}, while the grading is explained in \cref{induced grading}.

Moreover, from \cref{discrete lemma} we have that $\UV/\NL 1\UV \cong \UVL/\NL 1\UVL$, so that this quotient can be regarded both as a $(\UV,\UV_{0})$ bimodule and as a $(\UVL ,\UV_{0})$ bimodule. In particular, this shows that the ancillary algebra acts on the left on $\PhiL(W_0)$.

The functors $\PhiL_d$ have a universal property, described in \cref{prop:UnivPaper}, and proved more generally in \cref{universal property}.  Namely, given a continuous left $\UV$-module $W$, we define an $\Aa_d$-module $\Omega_d(W)$ by
\[\Omega_d(W) = {\mathrm{rAnn}}_{W}(\NL {d+1} U)= \{w \in W \mid (\NL {d+1} U)w = 0\},\]
where by  \cref{lem:Ann}, one has that the right annihilator ${\mathrm{rAnn}}_{W}(\NL {d+1} U)$  coincides with 
${\mathrm{rAnn}}_{W}(\LVL_{\le -d-1})$. It therefore follows that the definition of $\Omega_d(W)$ agrees with
\[\Omega_d(W) =  \{w \in W \mid \LVL_{-d-1} w = 0\},\]
as originally given in \cite{DongLiMason:HigherZhu}.

\begin{prop}\label{prop:UnivPaper}
Let $M$ be a $\UV$-module and $W_0$ an $\Aa_d$-module. Then there is a natural isomorphism of bifunctors:
\[\Hom_{\Aa_d}(W_0, \Omega_d(M)) = \Hom_{\UV}(\PhiL_d(W_0), M).\]
\end{prop}

\begin{remark}\label{sec:GeneratingSets}
 If $W_0$ if finite dimensional over $\mathbb{C}$, and if $V$ is $C_1$-cofinite,
then there are a finite number of elements $x^1$, $x^2$, $\ldots$, $x^r \in V$
such that $\PhiL(W_0)$ is  spanned by elements of the form
\[
x^1_{(-m_1)}\cdot x^2_{(-m_2)}\cdots x^r_{(-m_r)} \otimes u,\]
for some $u \in W_0$ and positive integers $m_1 \ge m_2 \ge \cdots \ge m_r \ge 1$.
\end{remark}

\subsubsection{Relation to the literature}\label{sec:Quotient}
By the universal property in \cref{prop:UnivPaper}, the functor $\PhiL$ is naturally isomorphic  the functor denoted $\overline{M}_0$ in \cite[page 258]{ZhuMod}, \cite[(4.4)]{DongLiMason:HigherZhu}, and \cite[page 3301]{BVY}.  Moreover, one can relate  $\PhiL$ to $Q_n(d)$ from  \cite{MNT},  which in the language used here, can be written $Q_n(d)={\ms{U}_d}/{\NL{n} \ms{U}_d}$. In particular for $n=1$ this gives the degree $d$ part of a generalized Verma module 
    \[\PhiL(W_0)_d={\ms{U}_d}/{\NL{1} \ms{U}_d}  \otimes_{\ms{U}_0} W_0 = Q_1(d)\otimes_{\ms{U}_0} W_0.\]
    In \cite[Eq 2.6.1]{MNT} a series of sub-algebras (called Quasi-finite algebras) of the universal universal enveloping algebra is defined for all $d\in \NN$, and by \cite[Thm 3.3.4]{MNT} if $V$ is $C_2$-cofinite, then for $d\gg0$ there is an equivalence of categories of $A_d$-modules and (admissible) 
    $V$-modules. 

    \subsection{Mode transition algebras and their action on modules}\label{sec:ModeAlgebra} 

In this section we introduce the \textit{mode transition algebra} $\Ac(V)$ associated with a vertex operator algebra $V$. A general treatment of these algebras is developed in \cref{sec:Generalized}, while we will list here the principal consequences of the general theory. We begin by introducing the space underlying $\Ac(V)$, often denoted $\Ac$ when $V$ is understood.

\begin{defn} \label{def:phi}
Let $V$ be a VOA and $\Aa=\Aa_0$ be the Zhu algebra associated to $V$. We define $\Ac=\Ac(V)$ to be the vector space
\begin{align*}
\Ac &= \PhiR(\PhiL(\Aa)) = \PhiL(\PhiR(\Aa))= \left(\UV/\NL 1 \UV\right) \otimes_{\UV_{ 0}} \Aa \otimes_{\UV_{ 0}} \left(\UV/\NR 1 \UV\right).	
\end{align*}
Moreover, using the notation $\Ac_{d_1, d_2} = \left(\UV/\NL 1 \UV\right)_{d_1} \otimes_{\UV_{ 0}} \Aa \otimes_{\UV_{ 0}} \left(\UV/\NR 1 \UV\right)_{d_2}$  we write
\[ \Ac = \bigoplus_{d_1 \in \mathbb Z_{\geq 0}} \ \bigoplus_{d_2 \in \mathbb Z_{\leq 0}}\Ac_{d_1, d_2}.
\]  \end{defn}

The isomorphism described in the following Lemma is crucial to the definition of an algebra structure on $\Ac$. We refer to \cref{innerstar lemma} for its proof.

\begin{lem}\label{lem:innerstar}There is an isomorphism:
\begin{align*} \left(\UV/\NR 1 \UV\right) \otimes_{\UV}
\left(\UV/\NL 1 \UV\right) &\to
\Aa \\
\ov{\alpha} \otimes \ov{\beta} &\mapsto \alpha \innerstar \beta
\end{align*}
where, for $\alpha, \beta \in \UV$ homogeneous, we define $\alpha \innerstar \beta$ as :
\[\alpha \innerstar \beta = \begin{cases}
0 & \text{if } \deg(\alpha) + \deg(\beta) \neq 0 \\
[\alpha \beta]_0 & \text{if } \deg(\alpha) + \deg(\beta) = 0
 \end{cases}
\]
and we extend the definition to general products by linearity.
\end{lem}

\begin{ex} \label{rmk:explostar} 
We explicitly describe the element $\alpha \ostar \beta \in \Aa$ when $\alpha$ and $\beta$ are homogeneous elements of opposite degrees. Three cases can occur: 
\begin{itemize}
    \item $\deg(\alpha)<0$. It follows that $\deg(\beta) > 0$ and so $\alpha\beta \in \Nn{1}\UV_0$, which gives $\alpha \ostar \beta =0$.
    \item $\deg(\alpha)=0=\deg(\beta)$. We have that $\alpha \ostar \beta = [\alpha]_0 \cdot [\beta]_0 $ since the map $\UV_0 \to \Aa_0$ is a ring homomorphism.
    \item $\deg(\alpha) > 0$. We first rewrite $\alpha \beta = \beta\alpha + [\alpha,\beta]$. Since $\beta\alpha \in \Nn{1}\UV_0$, we have that $\alpha \ostar \beta$ coincide with $[\, [\alpha,\beta] \,]_0$.
\end{itemize}
Note that if $\alpha, \beta \in \LVf$, then $[\alpha,\beta]\in \LVf_0$, so the above description tells us that $\alpha\ostar \beta$ is computed via the standard map $\LVf_0 \to \Aa$ described in \cite{lithesis}.
\end{ex}

\begin{remark}\label{rmkWhyEqualA}
  We note that one has that  $\Ac_{0,0}=\Aa$.  Indeed, by the definitions
  \begin{equation*}
      \Ac_{0,0} = \left({\UV}/{\NL{1} \UV}\right)_0 \otimes_{{\UV}_0}
      \left({\UV_0}/{\Nn{1} \UV_0}\right) \otimes_{\UV_0}
      \left({\UV}/{\NR{1} \UV}\right)_0 \cong 
       \Aa \otimes_{\Aa} \Aa \otimes_{\Aa} \Aa \cong \Aa.
  \end{equation*}
\end{remark}

We will next simultaneously describe the algebra structure on $\Ac$, and the action of this algebra $\Ac$ on generalized Verma modules.

\begin{defn} \label{def:abstract module associativityPaper}
Let $W_0$ be an $\Aa$-module. Following \cref{star product} we  define the map $\Ac \times \PhiL(W_0) \to \PhiL(W_0)$ as follows. For $\xx = u \otimes a \otimes u' \in \Ac$ and $\beta \otimes w \in \PhiL(W_0)$ we set
\[\xx \star (\beta \otimes w) := u \otimes a (u' \innerstar \beta) w.\] \end{defn}

By \cref{def:abstract module associativityPaper} this map defines an algebra structure on $\Ac$ and $\PhiL(W_0)$ becomes a left $\Ac$-module. Moreover the subspace $\Ac_d:= \Ac_{d,-d}$ is closed under multiplication, hence it defines a subalgebra of $\Ac$. The following is a special case of \cref{def:GoodTripleMTA notriple}. 

\begin{defn} \label{mode transition def} We call $\Ac(V)=\Ac=(\Ac,+, \star)$ the \textit{mode transition algebra} of $V$, and $\Ac_d=(\Ac_d, +, \star)$ the \textit{$d$-th mode transition algebra} of $V$.  
\end{defn}

\begin{remark}
More generally, the $d$-th mode transition algebra $\Ac_d$ acts naturally on the degree $d$ part of any weak $\NN$-graded module in a way which is compatible with the action of the enveloping algebra (see \cref{general module action}).
\end{remark}

\begin{remark} \label{rmk:VAenough} We observe that the underlying vector space and the algebra structure of $\Ac(V)$ does not depend on the existence of a conformal structure on $V$. Therefore $\Ac(V)$ can be defined for every graded vertex algebra $V$.
\end{remark}

We refer to \cref{rmk:rational} for an explicit description of the algebra structure of $\Ac(V)$ when $V$ is a $C_2$-cofinite and rational vertex operator algebra. The following assertion is straightforward

\begin{remark}\label{lem: action lemma} Let $W_0$ be an $\Aa_0$-module. Then the action of $\Ac_d$ on $\PhiL(W_0)_d$ factors through the action of $\Aa_d$ described in \cref{def:abstract module associativityPaper} via the map $\mu_d$. 
\end{remark}

\subsubsection{Relation to the literature} In \cite{DongAlgebras} a series of unital associative algebras $\Aa_{e,d}$, defined as quotients of $V$, with $\Aa_{d,d}\cong \Aa_d$. By \cref{def:abstract module associativityPaper}, the $\Ac_d$ act on the degree $d$ part of an induced module $W=\PhiL(W_0)$, as is true for the $\Aa_{e,d}$, although they differ. 
In \cite{HuangAlg1}, a related series of associative algebras $A^d(V)$ is defined. These contain higher level Zhu algebras as sub-algebras, and  act on (the sum of) components of a module up to degree $d$.  In \cite{HuangAlg2},  relations are established between bimodules for these 
associative algebras and (logarithmic) intertwining operators, and using these, in \cite{HuangAlg3}, modular invariance of (logarithmic) 
intertwining operators is proved.

\subsection{Strong unital action of \texorpdfstring{$\Ac_d$}{Ad} on modules}\label{sec:AdModuleAction}
The $\Ac_0=\Aa$  has an identity element given by the image of ${\bm 1}_{[-1]}$, denoted $1$. 
On the other hand, as we show in \cref{eg:virasoro}, for $d \in \mathbb{Z}_{>0}$,  the $\Ac_d$ may not admit multiplicative identity elements. However, if there are unities in $\Ac_d$ for all $d \in \NN$, we have the following results about them.

\begin{deflem} \label{lem:StrongUnitConditions} \label{def:StrongUnit}
Let $M$ be an $\Aa$-module, and assume that for every $d \in \NN$ the ring $\Ac_d$ is unital, with unity $\Is_d \in \Ac_d$. We say that $\Is_d$ is a strong identity element for every $d$ if one of the following equivalent conditions is verified: \begin{enumerate}[label={(\textit{\arabic*})}]
    \item \label{it:H1} For every  $d$,$n$,$m \in \NN$, for all $u \in \LV_d$, and   $\xx \in \Ac_{n,-m}$ one has \[(u \cdot \Is_n) \star \xx = u \cdot \xx \qquad \text{ and } \qquad \xx \star (\Is_m \cdot u) = \xx \cdot u.\] 
    \item \label{it:H2} For every $n,m \in \NN$ and for every $\xx \in \Ac_{n,-m}$ one has $\Is_n \star \xx =\xx= \xx \star \Is_m$.
    \item \label{it:H4} For every $d \in \NN$ the homomorphism $\Ac_{d} \to \End(\PhiL(\Aa)_d)$ is unital;
    \item \label{it:H5} For every $d \in \NN$, the homomorphism $\Ac_{d} \to \End(\PhiL(\Aa)_d)$ is unital and injective.
    \item \label{it:H3} For every $d \in \NN$ and $M$ an $\Aa$-module, the homomorphism $\Ac_{d} \to \End(\PhiL(M)_d)$ is unital.
\end{enumerate}
\end{deflem}

\begin{proof} We prove these conditions are equivalent. Since \ref{it:H5} implies \ref{it:H4} and \ref{it:H3} implies \ref{it:H4}, it will be enough to show the following implications: 
\[ \text{\ref{it:H1} $\Leftrightarrow$ \ref{it:H2} $\Leftrightarrow$  \ref{it:H4} $\Rightarrow$ \ref{it:H3} 
 \quad and \quad \ref{it:H4} 
 $\Rightarrow$ \ref{it:H5}}  \]

\ref{it:H1} $\Rightarrow$ \ref{it:H2}:  this follows by taking $d=0$ and $u = 1$.

\ref{it:H2} $\Rightarrow$ \ref{it:H1}: This follows from \cref{lem:abstract-module-associativity}. 

\ref{it:H2} $\Rightarrow$ \ref{it:H4}: This follows from the identification of $\PhiL(\Aa)_d$ with $\Ac_{d,0}$.

\ref{it:H4} $\Rightarrow$ \ref{it:H2}: By linearity, we can assume that $\xx \in \Aa_{n,-m}$ is represented by an element of the form $u \otimes a \otimes v$ with $u \in \UVL_{n}$,  $v \in \UVR_{-m}$ and $a \in \Aa$. Then 
\[ \Is_n \star ( u \otimes a \otimes v) = \Is_n \star (u \otimes \xx \otimes 1) \cdot (v)  =  (\Is_n \star (u \otimes a)) \otimes v = (u \otimes a \otimes 1) \cdot v = u \otimes a\otimes v
\] where \ref{it:H4} ensures that the third equality holds. 

\ref{it:H4} $\Rightarrow$ \ref{it:H3}: By linearity we can assume that an element of $\PhiL(M)_d$ is given by $u \otimes m$ for $u \in \UL_d$ and $m \in M$. Hence we obtain $\Is_d \star(u\otimes m)=(\Is_d \star u) \otimes m= u \otimes m$.

\ref{it:H4} $\Rightarrow$ \ref{it:H5}: By \cref{def:abstract module associativityPaper} the action of an element $\xx \in \Ac_d$ via $\star$ on $\Ac_d \subseteq \Ac= \PhiL(\PhiR(\Aa)) = \PhiL(\Aa) \otimes_{\UV_{ 0}} \left(\UV/\NR 1 \UV\right)$ is determined by the action of $\xx$ on $\PhiL(\Aa)$. In particular, as the former is injective when we have an identity element, the latter must be injective in this case as well.
\end{proof}

\subsection{Relation to the functor \texorpdfstring{$\Phi$}{Phi}}  \label{sec:BimodFun}
To define the mode transition algebra $\Ac$, we used the map  $\PhiL\PhiR=\PhiL\PhiR$, which we can interpret as a functor from the category of $\Aa$-bimodules to the category of $\UV$-bimodules. Its properties in a more general framework are described in \cref{lem:abstract-module-associativity}. 

We now show that this functor agrees with the functor denoted $\Phi$ in \cite[Definition 2.2]{DGK} which assigns to an $\Aa$-bimodule $M$ the $(\UVL)^{\otimes 2}$-module
\[ \Phi(M):=(\UVL \otimes \UVL) \underset{\UVL_{\leq 0} \otimes \UVL_{\leq 0}}{\fctensor} M,
\] where $\UVL_{\leq 0} \otimes \UVL_{\leq 0}$ acts on $M$ as follows:
\[ (u \otimes v) (m) = \begin{cases} u \cdot m \cdot \theta(v) & \text{ if } u,v \in \UV_{0}\\
0 & \text{ otherwise}.
\end{cases} 
\]
Here $\theta$ is the natural involution of $\UV_0$, from e.g. \cite[Eq.(7)]{DGK}), and which we describe briefly in \cref{sec:theta}. In \cref{lem:PhiLRPhi} we describe the relation between functors $\PhiL$ and $\PhiR$  to $\Phi$.

\subsubsection{The involutions, left and right actions}\label{sec:theta}
As we explain here, there is an anti-Lie algebra isomorphism $\theta$ used to transport the universal enveloping algebra, considered as an object that acts on modules on the left (denoted here by $\UVL$), to an analogous completion $\UVR$ that acts on modules on the right.

The map
$\theta \colon  V \otimes_\CC \Cl{t} \to V \otimes_\CC \Cl{t^{-1}}$ is defined, for  $a \in V$ homogeneous, by
\begin{equation}\label{eq:LieTheta} a \otimes \sum_{i \geq N} c_i t^{i} \mapsto (-1)^{\deg(a)} \sum_{j \geq 0}^{\deg(a)} \left( \dfrac{1}{j!} \left( L_1^j(a)\right) \otimes \sum_{i \geq N} c_i t^{2\deg(a)-i-j-2}\right), 
\end{equation}  and extended linearly.

The map $\theta$ is related to the involution $\gamma = (-1)^{L_0}e^{L_1}  \colon V \rightarrow V$ defined, for $a \in V$ homogeneous by
    \begin{equation} \label{eq:Gamma} a \mapsto (-1)^{\deg a} \sum_{i \geq 0} \frac 1 {i!} L_1^i(a),
\end{equation} and extended by linearity. To state the relation succinctly, for every homogeneous $a\in V$ we set
\begin{equation}\label{eq:J}J_n(a) := a_{[\deg(a) - 1 + n]}.
\end{equation}
 This notation, used in \cite{NT}, has the property that $\deg(J_n(v))=-n$, so that the degree of such an element is easily read. 
 
\begin{lem}\label{lem:ThetaGammId}For $a\in V$, homogeneous,
$\theta(J_n(a)) = J_{-n}(\gamma(a))$. 
    \end{lem} 
    \begin{proof}
        This follows by combining \eqref{eq:LieTheta} and \eqref{eq:Gamma} and using linearity.
    \end{proof}

\begin{lem} \label{lem:antiLietheta} The map $\theta$ induces a Lie algebra anti-isomorphism $\LVL \to \LVR$, which restricts to a Lie algebra involution on $\LVf$, such that $\theta(\LVL_{\leq d})=\LVR_{\geq -d}$ and $\theta(\LVf_{d}) = \LVf_{-d}$.
\end{lem}
\begin{proof}
One can check that this restricts to an endomorphism of $V \otimes \CC[t,t^{-1}]$, which by \cite[Proposition 4.1.1]{NT}, defines a Lie algebra involution of $\LVf$, that is, $\theta([\ell_1,\ell_2]) = -[\theta(\ell_1), \theta(\ell_2)]$, and $\theta^2 = \id$. Moreover, it is easy to verify that $\theta(\LVf_d) \subset \LVf_{-d}$. As the Lie algebras $\LVL, \LVR$ carry exhaustive and separated split filtrations by the graded subalgebra $\LVf$, they are naturally equipped with norms via \cref{filtered is seminormed} -- that is, by declaring that elements of large positive degree are large in $\LVL$ and small in $\LVR$. With this definition, it follows that $\theta$ is continuous, and as noted in \cref{filtered continuous mult}, that the multiplication on the Lie algebras is continuous. Finally, the fact that $\LVf$ induces a splitting of the filtrations, it follows that $\LVf$ is simultaneously dense in $\LVL$ and $\LVR$. Consequently, we see by continuity that $\theta$ induces an anti-homomorphism from $\LVL$ to $\LVR$, which is an anti-isomorphism as $\theta^2 = \id$.
\end{proof}

 If $R=(R,+, \cdot)$ is a ring, denote by $R^{\mathrm{op}}$ its opposite ring, that is $R^{\mathrm{op}}=(R, +, *)$ where $a*b:= b \cdot a$. Similarly, if $(L, [ \, ,\,])$ is a Lie algebra, we denote by $L^{\mathrm{op}}$ its opposite Lie algebra, where $[a,b]_{L^{\mathrm{op}}} := [b,a]_{L}$.

\begin{lem} \label{lem:U+U-op}For $\Uu(\LVR)$ the universal enveloping algebra of $\LVR$, \[\theta \colon \Uu(\LVL) \to \Uu(\LVR)^{\mathrm{op}}\]
is an isomorphism of rings. 
\end{lem}
\proof We have established in \cref{lem:antiLietheta} that $\theta \colon \LVL \to \LVR$ is an anti-isomorphism of Lie algebras, so that $\theta \colon \LVL \to (\LVR)^{\mathrm{op}}$ is a Lie-algebra isomorphism. Moreover, as $\Uu((\LVR)^{\mathrm{op}}) = \Uu(\LVR)^{\mathrm{op}}$, it follows that $\theta$ induces an isomorphism between $\Uu(\LVL)$ and $\Uu(\LVR)^{\mathrm{op}}$, as wanted.
Here we note that $\theta(\alpha \cdot \beta)= \theta(\beta) \cdot \theta(\alpha)$ for every $\alpha, \beta \in \LVL$ (and where $\cdot$ is the usual product in the $\Uu(\LVL)$ and $\Uu(\LVR)$).
\endproof

In particular $\theta$ is an isomorphism between $\Uu(\LVf)$ and $\Uu(\LVf)^{\mathrm{op}}$.

\begin{lem} \label{lem:PhiLRPhi} Let $B$ be an associative ring,  $W^1$ an $(\Aa,B)$-bimodule and $W^2$ a $(B,\Aa)$-bimodule. Then we have a natural identification
\[\PhiL(W^1)\otimes_{B} \PhiR(W^2) \cong \Phi(W^1\otimes_B W^2)\] of $(\UVL,\UVR)$-bimodules. In particular we have  $\Ac=\Phi(\Aa)$. \end{lem}

\begin{proof} We first note that there is a natural equivalence of categories between left $(\UVL)^{\otimes 2}$-modules and $(\UVL,(\UVL)^\text{op})$-modules. Moreover, as described in \cref{lem:U+U-op}, the involution $\theta$ provides an identification $\UVR \cong (\UVL)^\text{op}$. It follows that the map $\PhiL(W^1)\otimes_{B} \PhiR(W^2) \to \Phi(W^1\otimes_B W^2)$ induced by
\[ (u \otimes w_1) \otimes (w_2 \otimes v) \mapsto (u \otimes \theta(v)) \otimes (w_1 \otimes w_2)
\] for all $u \in \UVL$, $v \in \UVR$ and $w_i \in W_i$ is indeed an isomorphism.\end{proof}

\begin{remark}\label{rmk:rational} Let $V$ be a rational VOA.  Then the mode transition algebras $\Ac_d$ have strong identity elements. To see this, we note that the rationality of $V$ implies that $\Aa$ is finite and semi-simple, and thus has a bimodule decomposition \[\Aa\cong \prod_{k=1}^{m} I^k = \prod_{k=1}^{m} W_0^k\otimes (W_0^k)^\vee,\] where $\{1, \dots , m\}$ is the set indexing the isomorphism classes of simple $V$-modules. Equivalently, $W_0^k$ runs over all isomorphism classes of simple left $\Aa$-modules (so that $W_0^\vee$ is the corresponding dual right $\Aa$-module). By \cref{lem:PhiLRPhi} we therefore have:
\begin{align*}
    \Ac = \Phi(\Aa)= \Phi \left(\prod_{k=1}^m W_0^k \otimes (W_0^k)^\vee \right) = \prod_{k=1}^m \PhiL(W_0^k) \otimes_{\CC} \PhiR(W_0^k)^\vee.
\end{align*} 
Note that since $V$ is rational, the functor $\PhiL$ takes simple $\Aa$-modules to simple $V$-modules. This follows from the fact that $\PhiL$ preserves indecomposable modules \cite[Lemma 1.2]{DGK}, and that for a rational VOA, indecomposable modules are simple. Hence for $W_0$ a simple $\Aa$-module, $W := \PhiL(W_0)$ must be simple as well. Similarly, we have that also $\PhiR(W_0^\vee)$ must be a simple $V$-module. 
We give some details about this:  the involution $\theta$ identifies left and right $\Aa$-modules allowing us to consider $(W^k_0)^\vee$ as a left $\Aa$-module which we denote as $\mbox{}^\theta(W^k_0)^\vee$. From this perspective, and the above discussions relating $\PhiL$, $\PhiR$ and $\theta$, we thus have a natural identification 
\[ \PhiL(\mbox{}^\theta(W^k_0)^\vee)= \PhiR((W^k_0)^\vee) \qquad \text{and} \qquad \PhiL(\mbox{}^\theta(W^k_0)^\vee)_d= \PhiR((W^k_0)^\vee)_{-d},\] for every $d \in \NN$. Since $W^k_0$ was simple, also $\mbox{}^\theta(W^k_0)^\vee$ is simple, and thus we conclude that $\PhiR((W^k_0)^\vee)$ is simple as well. 

We now show that $\PhiR((W_0^k)^\vee)$ is isomorphic to the contragredient module $(W^k)'$ of $W^k$. The contragredient module $(W^k)'$ has zero part $(W^k)'_0 = \mbox{}^\theta(W^k_0)^\vee$ inducing a natural map $\PhiR((W^k_0)^\vee)=\PhiL(\mbox{}^\theta(W^k_0)^\vee)\to (W^k)'$. 
As the contragredient of a simple $V$-module is simple
, it follows that this map is an isomorphism. In particular one has $\PhiR((W^k_0)^\vee)_{-d}=(W^k_{d})^\vee$ and thus
\begin{equation} \label{eq:Acrtl}
    \Ac =\bigoplus_{d\in \ZZ_{\ge 0}}\prod_{k=1}^m W^k_d \otimes_{\CC}(W^k_{d})^\vee.
\end{equation} 
Using this decomposition, the $\star$-product is induced, by linearity, from 
\[ (a_{W^k} \otimes \varphi_{W^k_d})  \star (b_{W^j_e} \otimes \psi_{W^j}) = \begin{cases} \varphi_{W^k_d}(b_{W^j_e})( a_{W^k} \otimes \psi_{W^j}) &\text{if } k=j \text{ and } e=d \\ 0 &\text{otherwise},
\end{cases}\]
where $\varphi_{W^k_d} \colon W^k_d \to \CC$ and $b_{W_e^j} \in W_e^j$.
Note that, by rationality, the spaces $W^k_d$ are necessarily finite dimensional \cite[Definition 1.2.4]{ZhuMod}. It then follows that,  for all $d\in \mathbb{Z}_{\ge 0}$, we have an isomorphism of rings \[\Ac_d\cong \prod_{i=1}^m \text{End}_\CC(W^k_d)\] and thus ${\textbf{1}}_d:=\prod_{i=1}^m \text{Id}_{W^k_d}$ is its strong identity element. We remark that \cite[Proposition 7.2.1]{NT} states, without a detailed proof and using a different notation, that \eqref{eq:Acrtl} holds, under the assumption that $V$ is not only rational, but also $C_2$-cofinite. From our argument it is apparent how $C_2$-cofinitess is not needed.
\end{remark}

\section{Smoothing, limits, and coinvariants}\label{sec:LimAndCoinv}

In \cref{sec:CoinvNotation} we describe the sheaf of coinvariants on schemes $S$ parametrizing families of pointed and coordinatized curves in general terms, while in \cref{sec:CompleteCoinvariants}, we explain what we mean by sheaves defined over a scheme $S = \Spec R$, where $R$ is a ring complete with respect to some ideal $I$. In \cref{sec:settingSmoothing}, we describe the setup for considering coinvariants on smoothings of nodal curves, establishing some results needed for our geometric applications.  In particular, 
in \cref{sec:Step2}, for the proof of \cref{prop:smoothing}, we explicitly describe the sheaf $\Lc_{\Cs\setminus P_\bullet}(V)$ of Chiral Lie algebras.

Throughout this section,   $V$ is a VOA with no additional finiteness assumptions. 

\subsection{Coinvariants}\label{sec:CoinvNotation}
Let $S$ be a scheme and let $\Ws$ be a quasi-coherent sheaf of $\Os_S$-modules. Let $\Ls$ be a quasi-coherent sheaf of Lie algebras on $S$ acting on $\Ws$. We define the \textit{sheaf of coinvariants} $[\Ws]_{\Ls}$ on S as the cokernel 
\[{\Ls} \otimes_{\Os_{S}} {\Ws} \to {\Ws} \to [{\Ws}]_{{\Ls}} \to 0.\]

For future use, it will be helpful to note that the formation of the sheaves of coinvariants commutes with base change.
\begin{lem} \label{coinvariant pullback}
Let $\Ls$ be a quasi-coherent sheaf of Lie algebras on a scheme $S$ acting on a quasi-coherent sheaf $\Ws$. For any morphism $S' \to S$, we have $\left([{\Ws}]_{{\Ls}}\right)_{S'} \cong [{\Ws}_{S'}]_{{\Ls}_{S'}}$.
\end{lem}
\begin{proof}
This follows from right exactness of pullback of quasi-coherent sheaves (equivalently right exactness of tensor).
\end{proof}

\begin{remark}\label{family coinvariants}
Let $\pi \colon  \Cs \to S$ be a projective curve, with $n$ distinct smooth sections $P_\bullet \colon  S \to  \Cs$ and formal coordinates $t_\bullet$ at $P_\bullet$. Assume further that $\Cs \setminus \sqcup P_\bullet(S) \to S$ is affine. This assumption is possible by Propagation of Vacua \cite[Thm 3.6]{codogni} (see also \cite[Theorem 4.3.1]{DGT2}). Denote by $W^\bullet=W^1\otimes \dots \otimes W^n$ the tensor product of an $n$-tuple of $V$-modules and let ${\Ws} := W^\bullet \otimes \mathcal{O}_{S}$. The sheaf of Chiral Lie algebras $\Ls:=\Lc_{\Cs \setminus P_\bullet}(V)$, originally defined in this context for families of stable curves with singularities in \cite{ dgt1, DGT2},  is explicitly described with more details here in \cref{sec:Step2}.  The sheaf of coinvariants $[\Ws]_\Ls$ defined above will also be denoted $[W^\bullet]_{(\Cs,P_\bullet, t_\bullet)}$.  While quasi-coherent \cite{dgt1}, for $V$  $C_2$-cofinite, or if generated in degree $1$, this sheaf  is coherent \cite{DGK, DG}. 
\end{remark}

\subsection{Completions}\label{sec:CompleteCoinvariants} As in \cref{sec:CoinvNotation}, we consider coinvariants over $S = \Spec (R)$, where $R$ is a ring that is complete with respect to some ideal $I$. For
$k \in \mathbb{Z}_{\ge 0}$,  setting $S_k = \Spec(R_k)=\Spec( R/I^{k+1})$, pullbacks ${\Ls}_k$ and ${\Ws}_k$ of $\Ls$ and  $\Ws$ to $S_k$ respectively, we work with coinvariants $[{\Ws}_k]_{{\Ls}_k}$ for any $k \in \mathbb{Z}_{\ge 0}$. Due to quasicoherence, each of these can be thought of as a module over $R_k$,  with maps $[{\Ws}_{k+1}]_{{\Ls}_{k+1}} \to [{\Ws}_k]_{{\Ls}_k}$.

\begin{defn}
In the above situation, we define the \textit{formal coinvariants}, denoted $\wh{[{\Ws}]_{{\Ls}}}$ to be the $I$-adically complete $R$-module
\[\wh{[{\Ws}]_{{\Ls}}} = \varprojlim  [{\Ws}_k]_{{\Ls}_k}.\]
\end{defn}

\begin{prop} \label{prop:coker.lim}
We have an identification 
\[\wh{[{\Ws}]_{{\Ls}}} = 
\coker\left[ \varprojlim 
\pi_* {\Ls}_k 
\otimes_{\Os_{S_k}} {\Ws}_k \longrightarrow \varprojlim {\Ws}_k\right]\]
\end{prop}

\begin{prop} \label{prop:completion-of-coinvariants}
Suppose $[{\Ws}]_{{\Ls}}$ is finitely generated over an $I$-adically complete Noetherian ring $R$. Then the natural map $[{\Ws}]_{{\Ls}} \to \wh{[{\Ws}]_{{\Ls}}}$ is an isomorphism.
\end{prop}

\begin{proof}
Consider the exact sequence of $R$-modules (omitting the $\pi_*$ from the notation, and identifying the quasicoherent sheaves with the corresponding $R$-modules):
\[{\Ls} \otimes_{R} {\Ws} \longrightarrow {\Ws} \longrightarrow [{\Ws}]_{{\Ls}} \longrightarrow 0.\]
Tensoring with $R/I^k$ (or geometrically base-changing along $S_k \to S$) is a right exact operation, hence it yields an exact sequence 
\[ {\Ls}_k \otimes_{R_n} {\Ws}_k \longrightarrow {\Ws}_k \longrightarrow \left([{\Ws}]_{{\Ls}}\right)_{R_k} \longrightarrow 0,\]
which shows that we can identify $[{\Ws}_k]_{{\Ls}_k} = \left([{\Ws}]_{{\Ls}}\right)_{R_k}$.
In particular, the composition
\[[{\Ws}]_{{\Ls}} \longrightarrow \wh{[{\Ws}]_{{\Ls}}} \longrightarrow [{\Ws}_k]_{{\Ls}_k}\]
coincides with the surjection $[{\Ws}]_{{\Ls}} \to [{\Ws}]_{{\Ls}} \otimes_R R/I^k$. Since $[{\Ws}]_{{\Ls}}$ is finitely generated over a complete Noetherian ring, it is $I$-adically complete by \cite[\href{https://stacks.math.columbia.edu/tag/00MA}{Tag 00MA(3)}]{stacks-project}. Therefore we can identify $[{\Ws}]_{{\Ls}} = \varprojlim  [{\Ws}]_{{\Ls}} \otimes_R R/I^k = \varprojlim  [{\Ws}_k]_{{\Ls}_k} = \wh{[{\Ws}]_{{\Ls}}}$, giving the desired isomorphism.
\end{proof}

\subsection{Smoothing setup}\label{sec:settingSmoothing} In order to introduce the smoothing property for $V$, we will recall the notion of a smoothing of a nodal curve, and set a small amount of notation used throughout. Let $R = \CC[\![q]\!]$ and write $S = \Spec(R)$. Let $\Cs_0$ be a projective curve over $\CC$ with at least one node $Q$, smooth and distinct points $P_\bullet=(P_1, \ldots, P_n)$ such that $\Cs_0 \setminus P_\bullet$ is affine, and formal coordinates $t_\bullet=(t_1,\ldots, t_n)$ at $P_\bullet$. 
Let $\eta \colon \widetilde{\Cs}_0\rightarrow \Cs_0$ be the partial normalization of $\Cs_0$ at $Q$, which is naturally pointed by $Q_\pm:=\eta^{-1}(Q)$. We also suppose we have chosen formal coordinates at $Q_\pm$ and we call them $s_\pm$.

The choice of our formal coordinates $s_\pm$ determine a smoothing family $(\Cs, P_\bullet, t_\bullet)$ over $S$, with the central fiber given by $(\Cs_0, P_\bullet, t_\bullet)$. Let $(\widetilde{\Cs}, P_\bullet \sqcup Q_\pm, t_\bullet \sqcup s_\pm)$ denote the trivial extension $\widetilde{\Cs}_0 \times S$ with its corresponding markings. We will now discuss the relationship between coinvariants for $(\Cs, P_\bullet, t_\bullet)$ and $(\widetilde{\Cs}, P_\bullet \sqcup Q_\pm, t_\bullet \sqcup s_\pm)$.

 Let $W^1,\dots, W^n$ be an $n$-tuple of $V$-modules, or equivalently, smooth $\UV$-modules for $\UV$ the universal enveloping algebra of $V$ (defined in \cref{sec:UnivEnv}), and $W^\bullet$ their tensor product. As is described above in \cref{family coinvariants}, we may also consider the sheaf of coinvariants $[W^\bullet]_{(\Cs, P_\bullet, t_\bullet)}$.

As mentioned in the introduction, there is a map $\alpha_0 \colon W^\bullet \to W^\bullet \otimes \Phi(\Aa)$ which induces a map between coinvariants
\begin{equation*}
[\alpha_0] \colon  [W^\bullet]_{\left(\Cs_0, P_\bullet, t_\bullet\right)}{\overset{\cong}{\longrightarrow}} [W^\bullet \otimes \Phi(\Aa)]_{\left(\widetilde{\Cs}_0, P_\bullet \sqcup Q_\pm, t_\bullet \sqcup s_\pm\right)}.
\end{equation*}
Moreover, if $V$ is $C_1$-cofinite, then we will show in \cref{nodal isomorphism} that $[\alpha_0]$ is an isomorphism. We recall that  $\Phi(\Aa)=\Ac$, so we will generally use the notation $\Ac$ below.
   
The following result, which is a consequence of \cref{prop:completion-of-coinvariants}, allows us to describe coinvariants over $\widetilde{\Cs}$ whenever they are finite dimensional. The assumptions of the following result are satisfied when $V$ is $C_2$-cofinite, for all $V$-modules $W^\bullet$, and also more generally (by \cite{DGK}).

\begin{corollary}\label{cor:DiagQues}
Assume that the sheaf $[(W^{\bullet}  \otimes \Ac)[\![q]\!]]_{	(\widetilde{\Cs}, P_\bullet \sqcup Q_\pm,t_\bullet \sqcup s_\pm) }$ is coherent over $S$. Then one has identifications  
\begin{multline*}
	[
(W^{\bullet} \otimes \Ac)[\![q]\!]]_{
	\left(\widetilde{\Cs}, 
		P_\bullet \sqcup Q_\pm,
		t_\bullet \sqcup s_\pm\right)
	}\\  \cong 
[W^{\bullet} \otimes\Ac ]_{
	\left(\widetilde{\Cs}_0, 
		P_\bullet \sqcup Q_\pm,
		t_\bullet \sqcup s_\pm\right)
	}[\![q]\!] 
\\	\cong 
[ W^{\bullet}  \otimes \Ac]_{
    \left(\widetilde{\Cs}_0, 
        P_\bullet \sqcup Q_\pm,
        t_\bullet \sqcup s_\pm\right) } \otimes_\CC \CC [\![q]\!].
\end{multline*}
\end{corollary}

\begin{proof} The second isomorphism holds because the coherence assumption implies that  $[ W^{\bullet}  \otimes \Ac]_{	(\widetilde{\Cs}_0, P_\bullet \sqcup Q_\pm,t_\bullet \sqcup s_\pm) }$ is a finite dimensional vector space.  To prove the first isomorphism, we consider, for $R = \mathbb{C}[\![q]\!]$,  the $R$-module and  $R$-Lie algebra
\[\Ws=(W^{\bullet}  \otimes\Ac)[\![q]\!],  \ \mbox{ and } \  \Ls = \Lc_{\widetilde{\Cs}\setminus\{ P_\bullet \sqcup Q_\pm\}}(V).
\] Since $[\Ws]_\Ls$ is a finite dimensional $R$-module, for $R_k = \mathbb{C}[\![q]\!]/q^{k + 1}$,  and $S_k=\mathrm{Spec}(R_k)$, and one can show by \cref{prop:completion-of-coinvariants}, that
 \begin{equation} \label{eq:lim} 
 [\Ws
]_\Ls=\varprojlim \left( \left[\Ws \otimes_{R} R_k \right]_{
\Ls \otimes_R R_k}\right). \end{equation}
 Note further that 
$\Ws \otimes_{R} R_k=\left( W^{\bullet}
\otimes 
\Ac\right) \otimes_{\CC} R_k$, and similarly,  
\[\Ls \otimes_{R} R_k = \Lc_{\widetilde{\Cs_0}\setminus\{ P_\bullet \sqcup Q_\pm\}}(V) \otimes_{\CC} R_k.
\] Using this, together with \cref{prop:coker.lim}, we deduce that \eqref{eq:lim} is isomorphic to
\[ \varprojlim \left(  [W^{\bullet} \otimes \Ac]_{(\widetilde{\Cs_0}, P_\bullet \sqcup Q_\pm, t_\bullet \sqcup s_\pm)} \otimes_{\CC} R_k \right)\] which is indeed $[W^{\bullet} \otimes
\Ac]_{
	(\widetilde{\Cs}_0, 
		P_\bullet \sqcup Q_\pm,
		t_\bullet \sqcup s_\pm)
	}[\![q]\!]$, as was asserted. \end{proof}

\begin{remark}\label{rmk:Fin}
    \cref{cor:DiagQues} implies that, up to some assumptions of coherence, the sheaf of coinvariants associated with $W^\bullet \otimes \Ac$ over $\widetilde{\Cs}_0$ deforms trivially to the sheaf of coinvariants over the trivial deformation $\widetilde{\Cs}$ of $\widetilde{\Cs}_0$. Consequently, the target of the induced map $[\alpha]$, which extends the map $[\alpha_0]$ is therefore identified with the sheaf of coinvariants associated with $\widetilde{\Cs}$ (and not only with a completion thereof).
\end{remark}

We conclude this section with some criteria to show coherence of sheaves coinvariants over $S$. Throughout we will use the notation $R_k = \CC[\![q]\!]/q^{k + 1}$ and $S_k = \Spec(R_k)$ for every
$k\in \NN$. 

\begin{lem} \label{little generators} For $M$ any module over $R_k$, let $m_1, \ldots, m_r \in M$ be elements whose images generate $M \otimes_{R_k} R_0$. Then the elements $m_1, \ldots, m_r$ also generate $M$.
\end{lem}
\begin{proof}
We induct on $k$, the case $k = 0$ being automatic. 
For the induction step, suppose $m \in M$ and consider the $R_{k - 1}$ module $\ov M = M \otimes_{R_k} R_{k - 1}$. By the induction hypothesis, the elements $m_1, \ldots, m_r$ generate $\ov M$. Therefore we can find $a_1, \ldots a_r \in R_k$ so that 
\[m' = m - \sum a_i m_i \in M,\]
maps to $0$ in $\ov M$.

Now consider the submodule $M' = q^kM \subset M$. As $q^kM$ is exactly the kernel of the map $M \to \ov M = M \otimes_{R_k} R_{k-1}$ we find that $m' \in M'$, and therefore we can write $m'  = q^k x$ for some $x \in M$. Writing $x = \sum b_i m_i \pmod{q}$, we find $x - \sum b_i m_i = qy$ for some $y \in M$. But now we have
\[ m = \left(\sum a_i m_i\right) + m' = 
\left(\sum a_i m_i\right) + q^k\left( \left(\sum b_i m_i\right) + qy \right)
= \sum (a_i + q^kb_i) m_i,
\]
as desired.
\end{proof}

\begin{prop}\label{contingent smoothing}
If $[W^{\bullet}]_{(\Cs_0, P_\bullet, t_\bullet)}$ is a finite dimensional vector space, then both
\[[W^{\bullet}[\![q]\!]]_{(\Cs, P_\bullet, t_\bullet)} \quad \text{ and } \quad [(W^{\bullet} \otimes \Ac)[\![q]\!]]_{
\left(\widetilde{\Cs}, 
		P_\bullet \sqcup Q_\pm,
		t_\bullet \sqcup s_\pm \right)
	} \]
are coherent.  
\end{prop}

\begin{proof} For every $k \in \NN$ and for every scheme $X$ over $S$, denote the pullback of $S$ to $S_k$ by $X_k$. Define
\begin{gather*} 
M_k := [W^\bullet_{R_k}]_{\left(\Cs_k, P_\bullet, t_\bullet \right)}  \quad \mbox{ and } \quad 
\til M_k := [(W^\bullet \otimes \Ac)_{R_k}]_{\left(\widetilde{\Cs}_k, P_\bullet \sqcup Q_\pm, t_\bullet \sqcup s_\pm \right)}.
\end{gather*}
Let us first show that $M_k$ and $\til M_k$ are coherent. As we are considering modules over the Noetherian ring $R_k$, we only need to show that they are finitely generated. But by \cref{little generators}, for this it suffices to show that $\til M_0$ and $M_0$ are finitely generated. This holds because by assumption $M_0$ is finitely generated and $\alpha_0 \colon  M_0 \to \til M_0$ is an isomorphism by \cite{DGT2}.

For simplicity, denote 
\begin{gather*}
M = [W^\bullet[\![q]\!]]_{\left(\Cs, P_\bullet, t_\bullet \right)}  \quad \mbox{ and } \quad
\til M = [(W^\bullet \otimes \Ac)[\![q]\!]]_{\left(\widetilde{\Cs}, P_\bullet \sqcup Q_\pm, t_\bullet \sqcup s_\pm \right)}.
\end{gather*}
By \cref{coinvariant pullback}, it follows that $M_k = M \otimes_{R} R_k$ and $\widetilde{M}_k = \widetilde{M} \otimes_R R_k$. Consequently the natural maps $\widetilde{M}_k \to \widetilde{M}_{k - 1}$ and $M_k \to M_{k - 1}$ are surjective. It follows therefore from \cite[Lemma 087W]{stacks-project} that $M$ and $\widetilde{M}$ will be finitely generated over $R$ whenever $M_k$ and $\widetilde{M}_k$ are finitely generated over $R_k$ for every $k$. This is what we have just shown and so $M$ and $\widetilde{M}$ are coherent. \end{proof}

\subsection{The sheaf of Chiral Lie algebras}\label{sec:Step2} The sheaf of Chiral Lie algebras $\Lc_{\Cs\setminus P_\bullet}(V)$ can be identified with a quotient of the space of sections of the sheaf $\mc{V}_{\Cs}\otimes_{\mc{O}_{\Cs}} \omega_{\Cs/S}$ on the affine open set $\Cs \setminus P_\bullet \subset \Cs$. Here, for later use in the proof of \cref{prop:smoothing}, in order to describe the action of $\Lc_{\Cs\setminus P_\bullet}(V)$, we explicitly describe the sheaf $\mc{V}_{\Cs}\otimes_{\mc{O}_{\Cs}} \omega_{\Cs/S}$, where $\mc{V}_{\Cs}$ is the contracted product $\left(V\otimes_{\mathbb{C}}\mathcal{O}_C\right)\times_{\mathcal{A}ut\mathcal{O}}\mathcal{A}ut_{\Cs}$ (see \cref{rem:Correction}).   

For this, suppose we are given a relative curve $\Cs$, projective over $S = \Spec \CC[\![q]\!]$, with closed fiber $\Cs_0$ (cut out by the ideal generated by $q$), and an $(n+1)$-tuple of distinct closed points $P_0, \ldots, P_n \in \Cs_0$ with affine complement $\Cs_0\setminus P_\bullet=\Cs_0 \setminus \bigcup_i P_i$. Let $B = \mc O_{\Cs}(\Cs_0\setminus P_\bullet)$ denote those rational functions on $\Cs$ which are regular at every scheme-theoretic point of $\Cs_0\setminus P_\bullet$ and let $\wh B$ denote its $q$-adic completion.  By \cite[Theorem~3.4]{Pries}, coherent sheaves on $\Cs$ may be described by specifying coherent sheaves $M_U$ on $U = \Spec \wh B$, coherent sheaves $M_{i}$ on $D_i:=\Spec \wh{\mc O}_{\Cs, P_i}$ for each $i$,  together with ``gluing data on the overlaps.'' 

The overlaps in this case are described as the formal completions $D^\times_i$ of the fiber products $\Spec \widehat{B} \times_{\Cs} \Spec \wh{\mc O}_{\Cs, P_i}$, and the gluing data is a choice of an isomorphism $(M_i)_{D_i^\times} \cong (M_U)_{D_i^\times}$. More concretely, the $D^\times_i$ can be described as follows. In a given complete local ring $\wh{\mc O}_{\Cs, P_i}$, the ideal generated by $q$ which describes the closed fiber will factor into a product of of primes $\wp_{i,j}$. For each of these we can consider the localization and completion at the prime. We find that $D^\times_i$ is the disjoint union of the formal spectra of the rings $\left((\wh{\mc O}_{\Cs, P_i})_{\wp_{i,j}}\right)^{\wh{\ \ } \wp_{i,j}}$. In particular, a coherent sheaf over $D^\times_i$ is the data of a finitely generated module over the Noetherian ring $\left((\wh{\mc O}_{\Cs, P_i})_{\wp_{i,j}}\right)^{\wh{\ \ } \wp_{i,j}}$.

In our case, we consider a semistable family of curves $\Cs/S$, such that $\Cs$ is a regular scheme and the closed fiber is reduced. We focus our attention on an isolated node $Q$, and choose points $P_\bullet$ with $Q = P_0$ and with $\Cs_0 \setminus P_\bullet$ smooth. We then find that in $\wh{\mc O}_{\Cs, Q}$, the complete (regular) local ring at $Q$, we may factor $q = s_+ s_-$. Consequently, we may write $\wh{\mc O}_{\Cs, Q} \cong \CC[\![s_+, s_-]\!]$. That is, we have
\[\widehat{ \mc O}_{\Cs, Q} \cong \CC[\![s_+, s_-, q]\!]/(s_+s_- - q) \cong \CC[\![s_+, s_-]\!].\]
In this case, if we let $\wp_+$ be the prime generated by $s_-$ and $\wp_-$ be the prime generated by $s_+$ (in $\widehat{ \mc O}_{\Cs, Q}$), then we find 
\[\left((\wh{\mc O}_{\Cs, P_i})_{\wp_\pm}\right)^{\wh{\ \ } \wp_\pm} = \CC(\!(s_\pm)\!)[\![q]\!].\]
As $\mc{V}_{\Cs}$ (and similarly $\mc{V}_{\Cs} \otimes_{\mc O_\Cs} \omega_{\Cs/S}$) is a limit of coherent sheaves $(\mc{V}_\Cs)_{\leq k}$, we may use the above procedure to describe it.

We choose $U$ so that the torsor $\mathcal{A}ut_{\Cs/S}$ is trivial over $\Spec \widehat{B}$ via the choice of a function $s \in \widehat{B}$ such that $ds$ is a free generator of $\omega_{\Cs/S}(\Spec \widehat{B})$ as an $\widehat{B}$-module. In other words, $s$ is a coordinate on $U$. In particular, sections of $\mc{V}_{\Cs}\otimes_{\mc{O}_{\Cs}} \omega_{\Cs/S}$ on $\Spec \widehat{B}$ can be described as the $\widehat{B}$ module: 
\begin{equation} \label{open trivialization}
    \left(\mc{V}_{\Cs}\otimes_{\mc{O}_{\Cs}} \omega_{\Cs/S}\right)(\Spec \widehat{B}) = \bigoplus_{k \in \NN} V_k \otimes_\CC \widehat{B} \ (\dds )^{k-1}.
\end{equation}
\begin{remark}
    It is important to note that these expressions are not intrinsic to $\mc V_\Cs \otimes_{\mc O_{\Cs}} \omega_{\Cs/S}$ as a sheaf on $\Cs$, but rather depend on a choice of parameter $s$. Different choices give different identifications which correspond to inhomogeneous isomorphisms between the direct sums, but which do preserve the filtrations $(\mc{V}_{\Cs}\otimes_{\mc{O}_{\Cs}} \omega_{\Cs/S})_{\leq k}$.
\end{remark}

Similarly, on $D_{Q}={\mathrm{Spec}}(\widehat{ \mc O}_{\Cs, Q})$, either $s_+$ or $s_-$ can be used to define a trivialization of the torsor $\mathcal{A}ut_{\Cs}$, this time corresponding to the two possible choices of generators $ds_+/s_+$ or $ds_-/s_-$ of $\omega_{\Cs/S}$. These choices allow us to give the following expressions for the sections of our sheaf on $D_{Q}$ as a $\widehat{ \mc O}_{\Cs, Q}$-module:
\begin{equation}\label{eq:TheChiralOnDisk}
   \left(\mc{V}_{\Cs}\otimes_{\mc{O}_{\Cs}} \omega_{\Cs/S}\right)(D_{Q}) = \bigoplus_{k \in \NN} V_k \otimes_\CC \CC[\![s_+, s_-]\!] s_\pm ^{k-1}(\dds_\pm )^{k-1}. 
\end{equation}
In particular, we may express a section $\sigma$ on $D_{Q}$ with respect to either the trivialization given by $s_+$ or by $s_-$. Since $\gamma(s_+)=s_-$, the trivializations of $\mathcal{A}ut_{\Cs}$ associated to the coordinates $s_+$ and $s_-$ (regarded as sections of the torsor) are related by the order $2$ element $(-1)^{L_0}e^{L_1} \in \mathcal{A}ut\mathcal{O}$, which acts on $V$ via the involution $\gamma$ described in \cref{eq:Gamma}.  Hence, we can write sections of the contracted product $(V \otimes_{\CC} \mc O_{\Cs}) \times_{\mathcal{A}ut\mathcal{O}} \mathcal{A}ut_{\Cs}$ over $D_Q$ as
\[ \left(v \otimes f, s_+ \right) = \left(v \otimes f, \gamma s_-\right)  \sim \left( \gamma(v)\otimes f, s_-\right),\] for $f \in \Oc_{\Cs}$. Choosing $v \in V_\ell$, the element of \eqref{eq:TheChiralOnDisk} which in the $s_+$ trivialization is represented by
\begin{equation*} \label{Q pos trivial}
    \sum_{i,j \geq 0} v \otimes x_{i,j} s_+^i s_-^j s_+^{\ell-1} (\dds_+)^{\ell-1},
\end{equation*} is represented with respect to the $s_-$ trivialization as
\[\sum_{i,j \geq 0} \sum_{m = 0}^\ell \frac 1 {m!} L_1^m v \otimes x_{i,j} s_+^i s_-^j s_-^{\ell - m -1} (\dds_-)^{\ell - m -1}. \]
More generally, one should consider a sum of such terms for various values of $\ell$.

Finally we consider the sheaf $\mc V_{\Cs} \otimes_{\mc O_\Cs} \omega_{\Cs/S}$ on $D_{\pm}^\times=\mathrm{Spec}(\mathbb{C}(\!(s_\pm)\!)[\![q]\!])$.  In $D_\pm^\times$, as in $D_Q$, we may use the functions $s_\pm$ to trivialize our torsor. Consequently we have:

\begin{equation} \label{punctured disk expression}
   \left(\mc{V}_{\Cs}\otimes_{\mc{O}_{\Cs}} \omega_{\Cs/S}\right)(D_\pm) = \bigoplus_{k \in \NN} V_k \otimes_\CC \CC(\!(s_\pm)\!)[\![q]\!] \ s_\pm^{k-1} (\dds_\pm )^{k-1}. 
\end{equation}

Without loss of generality the trivializing coordinate $s$ on $U$ maps to our previously chosen trivializing coordinate $s_+$ in $D_+^\times$. That is, the map $i_+ \colon  D_+^\times \hookrightarrow U$ corresponds to maps of rings
\begin{equation} \widehat{B} \rightarrow \mathbb{C}(\!(s_+)\!)[\![q]\!],  \ s\mapsto s_+. \end{equation}
Although it is unnecessary here, to map $s$ to both $s_+$ and $s_-$ simultaneously, one could work \'etale locally.

For notational convenience, it is useful to consider the action of $\mathcal{A}ut\mathcal{O}$ as on  $\LVf_0$, the degree $0$ part of the ancillary algebra and to recall the notation \eqref{eq:J}. For $\rho \in \mathcal{A}ut\mathcal{O}$ and a homogeneous element $a \in V$, we have $\rho J_0(a) = J_0(\rho a)$.  Further, when we use a coordinate $s$ to trivialize our torsor $\mathcal{A}ut_{\Cs}$, we will identify the expression $a_{[\deg(a) - 1 + k]}$ with the element $J_k(a) \in \LVf_{-k}$. Finally, we simplify notation further by omitting the factors of the form $\dds$ from our presentations.

Given $\sigma_U \in \left(\mc{V}_{\Cs}\otimes_{\mc{O}_{\Cs}} \omega_{\Cs/S}\right)(\Spec \widehat{B})$ we write $(\sigma_U)_\pm$ for its restriction to $D_\pm^\times$. Using the notation above, following the explicit expressions of \eqref{open trivialization} and \eqref{punctured disk expression}, we find that if
\[ \sigma_U = \sum_{\ell = 0}^k v_\ell \otimes f_\ell, \] 
then writing $f_+$ for the expansion (restriction) of the regular function $f$ to $\CC(\!(s_+)\!)[\![q]\!]$, we have (as the coordinates are compatible) 
\[ (\sigma_U)_+ = \sum_{\ell = 0}^k v_\ell \otimes (f_\ell)_+ = \sum v_\ell \otimes (g_\ell)_+ s_+^{\ell-1}. \]
On the other hand, if $\sigma_Q \in \left(\mc{V}_{\Cs}\otimes_{\mc{O}_{\Cs}} \omega_{\Cs/S}\right)(D_Q)$, is written as $\sum_{\ell = 0}^k\sum_{i, j \geq 0} v_\ell \otimes x_{i,j}^\ell s_+^{i + \ell-1} s_-^j$. If the section $\sigma_Q$, so represented, is to be compatible and glue together with the section $\sigma_U$ above, we find that 
\begin{equation}\label{eq:SigmaPlus}
    \sum_{\ell = 0}^k \sum_{i, j \geq 0} J_0(v_\ell) x_{i,j}^\ell s_+^i s_-^j = \sum_{i,j,\ell} J_0(v_\ell) x_{i,j}^\ell s_+^{i - j} q^j = \sum_{i,j,\ell} J_{i - j}(v_\ell) x_{i,j}^\ell q^j
\end{equation} 
must represent the expression for $\sigma_U$ restricted to $D_+^\times$. 

To express $\sigma_U$ restricted to $D_-^\times$, following \eqref{eq:LieTheta}, we will make use of the anti-isomorphism $\theta \colon \LVL \to \LVR$ described in \eqref{eq:LieTheta} and related to $\gamma$ via \cref{lem:ThetaGammId}. We then conclude that $\sigma_U$ restricted to $D_-^\times$ is given by the expression 
\begin{align*}\label{eq:SigmaMinus}\sum_{\ell = 0}^k\sum_{i,j \geq 0} \gamma(J_0(v_\ell)) x_{i,j}^\ell s_+^i s_-^j &= \sum_{i,j,\ell} J_0(\gamma(v_k)) x_{i,j}^\ell s_-^{j - i} q^i 
\\ &= \sum_{i,j,\ell} J_{j - i}(\gamma(v_k)) x_{i,j}^\ell q^i 
= \sum_{i,j,\ell} \theta(J_{i - j}(v_\ell)) x_{i,j}^\ell q^i.
\end{align*}

\begin{remark}\label{rem:Correction} Through the above description, we have that the sheaf $\mc{V}_{\Cs}$ discussed at length in \cite{DGT2} agrees, even on the boundary of $\overline{\mathcal{M}}_{g,n}$, with the sheaf $\ms{V}_{\Cs}$ described in \cite{dgt1}.
    \end{remark}

We conclude with two lemmas which will be useful in our applications in the next section.
\begin{lem} \label{chiral density}
Let $\Cs$ be a family of curves over $S$, possibly with nodal singularities. Consider a collection of sections $P_1, \ldots, P_n$ such that $\Cs \setminus P_\bullet = U \subset \Cs$ is affine, and let $Q_1, \ldots Q_k \subset \Cs$ be a finite collection of distinct closed points in $U$ (possibly including nodes). Let $D_{Q_i} = \Spec \wh{\mc{O}}_{\Cs, Q_i}$ be the complete local ring at $Q_i$ with maximal ideal $\wh{\mf{m}}_{\Cs, Q_i}$. Then for any $\ell \geq 0$ and any invertible sheaf of $\mc O_{\Cs}$-modules $\mc L$, the natural map
\[ \left(\mc{V}_{\Cs}\otimes_{\mc{O}_{\Cs}} \mc L\right)(U) \to  
\bigoplus_i \left(\mc{V}_{\Cs}\otimes_{\mc{O}_{\Cs}} \mc L\right)(\Spec \wh{\mc{O}}_{\Cs, Q_i}/\left(\wh{\mf{m}}_{\Cs, Q_i}\right)^{\ell}) \]
is surjective.
\end{lem}
\begin{proof}
As $\mc{V}_{\Cs}\otimes_{\mc{O}_{\Cs}} \mc L = \bigcup\limits_{k} \left(\mc{V}_{\Cs}\otimes_{\mc{O}_{\Cs}} \mc L\right)_{\leq k}$, it suffices to show that the map 
\[ \left(\mc{V}_{\Cs}\otimes_{\mc{O}_{\Cs}} \mc L\right)_{\leq k}(U) \to  
\bigoplus_i \left(\mc{V}_{\Cs}\otimes_{\mc{O}_{\Cs}} \mc L\right)_{\leq k}(\Spec \wh{\mc{O}}_{\Cs, Q_i}/\left(\wh{\mf{m}}_{\Cs, Q_i}\right)^{\ell}) \]
is surjective for all $k$. Since the sheaf $\left(\mc{V}_{\Cs}\otimes_{\mc{O}_{\Cs}} \mc L\right)_{\leq k}$ is free of finite rank  over $\mc{O}_{\Cs}$, then this holds true. Indeed, for any coherent sheaf of modules $M$ on $\Cs$, the natural map $M(U) \to \bigoplus_i M(\Spec \mc O_{\Cs}(U)/\mf m_{\Cs, Q_i}(U)^\ell) = \bigoplus_i M(U) \otimes_{\mc O_{\Cs}(U)} \mc O_{\Cs}(U)/\mf m_{\Cs, Q_i}(U)^\ell$ is seen to be surjective,
using the fact that $\mc O_{\Cs}(U) \to \bigoplus_i \mc O_{\Cs}(U)/\mf m_{\Cs, Q_i}(U)^\ell$ is surjective by the Chinese Remainder Theorem, and that tensoring with $M$ is right exact.
\end{proof}

\begin{lem} \label{nodal isomorphism}
As in \cref{sec:settingSmoothing}, let $\Cs_0$ be a projective curve over $\CC$ with at least one node $Q$, smooth and distinct points $P_\bullet=(P_1, \ldots, P_n)$ such that $\Cs_0 \setminus P_\bullet$ is affine, and formal coordinates $t_\bullet=(t_1,\ldots, t_n)$ at $P_\bullet$. 
Let $\eta \colon \widetilde{\Cs}_0\rightarrow \Cs_0$ be the partial normalization of $\Cs_0$ at $Q$, pointed by $Q_\pm:=\eta^{-1}(Q)$, and choose formal coordinates $s_\pm$ at $Q_\pm$. Let $W^1,\dots, W^n$ be an $n$-tuple of $V$-modules.
Then the map $\alpha_0 \colon W^\bullet \to W^\bullet \otimes \Ac$ defined by $\alpha_0(w) = w \otimes 1$ induces a map between the vector spaces of coinvariants:
\begin{equation*}
[\alpha_0] \colon  [W^\bullet]_{\left(\Cs_0, P_\bullet, t_\bullet\right)}{\overset{\cong}{\longrightarrow}} [W^\bullet \otimes \Ac]_{\left(\widetilde{\Cs}_0, P_\bullet \sqcup Q_\pm, t_\bullet \sqcup s_\pm\right)},
\end{equation*}
which is an isomorphism in case $V$ is $C_1$-cofinite.
\end{lem}
\begin{proof}
Suppose $\Cs_0$ has $m$ nodes in total (including $Q$) and let $\til{\Cs_0}'$ be the (full) normalization of $\Cs_0$. Following \cite[Remark~3.4]{DGK} we find we have maps
\[\xymatrix@C=1.5cm{
W^\bullet \ar@/^1.2pc/[rr]^{\alpha_0'} \ar[r]_-{\alpha_0} & W^\bullet \otimes \Ac \ar[r]_-{\alpha_0''} & W^\bullet \otimes \Ac^{\otimes m}
}\]
inducing corresponding maps $[\alpha_0], [\alpha_0'], [\alpha_0'']$ on the respective coinvariants such that $[\alpha_0']$ an isomorphism. It follows that $[\alpha_0]$ is injective and therefore remains only to show that it is also surjective.

For surjectivity, we follow the spirit of the proof of \cite[Prop.~6.2.1]{DGT2}. We may represent an element of $\Ac$ as given by an expression $a^1_{[n_1]} \cdots a^k_{[n_k]} \otimes 1 \otimes b^1_{[m_1]} \cdots b^r_{[m_r]}$. For simplicity of notation, let us write $a = a^1_{[n_1]} \cdots a^k_{[n_k]}$, $a' = a^2_{[n_2]} \cdots a^k_{[n_k]}$ and $b = b^1_{[m_1]} \cdots b^r_{[m_r]}$. We will show that all elements of the form
$[w \otimes (a \otimes 1 \otimes b)]$ are in the image of $\alpha_0$ by induction on $k - m$, the base case $k - m = 0$ being true by construction (note $b$ has nonpositive degree by definition). For the induction step, let us suppose that $k > 0$ (the case $m < 0$ being similar), and let $d_+'$ be the degree of $a'$ and $d_-$ the degree of $b$. Without loss of generality, we may assume $\deg(a^1_{[n_1]}) \geq \cdots \geq \deg(a^k_{[n_k]}) \geq 0$. 
By \cref{chiral density}, setting $\mc L = \omega_{\Cs_0}(n_1 Q_+ + N Q_-)$ for $N > d_- - \deg(a^1)$, we may find a section $\sigma = a^1 \otimes f$,
\[ \sigma \in 
\left(\mc{V}_{\til \Cs_0}\otimes_{\mc{O}_{\Cs_0}} \omega_{\Cs_0}(n_1 Q_+ + (d_- - 1) Q_-)\right)(\til \Cs_0 \setminus P_\bullet) 
\subset 
\left(\mc{V}_{\til \Cs_0}\otimes_{\mc{O}_{\Cs_0}} \omega_{\Cs_0}\right)(\til \Cs_0 \setminus P_\bullet \sqcup
Q_\pm) 
\] 
such that the image $\sigma^{\mathsf{L}}_{Q_+}$ of $\sigma$ in $\left(\mc{V}_{\til \Cs_0}\otimes_{\mc{O}_{\Cs_0}} \omega_{\Cs_0}\right)(\wh{\mc O}_{\til \Cs_0, Q_+}) \cong \LVL$ has the form $a^1_{[n_1]} + \til a$ where $\deg(\til a) < -d_+'$. 
By construction, $\sigma^{\mathsf{L}}_{Q_-}$ has degree $< d_-$ and consequently 
$\sigma^{\mathsf{R}}_{Q_+}$ has degree $> -d_-$. So we find $\sigma^{\mathsf{L}}_{Q_+} (a' \otimes 1 \otimes b) = a \otimes 1 \otimes b$ and $(a \otimes 1 \otimes b) \sigma^{\mathsf{R}}_{Q_-} = 0$. This tells us
\[\sigma \cdot \big( w \otimes (a' \otimes 1 \otimes b)\big) = (\sigma w) \otimes (a' \otimes 1 \otimes b) +
w \otimes (a \otimes 1 \otimes b)\]
yielding $[w \otimes (a \otimes 1 \otimes b)]
= - [(\sigma w) \otimes (a' \otimes 1 \otimes b)]$,
completing the induction step.
\end{proof}

\section{Smoothing via strong identity elements}\label{sec:SewingThmProof}

Here we prove \cref{thm:Sewing}, which relates the smoothing property of $V$, described here in \cref{def:Sewing2}, to the existence of strong identity elements in $\Ac$. \cref{thm:Sewing} relies crucially on \cref{prop:smoothing}. These results are proved in \cref{sec:SmoothingProofs}. Geometric consequences regarding coinvariants are given in \cref{sec:VBCorProof}.

Throughout this section we will use the notation introduced in \cref{sec:settingSmoothing}, considering two families of marked, parametrized curves $(\Cs, P_\bullet, t_\bullet)$ and $(\widetilde{\Cs}, P_\bullet \sqcup Q_\pm, t_\bullet \sqcup s_\pm)$ over the base scheme $S = \Spec(\CC[\![q]\!])$. As usual, $\Ac_0=\Aa$.

\begin{defn}\label{def:Sewing2} Given a family $(\Cs, P_\bullet, t_\bullet)$, and collection of $V$-modules $W^1,\ldots,W^n$, an element $\Is=\sum_{d \geq 0} \Is_d q^d \in \Ac[\![q]\!]$  defines a \textit{smoothing map} for $W^\bullet$ over $(\Cs, P_\bullet, t_\bullet)$, if  $\Is_0 = 1 \in \Ac_0$, and the map $W^{\bullet} \to   W^{\bullet} \otimes \Ac[\![q]\!]$,  $w \mapsto w \otimes \Is$ extends by linearity and $q$-adic continuity to an $\Lc_{\Cs\setminus P_{\bullet}}(V)$-module homomorphism
$\alpha \colon W^{\bullet}[\![q]\!] \longrightarrow  (W^{\bullet} \otimes \Ac)[\![q]\!]$.
We say that $\Is=\sum_{d \geq 0} \Is_d q^d \in \Ac[\![q]\!]$  defines a \textit{smoothing map} for $V$, if it defines a smoothing map for all $V$-modules $W^\bullet$, over all families $(\Cs, P_\bullet, t_\bullet)$.\end{defn}

\begin{defn}
Smoothing holds for $W^\bullet$ over the family $(\Cs, P_\bullet, t_\bullet)$, if there is an element  $\Is=\sum_{d \geq 0} \Is_d q^d \in \Ac[\![q]\!]$ giving a smoothing map for $W^\bullet$ over  $(\Cs, P_\bullet, t_\bullet)$. \\ $V$ \textit{satisfies smoothing} if smoothing holds for all $W^\bullet$, over all families $(\Cs, P_\bullet, t_\bullet)$.
\end{defn}

\begin{thm}\label{thm:Sewing} Let $V$ be a VOA. Then the algebras $\Ac_d$ admit strong identity elements for all $d\in \NN$ if and only if $V$ satisfies smoothing.
\end{thm}

\subsection{Proof of \texorpdfstring{\cref{thm:Sewing}}{Theorem 5.0.3}}\label{sec:SmoothingProofs}

Following the idea of \cref{lem:StrongUnitConditions}\ref{it:H2}, we make the following definition:
\begin{defn}\label{strong unit equations} We say that a sequence $(\Is_d)_{d \in \NN}$, with $\Is_d \in \Ac$ satisfies the \textit{strong identity element equations} if for every homogeneous $a \in V$, and $n \in \ZZ$ such that  $n \le d$, we have
\begin{equation}\label{eq:StrongUnitEquation}
    J_n(a) \Is_d = \Is_{d - n} J_n(a).
\end{equation} 
\end{defn}
In \cref{strong unit equations} there is no assumption on the (bi-)degrees of the elements $\Is_d \in \Ac$. However, if $\Is_d \in \Ac_d$ is an identity element for each $d$, then by \cref{lem:StrongUnitConditions} they satisfy the strong identity element equations if and only if they are strong identity elements.

\begin{prop} \label{prop:smoothing} Let $V$ be a VOA and let $\Is_d \in \Ac$ for $d \in \NN$. Then $\Is = \sum \Is_d q^d$ defines a smoothing map for $W^\bullet$ over $(\Cs, P_\bullet, t_\bullet)$ if and only if the sequence $(\Is_d)$ satisfies the strong identity element equations \eqref{eq:StrongUnitEquation}. 
\end{prop}
\begin{proof}
The map $\alpha \colon W^\bullet[\![q]\!] \to (W^\bullet \otimes \Ac)[\![q]\!]$ is a map of $\Lc_{\Cs\setminus P_\bullet}(V)$-modules if and only if, for every $\sigma \in  \Lc_{\Cs\setminus P_\bullet}(V)$ and $u\in W^{\bullet}$, one has $\alpha(\sigma(u)) = \sigma(\alpha(u))$. Here, the left hand side equals $(\sigma \cdot u) \otimes \Is$. To describe the right hand side, as is explained in the beginning of \cref{sec:Step2}, we recall that elements of the Lie algebra $\Lc_{\Cs\setminus P_{\bullet}}(V)$ are represented by sections of the sheaf $\mc{V}_{\Cs}\otimes_{\mc{O}_{\Cs}} \omega_{\Cs/S}$ over the affine open set $\Cs\setminus P_{\bullet}$. Consequently we can understand the right hand side in terms of the maps
\begin{align*} 
\left(\mc{V}_{\Cs}\otimes_{\mc{O}_{\Cs}} \omega_{\Cs/S}\right)(\Cs \setminus P_\bullet) &\to \left(\mc{V}_{\Cs}\otimes_{\mc{O}_{\Cs}} \omega_{\Cs/S}\right)(D_\pm^\times) \cong V \otimes_\CC \CC(\!(t)\!)[\![q]\!] \\
\sigma &\mapsto \sigma_\pm^{\mathsf{L}}.
\end{align*}
We let $\sigma_-^{\mathsf{R}} = \theta(\sigma_-^{\mathsf{L}}) \in \CC(\!(t^{-1})\!)[\![q]\!]$.
We then have
\[ \sigma(u \otimes \Is)=\sigma(u) \otimes \Is + u \otimes (\sigma_{+}^{\mathsf{L}} \otimes 1 + 1 \otimes \sigma_{_-}^{\mathsf{R}})(\Is).
\] It follows that $\alpha$ is a map of $\Lc_{\Cs\setminus P_\bullet}(V)$-modules if and only if 
\begin{equation}\label{eq:LModMap}
    \sigma  \cdot \Is =\left(\sigma_{+}^{\mathsf{L}}\otimes 1 +1\otimes  \sigma_{-}^{\mathsf{R}}\right)  \cdot \Is =0.
\end{equation}

We now reframe this in the language developed towards the end of \cref{sec:Step2}. For a section $\sigma \in \left(\mc{V}_{\Cs}\otimes_{\mc{O}_{\Cs}} \omega_{\Cs/S}\right)_{\leq k}(\Cs \setminus P_\bullet)$, writing $s_+ s_-  = q$ on $\wh {\mc O}_{\Cs, Q}$, we may write (in terms of the local trivializations of \cref{sec:Step2}) 
\[\sigma|_{D_Q} = \sum_{\ell = 0}^k\sum_{i,j \geq 0} J_0(v_\ell) \ x_{i,j}^\ell \ s_+^i s_i^j\]
and for this section $\sigma$ we have
\[\sigma_+^{\mathsf{L}} = \sum_{\ell = 0}^k\sum_{i,j \geq 0} J_{i - j}(v_\ell)x^\ell_{i,j}q^j \text{ \quad and \quad  } 
\sigma_-^{\mathsf{R}} = \sum_{\ell = 0}^k\sum_{i,j \geq 0} J_{i - j}(v_\ell)x^\ell_{i,j}q^i.
\]

Putting this together with \eqref{eq:LModMap}, we find that smoothing holds if and only if for all $\sigma$ as above (and for all $k$), we have
\[
\sum_{\ell = 0}^k\sum_{i,j, d \geq 0} x_{ij} \left( J_{i-j}(v_\ell) \cdot   \Is_d \ q^{d+j} -  \Is_d \cdot J_{i-j}(v_\ell) \  q^{d+i}\right) =0.
\]

This in turn holds if and only if each coefficient of $q^m$ is zero, translating to the statement
\begin{equation} \label{eq:coeffq} 
\sum_{\ell = 0}^k\ \sum_{0 \leq i, j \leq m}  x_{ij} \left( J_{i-j}(v_\ell)  \cdot \Is_{m-j} -  \Is_{m-i}  \cdot J_{i-j}(v_\ell) \right) =0,
\end{equation} 
for every $m \geq 0$. 

We note that the systems of equations 
\[J_n(v_\ell) \Is_d = \Is_{d - n} J_n(v_\ell),   \text{ with } n \leq d, \text{ and } d \in \NN, v_\ell \in V_\ell, \ell \in \NN,
\] 
and \[J_{i-j}(v_\ell)  \cdot \Is_{m-j} -  \Is_{m-i}  \cdot J_{i-j}(v_\ell) =0,    \text{ with } 0 \leq i,j \leq m, \text{ and } m \in \NN, v_\ell \in V_\ell, \ell \in \NN, \]
are equivalent after a change of variables. We then have showed that, if $(\Is_d)_{d \in \NN}$ satisfies the strong identity element equations, it follows that $\Is$ defines a smoothing map. It remains to show the converse, namely that if \eqref{eq:coeffq} holds for every $\sigma$, then the strong identity element equations hold. 

We do this by the following strategy: we will show that for every $0 \leq i_0, j_0 \leq m$, 
$m \in \NN$ and $v_{\ell_0}' \in V_{\ell_0}$, we may find a section $\sigma \in \left(\mc{V}_{\Cs}\otimes_{\mc{O}_{\Cs}} \omega_{\Cs/S}\right)(\Cs \setminus P_\bullet)$ so that the expansion of $\sigma$ at $Q$ has the form
\begin{equation}\label{scalpel formula}
\sigma|_{D_Q} = \sum_{\ell = 0}^k\sum_{i,j \geq 0} J_0(v_\ell) \ x_{i,j}^\ell \ s_+^i s_i^j = 
J_{0}(v'_{\ell_0})s_+^{i_0} s_-^{j_0} + \sum_{\ell = 0}^k\sum_{i,j \geq m} J_0(v_\ell) \ x_{i,j}^\ell \ s_+^i s_i^j.  \end{equation}
That is, we argue that the coefficients $x_{i,j}$ in of the terms in \eqref{scalpel formula} of degree less than $m$ are only nonzero in the case $i = i_0, j = j_0$, and in this case $x_{i_0, j_0} = 1$. For such a section $\sigma$, \eqref{eq:coeffq} simply becomes $J_{i_0-j_0}(v_\ell)  \cdot \Is_{m-j_0} -  \Is_{m-i_0}  \cdot J_{i_0-j_0}(v_\ell) =0$, which, as has been noted, is equivalent to the strong identity element equations once we run this argument for all $i_0,j_0$ and $m$.
For this final step, we note that by \cref{chiral density} we have a surjective map
\[ \xymatrix@R-1.5pc@C=.4cm{
\left(\mc{V}_{\Cs}\otimes_{\mc{O}_{\Cs}} \omega_{\Cs/S}\right)(\Cs \setminus P_\bullet) 
\ar[r] & \left(\mc{V}_{\Cs}\otimes_{\mc{O}_{\Cs}} \omega_{\Cs/S}\right)_{\leq k}(\Spec \wh{\mc{O}}_{\Cs, Q}/\left(\wh{\mf{m}}_{\Cs, Q}\right)^{2m}) \ar@{=}[d] \\ 
&  \left(\mc{V}_{\Cs}\otimes_{\mc{O}_{\Cs}} \omega_{\Cs/S}\right)_{\leq k}(\Spec \wh{\mc{O}}_{\Cs, Q})
\otimes_{\wh{\mc{O}}_{\Cs, Q}} 
\wh{\mc{O}}_{\Cs, Q}/(\wh{\mf{m}}_{\Cs, Q})^{2m}
}\]
for every $m \geq 0$. Hence there exists $\sigma \in \left(\mc{V}_{\Cs}\otimes_{\mc{O}_{\Cs}} \omega_{\Cs/S}\right)(\Cs \setminus P_\bullet)$ 
whose image in $\left(\mc{V}_{\Cs}\otimes_{\mc{O}_{\Cs}} \omega_{\Cs/S}\right)(\Spec \wh{\mc{O}}_{\Cs, Q})$ is congruent modulo $(\wh{\mf{m}}_{\Cs, Q})^{2m} = (s_+, s_-)^{2m}$ to 
$J_{0}(v'_{\ell_0})s_+^{i_0} s_-^{j_0}$. It follows \eqref{scalpel formula} holds for this $\sigma$ as desired and the proof is complete. \end{proof}

  We note that  already \cref{prop:smoothing} shows that the smoothing property never depends on modules or specific families of curves:

\begin{cor} \label{cor:smoothing one for all}
Smoothing holds for $W^\bullet$ over a family $(\Cs, P_\bullet, t_\bullet)$ if and only if $V$ satisfies smoothing.
\end{cor}

\begin{proof} If smoothing holds for $W^\bullet$ over a family $(\Cs, P_\bullet, t_\bullet)$, then by \cref{prop:smoothing}, the sequence $(\Is_d)_{d\in\ZZ}$ satisfies the strong identity element equations. But then, invoking again \cref{prop:smoothing}, we deduce that this sequence defines a smoothing map for any choice of modules and family of curves. 
\end{proof}

In what follows, for an element $b \in \Ac = \bigoplus_{i,j} \Ac_{i, j}$, we write $b_{i,j} \in \Ac_{i,j}$ for the corresponding homogeneous component of $b$.

\begin{lem} \label{homogenize units}
For $V$ a VOA, if the sequence  $(\Is_d)_{d \in \NN}$ with $\Is_d \in \Ac$ satisfies the strong identity element equations \eqref{eq:StrongUnitEquation}, then so does the sequence $(\Is_d')_{d \in \NN}$ where $\Is_d' := (\Is_d)_{d, -d}$. \end{lem}
\begin{proof}
Suppose we have a sequence $(\Is_d)$ satisfying the strong identity element equations. When we equate terms of like degree in the expression
\[ J_n(a)  \cdot \Is_d = \Is_{d - n} \cdot J_n(a),\]
we obtain
\[ J_n(a) \cdot (\Is_d)_{i + n,j} = (\Is_d)_{i,j - n} \cdot J_n(a) \]
for every $i, j$. In particular, for $\Is_d' = (\Is_d)_{d, -d}$, we find that the strong identity element equations \eqref{eq:StrongUnitEquation} hold for the sequence $(\Is_d')_{d \in \NN}$, as was claimed.
\end{proof}

In what follows we will use the following equalities, which are a direct consequence of  \cref{lem:abstract-module-associativity}. Let $\xx, \yy \in \Ac$.  Then for every $u, w \in \UV$ we have
\begin{equation} \label{eq:lemStarUV1} u \cdot (\xx \star \yy) = (u \cdot \xx) \star \yy \qquad \text{ and }\qquad (\xx \star \yy) \cdot w = \xx \star (\yy \cdot w).\end{equation}

\begin{lem}\label{lem:MotherLemma}
 Suppose we have a collection of elements $\Is_d \in \Ac_d$ for each $d \geq 0$, with $\Is_0 = 1 \in \Ac_0$. Then,  $\Is_d$ is a strong identity element in $\Ac_d \subset \Ac$, for all $d\in \NN$, if and only if  the sequence $(\Is_d)_{d \in \NN}$ satisfies the strong identity element equations \eqref{eq:StrongUnitEquation}. \end{lem}

\begin{proof}
\Cref{lem:StrongUnitConditions}\ref{it:H2} with $\xx=J_n(v)$ implies that strong identity elements satisfy the strong identity element equations \eqref{eq:StrongUnitEquation}, so we are left to prove the converse statement. 
To show that $\Is_d$ is a strong identity element for each $d$, it suffices to show that $\Is_d$ acts as the identity element on $\Ac_{d,e}$ for every $e \in \ZZ$. That is, for every $\xx\in \Ac_{0,e}$, and $n_1 \le \cdots \le n_r <0$ with $\sum n_i = -d$, we need to show
\[\Is_d \star \left(J_{n_1}(v_1)\cdots J_{n_r}(v_r) \cdot \xx\right)=
J_{n_1}(v_1)\cdots J_{n_r}(v_r)  \cdot \xx.\]
We argue by induction on $r$. 
The base case  $r=0$ holds since by assumptions $\Is_0 = 1 \in \Ac_0=\Aa$, hence $\Is_0 \star \xx = 1 \cdot \xx = \xx$. For the inductive step, we write: 
\begin{align*}
\Is_d \star \left(J_{n_1}(v_1) \cdot J_{n_2}(v_2) \cdots J_{n_r}(v_r) \cdot \xx \right)
&= \Is_d \star \left(\left(J_{n_1}(v_1)\right)\left( J_{n_2}(v_2) \cdots J_{n_r}(v_r) \cdot \xx \right)\right) \\
\scriptsize{\eqref{eq:lemStarUV1}}\quad &= \left(\Is_d  \cdot J_{n_1}(v_1)\right)\star\left( J_{n_2}(v_2) \cdots J_{n_r}(v_r) \cdot \xx \right) \\
\scriptsize{\eqref{eq:StrongUnitEquation}}\quad &= \left(J_{n_1}(v_1) \cdot \Is_{d+n_1}\right) \star \left(J_{n_2}(v_2) \cdots J_{n_r}(v_r) \cdot \xx \right)\\
\scriptsize{\eqref{eq:lemStarUV1}}\quad  &= J_{n_1}(v_1) \cdot \left(\Is_{d+n_1} \star \left(J_{n_2}(v_2) \cdots J_{n_r}(v_r) \cdot \xx \right)\right)\\
\text{\scriptsize{(\textit{by induction})}}\quad &=J_{n_1}(v_1)J_{n_2}(v_2) \cdots J_{n_r}(v_r) \cdot \xx,
\end{align*}
where the last identity holds by induction.
\end{proof}

We may now complete the proof of \cref{thm:Sewing}.

\begin{proof}[Proof of \cref{thm:Sewing}]
Suppose the algebras $\Ac_d$ admit strong identity elements. Writing $\Is_d$ for these unities, we can apply \cref{lem:MotherLemma} to deduce that the sequence $(\Is_d)_{d \in \NN}$ satisfies the strong identity element equations and therefore, by \cref{prop:smoothing}, the element $\Is = \sum \Is_d q^d$ defines a smoothing map for any family of marked curves and choice of modules $W^\bullet$. Hence $V$ satisfies smoothing.

Conversely, if $V$ satisfies smoothing, there exists  $\Is = \sum \Is_d q^d$ which defines a smoothing map for any family of marked curves and choice of modules $W^\bullet$, then by \cref{prop:smoothing}, the sequence $(\Is_d)_{d \in \NN}$ satisfies the strong identity element equations. Using \cref{homogenize units} we may find a new sequence $(\Is_d')_{d \in \NN}$ with $\Is_d' \in \Ac_d$ which also satisfy the strong identity element equations. It follows from \cref{lem:MotherLemma} that the elements $\Is_d'$ are strong identity elements.
\end{proof}

\subsection{Geometric results}
\label{sec:VBCorProof}  We describe in this section some statements about coinvariants, most of which are implications of \cref{thm:Sewing}.

\begin{cor} \label{cor:sewing}
For any VOA $V$,  let $W^\bullet$ be $V$-modules such that the sheaf $[(W^\bullet \otimes \Ac)[\![q]\!]]_{(\widetilde{\Cs}, P_\bullet \sqcup Q_\pm, t_\bullet \sqcup s_\pm)}$  is coherent over $S$. Assume that $\Ac_d$ admits a strong identity element $\Is_d$ for every $d \in \NN$. Set $\Is=\sum_{d \geq 0} \Is_d q^d$, and let $\alpha \colon W^{\bullet}[\![q]\!] \to (W^\bullet \otimes \Ac){[\![q]\!]}$ be the map induced by $w \mapsto w \otimes \Is$ (see \cref{def:Sewing2}). Then the diagram 
\[ \begin{tikzcd}[column sep=3cm]
{[W^{\bullet}[\![q]\!]]_{(\Cs, P_\bullet, t_\bullet)}} {\arrow[r, "{[\alpha]}"]} {\arrow[d , two heads ]}
    & {[W^{\bullet}\otimes \Ac]_{(\widetilde{\Cs}_0 , P_\bullet \sqcup Q_\pm, t_\bullet \sqcup s_\pm)}[\![q]\!]} {\arrow[d, two heads]} \\
    {{[W^{\bullet}]}_{(\Cs_0, P_\bullet, t_\bullet)}} \arrow[r,  "{[\alpha_0]}" ]
& {[W^\bullet \otimes \Ac ]}_{(\widetilde{\Cs}_0 , P_\bullet\sqcup Q_\pm, t_\bullet \sqcup s_\pm)}
\end{tikzcd}\]
commutes, where $\alpha_0 \colon W^\bullet \to W \otimes \Ac$ is given by $w \mapsto w \otimes \Is_0$.
\end{cor}

\begin{proof} The vertical maps are given  by imposing the condition $q=0$, and are surjective. After the identification of $\Ac$ with $\Phi(\Aa)$ provided in \cref{lem:PhiLRPhi}, we see
that the map $[\alpha_0]$ is well defined as in \cite[Proposition 3.3]{DGK}.

By the proof of \cref{thm:Sewing} we deduce that the map $\alpha$ is a map of $\Lc_{\Cs\setminus P_\bullet}(V)$-modules and since $\Lc_{\Cs\setminus P_\bullet}(V) \subset \Lc_{\widetilde{\Cs}\setminus (P_\bullet \sqcup Q_\pm)}(V)$, this induces a map of coinvariants 
\[ [W^{\bullet}[\![q]\!]]_{(\Cs, P_\bullet, t_\bullet)} \longrightarrow [(W^{\bullet}\otimes \Ac)[\![q]\!]]_{(\widetilde{\Cs}, P_\bullet \sqcup Q_\pm, t_\bullet \sqcup s_\pm)}\] 
whose reduction modulo $q$ is indeed $[\alpha_0]$. Finally, we use \cref{cor:DiagQues} to identify 
\[[(W^{\bullet}\otimes \Ac)[\![q]\!]]_{(\widetilde{\Cs}, P_\bullet \sqcup Q_\pm, t_\bullet \sqcup s_\pm)} \cong [W^{\bullet}\otimes \Ac]_{(\widetilde{\Cs}_0, P_\bullet \sqcup Q_\pm, t_\bullet \sqcup s_\pm)}[\![q]\!] 
\] which concludes the proof.
\end{proof}

To state the following consequence, we recall that sheaves of coinvariants $\VV(V;W^\bullet)$ are attached to coordinatized curves $(C,P_\bullet, t_\bullet)$ such that $C \setminus P_\bullet$ is affine. By Propagation of vacua \cite{codogni,DGT2}, we may drop the latter condition, so that $\VV(V;W^\bullet)$ can be considered a sheaf on $\tMgn[n]$, the stack of stable coordinatized curves. Depending on $V$ and on $W^\bullet$, this further descends to a sheaf over $\bMgn[n]$. To formulate our next result it is convenient to introduce the following notation.

\begin{defn}
Let $V$ be a VOA. We say that $V$ has \textit{coherent coinvariants} if for every family of stable and pointed coordinatized curves $(C, P_\bullet, t_\bullet)$, and modules $W^\bullet$, the sheaf of coinvariants $[W^{\bullet}]_{(C, P_\bullet, t_\bullet)}$ is coherent.
\end{defn}
\begin{defn}
Let $V$ be a VOA. We say that $V$ has \textit{finite gluing} if for every stable and pointed coordinatized curve $(\Cs, P_\bullet, t_\bullet)$ with a node $Q$, and modules $W^\bullet$, the space of coinvariants 
$[(W^{\bullet}\otimes \Ac)[\![q]\!]]_{(\widetilde{\Cs}, P_\bullet \sqcup Q_\pm, t_\bullet \sqcup s_\pm)}$ is coherent over $\Spec(\CC[\![t]\!])$. 
\end{defn}

\begin{remark} \label{rmk:fgenlocfree}We make two observations:
\begin{enumerate}[label={(\roman*)}]
    \item We note that if $V$ is $C_2$-cofinite, then by \cite[Corollary 4.2]{DGK}, $V$ has coherent coinvariants and finite gluing. As we shall see, this is also true for VOAs like the Heisenberg which are generated in degree 1, but are not $C_2$-cofinite. 
    \item \label{it:fgenlocfree} By \cref{cor:DiagQues}, if $V$ has finite gluing it follows that  $[(W^{\bullet}\otimes \Ac)[\![q]\!]]_{(\widetilde{\Cs}, P_\bullet \sqcup Q_\pm, t_\bullet \sqcup s_\pm)}$ is actually free over $\CC[\![q]\!]$. \end{enumerate}
\end{remark}

We begin with an auxiliary result. 

\begin{lem}  \label{lem:alphaiso}
Let $V$ be a $C_1$-cofinite VOA that satisfies smoothing and such that ${[W^{\bullet}[\![q]\!]]_{(\Cs, P_\bullet, t_\bullet)}}$ and ${[(W^{\bullet}\otimes \Ac)[\![q]\!]]_{(\widetilde{\Cs}, P_\bullet \sqcup Q_\pm, t_\bullet \sqcup s_\pm)}}$ are coherent over $S$. Then the map $[\alpha]$ defined in \cref{cor:sewing} is an isomorphism.
\end{lem}

\begin{proof} Since $V$ is $C_1$-cofinite, \cref{nodal isomorphism} ensures that $[\alpha_0]$ is an isomorphism. Since the source  and target of $[\alpha]$ is finitely generated and the target is locally free (see \cref{rmk:fgenlocfree} \ref{it:fgenlocfree}), Nakayama's lemma ensures that $[\alpha]$ is an isomorphism as well.    
\end{proof}

To state the next results, we shall refer to the moduli stacks  $\tMgn[n]$, parametrizing families of stable pointed curves of genus $g$ with coordinates,  and $\overline{\mathcal{J}}_{g,n}$, of stable pointed curves of genus $g$ with first order tangent data, and projection maps 
\[\tMgn[n] \rightarrow \overline{\mathcal{J}}_{g,n} \rightarrow \overline{\mathcal{M}}_{g,n},\]
discussed in detail in \cite[\S 2]{DGT2}.  Recall the notation from \cref{family coinvariants}.

\begin{cor}  \label{cor:SewingAndFactorizationCorollary}
Let $W^1,\dots, W^n$ be simple modules over a $C_1$-cofinite vertex operator algebra $V$, such that coinvariants are coherent for curves of genus $g$, and such that $\mathfrak{A}_d(V)$ admit strong identity elements for all $d\in \mathbb{Z}_{\ge 0}$. Then sheaves of coinvariants are locally free, giving rise to a vector bundle $\VV_g(V;W^\bullet)^{\overline{\mathcal{J}}_{g,n}}$ on $\overline{\mathcal{J}}_{g,n}$. If the conformal dimensions of $W^1,\dots, W^n$  are rational, these sheaves define vector bundles $\VV_g(V;W^\bullet)$ on $\bMgn[n]$.
\end{cor}

\begin{proof} 
Since a sheaf of $\mathcal{O}_S$-modules is locally free if and only if is coherent and flat, in order to show  a coherent sheaf  $[W^{\bullet}]_{\mathcal{L}}$ is locally free, it suffices to show that it is flat. For this, we can use the valuative criteria of \cite[Thm~11.8.1, \S 3]{EGA4}  to reduce to the case that our base scheme is $S=\Spec(\CC[\![q]\!])$. By \cite[Ex. II.5.8]{hartshorne:77}, since $S$ is Noetherian and reduced, and since formation of coinvariants commutes with base change, by \cref{coinvariant pullback},  it 
suffices to check that vector spaces of coinvariants have the same dimension over all pointed and coordinatized curves.

Our strategy for checking this condition holds is to argue by induction on the number of nodes, reducing to the base case where the curve has no nodes. 

To take the inductive step, following the notation of \cref{cor:sewing},  let $\Cs_0\rightarrow {\mathrm{Spec}}(k)$ be a nodal curve with $k+1$ nodes, and let $\Cs \rightarrow {\mathrm{Spec}}(\CC[\![q]\!])$ be a smoothing family with $\Cs_0$ the special fiber. By \cref{contingent smoothing} and by \cref{lem:alphaiso}, we deduce that $[\alpha]$ is an isomorphism, so that the dimension of the space of coinvariants associated with $\Cs_0$  agrees with the dimension of the space of coinvariants for the partial normalization $\widetilde{\Cs}_0$, a curve with $k$ nodes. Therefore, by induction, the vector space $[W^{\bullet}]_{({\Cs}_0, P_\bullet, t_\bullet)}$ has the same dimension as the vector space of coinvariants associated with a smooth curve.

We are then left to show that spaces of coinvariants associated with smooth curves of the same genus have the same dimensions.  This holds since coinvariants $[W^{\bullet}]_{\mathcal{L}}$ are by assumption coherent, and moreover, when restricted to families of smooth coordinatized curves, they define a sheaf which admits a projectively flat connection \cite{bzf, dgt1}. We have shown that $[W^{\bullet}]_{\mathcal{L}}$ is flat, giving rise to a coherent and locally free sheaf  on $\tMgn[n]$. As shown in \cite{DGT2}, this sheaf of coinvariants descends to a sheaf of coinvariants  $\VV_g(V;W^\bullet)^{\overline{\mathcal{J}}_{g,n}}$ on $\overline{\mathcal{J}}_{g,n}$. Moreover, for any collection of simple $V$-modules $W^\bullet$ with rational conformal weights,  as is explained in \cite[\S 8.7.1]{DGT2}, the sheaves are independent of coordinates and  will further descend to vector bundles on $\bMgn[n]$, denoted $\VV_g(V;W^\bullet)$.
 \end{proof}

\begin{remark} \label{rmk:morerational}
 We note the following consequences of \cref{cor:SewingAndFactorizationCorollary}:
\begin{enumerate}\item For a collection of simple modules over a $C_2$-cofinite VOA, the sheaf of coinvariants will give vector bundles $\VV_g(V;W^\bullet)$ on $\bMgn[n]$ whenever the algebras $\mathfrak{A}_d(V)$ admit strong identity elements. To see this, we note that the coinvariants will be coherent by \cite{DGK}, and by \cite[Corollary 5.10]{MiyamotoC2} any simple module over a $C_2$-cofinite $V$ has rational conformal weight. 
\item  \label{it:rat} Combining \cref{rmk:rational} with \cref{cor:SewingAndFactorizationCorollary} one may show that sheaves of coinvariants from $C_2$-cofinite and rational VOAs define vector bundles on $\bMgn[n]$, recovering \cite[VB Corollary]{DGT2}.

\item \label{item:gg}By \cite{DG}, sheaves defined by simple modules over VOAs that are generated in degree 1 are coherent over rational curves. If $V$ satisfies smoothing, such sheaves of coinvariants descend to vector bundles $\VV_0(V;W^\bullet)^{\overline{\mathcal{J}}_{0,n}}$ on $\overline{\mathcal{J}}_{0,n}$. If the conformal dimensions of the modules are in $\mathbb{Q}$, they further descend to vector bundles $\VV_0(V;W^\bullet)$ on $\overline{\mathcal{M}}_{0,n}$. Moreover, by \cite{DG}, these bundles are globally generated. We refer to \cref{sec:Heisenberg} and \cref{ex:HeisVB} for an application of this using the Heisenberg VOA.
\end{enumerate}
\end{remark}

\section{Higher Zhu algebras and mode transition algebras} \label{sec:BigThmProof}
Recall that if any of the equivalent properties of \cref{lem:StrongUnitConditions} hold, we say that $\Is_d \in \Ac_d$ is a strong identity element. Here we prove \cref{thm:Big}, one of our two main results. In order to formulate it, we introduce the map 
\[\mu_d \colon \Ac_d \to \Aa_d, \qquad \mu_d(\alpha \otimes u \otimes \beta)= [\alpha u \beta]_d.\] This map is well defined and  fits into an exact sequence (see \cref{lem:right exact seq})
\begin{equation}\label{eq:RES}
\Ac_d \overset{\mu_d}{\longrightarrow} \Aa_d \overset{\pi_d}{\longrightarrow} \Aa_{d-1} \longrightarrow 0.\end{equation}

\begin{thm}\label{thm:Big}
\begin{enumerate}
\item \label{thma} If the mode transition algebra $\Ac_d$ admits an identity element
    element, then the map $\mu_d$ in \eqref{eq:RES} is injective, and the sequence splits, yeilding an isomorphism
        $\Aa_{d} \cong \Ac_d \times \Aa_{d-1}$ as rings. In particular, if
        $\Ac_j$ admits an identity element for every $j \leq d$, then $\Aa_d \cong \Ac_d
        \oplus \Ac_{d-1} \oplus \cdots \oplus \Ac_0.$
\item \label{thmb} If  $\Ac_d$ admits a strong identity element for all $d\in \NN$, so that smoothing holds for $V$, then given any generalized Verma module $W=\PhiL(W_0)=\bigoplus_{d\in \NN} W_d$ where $L_0$ acts on $W_0$ as a scalar with eigenvalue $c_W \in \CC$, there is no proper submodule $Z \subset W$ with $c_Z - c_W >0$ for every eigenvalue $c_Z$ of $L_0$ on $Z$ (see \cref{rmk:cw}). 
\end{enumerate}
\end{thm}

We note that \cref{abstract splitting} specializes to Part~\eqref{thma} of \cref{thm:Big}. It therefore remains to prove Part~\eqref{thmb} of \cref{thm:Big}.

\begin{proof}
We say that an induced admissible module $W=\PhiL(W_0)$ has the LCW property if  $L_0$ acts on $W_0$ as a scalar with eigenvalue $c_W \in \CC$, and there is no proper submodule $Z \subset W$ with $c_Z - c_W >0$ for every eigenvalue $c_Z$ of $L_0$ on $Z$.  Suppose for contradiction that $V$ admits a module $W=\PhiL(W_0)$, and $W$ does not have the  LCW property.  We will show that there must be a $d\in \NN$ such that $\Ac_d$ is not unital, contradicting our assumptions. 

By hypothesis, $W$ has a proper submodule $Z$ with $c_Z - c_W >0$ for every eigenvalue $c_Z$ of $L_0$ on $Z$. In particular, $Z$ is not induced in degree zero over $\Aa$.  Let $z_d$ be any homogeneous element in  $Z$ of smallest degree $d >0$, so that $z_d \in W_d$.  
By assumption $\Ac_d$ is unital, with unity $\uu_d=\sum_{i}\alpha_i \otimes     1 \otimes \beta_i$, where each $\alpha_i$ has degree $d$ and each $\beta_i$ has degree $-d$.
The action of $\Ac$ on $W$ restricts to an action of $\Ac_d$ on $W_d$, and since $\uu_d$ is the unity of $\Ac_d$ we have
\[\Ac_d \times W_d \longrightarrow W_d, \qquad(\uu_d,z_d)\mapsto \uu_d\star z_d = z_d.\]
Unraveling the definition of $\star$ and its associativity properties we have $\uu_d \star z_d = \sum_{i}\alpha_i \cdot ( \beta_i \cdot z_d)$. But now since the degree of $\beta_i \cdot z_d$ is zero and $Z$ is a submodule, we have that $\beta_i \cdot z_d \in Z \cap W_0=0$, since $Z$ does not have a degree zero component. It then follows that $z_d=\uu_d \star z_d = 0$, giving a contradiction since we assumed $z_d \neq 0$.  
\end{proof}

\begin{remark} \label{rmk:cw} Although the eigenvalues $c_Z$ and $c_W$ are in general complex numbers, the difference $c_Z- c_W$ is always an integer, hence it makes sense to require that this number be positive.  In fact, every eigenvalue of the action of $L_0$ on $W$ will be obtained by shifting $c_W$ by a non-negative integer. The condition $c_Z -c_W > 0 $ coincides then with $c_Z \neq c_W$. We remark that when $V$ is $C_2$-cofinite, then the eigenvalues of $L_0$ are necessarily rational numbers \cite{MiyamotoC2}.
\end{remark}

\section{Mode transition algebra for the Heisenberg vertex algebra}\label{sec:Heisenberg}
In this section we describe the mode transition algebras for the Heisenberg vertex algebra.  This result is stated in \cref{claim:HigherHeis} and, as a consequence, in \cref{sec:ProofCon} we obtain that \cite[Conjecture 8.1]{addabbo.barrow:level2Zhu} holds.  We refer \cite{bzf,lepli, mint2,BVY, addabbo.barrow:level2Zhu} for more details about the vertex algebra, denoted $\pi$, $V_{\hat{\mathfrak{h}}}(1,\alpha)$, $M_a(1)$ and $M(1)_a$ in the literature, and which we next briefly describe.

\subsection{Background on the Heisenberg vertex algebra}\label{sec:HBackground}

Let $\mathfrak{h}=H \CC(\!(t)\!) \oplus k\CC$ be the extended Heisenberg algebra and consider the Heisenberg vertex algebra $V=\pi$. Let $\Uu_1(\mathfrak{h})$ denote the quotient of the universal enveloping algebra $\Uu(\mathfrak{h})$ by the two sided ideal generated by $k-1$. Following \cite[Section 4.3]{bzf} the Lie algebra $\LVL$ is naturally embedded inside
\[ \overline{\Uu(\mathfrak{h})}^{\mathsf{L}}:= \varprojlim \dfrac{\Uu_1(\mathfrak{h})}{\Uu_1(\mathfrak{h})\circ 
 H t^N\CC[t]}.
\] The map is induced by $(b_{-1})_{[n]} \mapsto Ht^n$. This embedding induces a natural isomorphism between $\UVL$ and $\overline{\Uu(\mathfrak{h})}^{\mathsf{L}}$ which translates the filtration on $\UVL$ into the canonical filtration on $\overline{\Uu(\mathfrak{h})}^{\mathsf{L}}$ induced by the filtration on $\CC(\!(t)\!)$ given by $F^p\CC(\!(t)\!) = t^{-p}\CC[t^{-1}]$.

A similar construction holds for $\LVR$ and $\UVR$, where the extended Heisenberg algebra $\mathfrak{h}=H \CC(\!(t)\!) \oplus \CC$ is replaced by $\mathfrak{h}=H \CC(\!(t^{-1})\!) \oplus \CC$.

The sub ring $\UV$ of $\UVL$ and $\UVR$ has a natural gradation induced by $\deg( Ht^n) = -n$. We can then deduce that the associated zero mode algebra $\Ac_0$ is isomorphic to the commutative ring $\CC[x]$, where the element $(b_{-1})_{[0]} = H \in \UV_0$ is identified with the variable $x$. Combining these results we can explicitly compute all the mode transition algebras.

\subsection{Mode transition algebras for the Heisenberg vertex algebra}

\label{sec:ProofHigherHeis}
We can now state and prove the main result of this section. 

\begin{prop}\label{claim:HigherHeis}  There is a natural identification $\Ac_d(\pi)\cong \mathrm{Mat}_{p(d)}(\CC[x])$, where $p(d)$ is the number of ways to decompose $d$ into a sum of positive integers. In particular $\Ac_d$ is unital for every $d \in \NN$. 
\end{prop}

\begin{proof} Denote by $P(d)$ the set of partitions of $d$ into positive integers, so that $|P(d)|=p(d)$. We represent every element $[r_1|\cdots|r_n]={\bm{r}} \in P(d)$ by a decreasing sequence  of positive integers $r_1 \geq \dots \geq r_n \geq 1$ such that $\sum_i r_i=d$ and for some $n \in \NN$. For every pair $({\bm{r}}, \bm{s}) \in P(d)^2$, we denote by $\varepsilon_{\bm{r},{\bm{s}}}$ the element in $\Ac_d$ given by 
\[ Ht^{-r_1} \circ \cdots \circ Ht^{-r_n} \otimes 1 \otimes Ht^{s_m} \circ \cdots \circ Ht^{s_1}.
\] From the explicit description  of $\UV$ given above, and the fact that the Zhu algebra $\Aa=\CC[x]$  at level zero is Abelian, we have that the set whose elements are $\varepsilon_{\bm{r},{\bm{s}}}$ freely generates $\Ac_d$ as an $\Aa$-module. Moreover, using a computation similar to \cref{rmk:explostar}, one may show that 
\[Ht^{s_m} \circ \cdots \circ Ht^{s_1}\ostar Ht^{-r_1} \circ \cdots \circ Ht^{-r_m} =\begin{cases}
||\bm{r}||
&\text{if } {\bm{s}} = {\bm{r}}\\
0 & \text{otherwise},
\end{cases}\]
where $||\bm{r}||$ is a non-zero, positive integer entirely depending on $\bm{r}$. It then follows that 
\[ \varepsilon_{\bm{r'},{\bm{s}}} \star \varepsilon_{{\bm{r}},{\bm{s}'}} = \begin{cases}
     ||\bm{r}|| \varepsilon_{\bm{r'},{\bm{s}'}} &\text{if } {\bm{s}} = {\bm{r} }\\
     0& \text{otherwise}.
\end{cases}
\]
By identifying $\varepsilon_{\bm{r},{\bm{s}}}$ with the element of $\mathrm{Mat}_{p(d)}(\CC)$ having $\sqrt{||\bm{r}||}\sqrt{||\bm{s}||}$ in the $(\bm{r},{\bm{s}})$-entry, and zero otherwise, the above description gives an isomorphism of rings between $\Ac_d$ and $\CC[x] \otimes \mathrm{Mat}_{p(d)}(\CC) = \mathrm{Mat}_{p(d)}(\CC[x])$, as is claimed.\end{proof}

\begin{ex} \label{ex:Heis=2} We can explicitly compute the coefficient $||\bm{r}||$ appearing in the proof of \cref{claim:HigherHeis}. Let $\bm{r}=[r_1|\dots|r_n]$ be a partition of $d$ consisting of $s$ many distinct elements $r_{i_1}, \dots, r_{i_s}$ (at most $s=d$). For every $j \in \{1, \dots, s\}$, let $m_j$ be the multiplicity of $r_{i_j}$ in $\bm{r}$. Then we have
\[ ||\bm{r}||= \prod_{i=1}^n r_i \cdot \prod_{j=1}^s m_j!.
\]
 For instance $||[1| \cdots|1]||=d!$ and $||[d]||=d$. Moreover $||[r_1|\cdots|r_d]||=r_1 \cdots r_d$ if the $r_i$'s are all distinct. 
\end{ex}

\subsection{The conjecture of Barron and Addabbo}\label{sec:ProofCon} We now prove \cite[Conj. 8.1]{addabbo.barrow:level2Zhu}. 

\begin{cor}\label{cor:ConjABCor}For all $d\in \NN$, one has that
    $\Aa_d(\pi)\cong \mathrm{Mat}_{p(d)}(\CC[x]) \oplus \Aa_{d-1}(\pi)$.
\end{cor}

\begin{proof} This follows from \cref{claim:HigherHeis} and Part \textit{(a)} of \cref{thm:Big}.     \end{proof}

\begin{remark}\label{rmk:ab} By \cite[Remark 4.2]{BVY},  $\Aa_0(\pi)\cong \mathbb{C}[x]$,  $\Aa_1(\pi)\cong \mathbb{C}[x]\oplus \Aa_0(\pi)$, and by \cite[Theorem 7.1]{addabbo.barrow:level2Zhu}, $\Aa_2(\pi)\cong \mathrm{Mat}_{p(2)}(\CC[x]) \oplus \Aa_{1}(\pi)$. 
\end{remark}

\subsection{Vector bundles from the Heisenberg VOAs}
We now equip $\pi$ with a conformal vector $\omega$, so that it becomes a VOA. The following result shows that the application of  \cref{thm:Big}  produces new examples, beyond the well-studied case of sheaves of coinvariants defined by rational and $C_2$-cofinite VOAs.

Let $\overline{\mathcal{J}}_{0,n}$ be the stack parametrizing families of stable pointed curves of genus zero with first order tangent data, and recall that the forgetful map  $\pi: \overline{\mathcal{J}}_{0,n} \rightarrow \overline{\mathcal{M}}_{0,n}$ makes $\overline{\mathcal{J}}_{0,n}$ a $\mathbb{G}_m^{\oplus n}$-torsor over $\overline{\mathcal{M}}_{0,n}$.
\begin{cor}\label{ex:HeisVB} Sheaves of coinvariants defined by simple modules over the Heisenberg VOA form globally generated vector bundles on $\overline{\mathcal{J}}_{0,n}$.  If conformal dimensions of modules are in $\mathbb{Q}$, these descend to form globally generated vector bundles on $\overline{\mathcal{M}}_{0,n}$.
    \end{cor}

\begin{proof}
    By \cref{claim:HigherHeis}, the mode transition algebras for the Heisenberg VOAs are unital.  Moreover, the formula of the star product implies that these are strong identity elements. Hence by \cref{thm:Big}, the Heisenberg VOA satisfies smoothing. Since the Heisenberg VOA is by definition generated in degree $1$, the assertion follows from \cref{cor:SewingAndFactorizationCorollary}, as described in \cref{rmk:morerational}~(\ref{item:gg}),. 
    \end{proof}  

\begin{remark}
    Unlike bundles of coinvariants given by representations of rational and $C_2$-cofinite VOAs, higher Chern classes of  bundles on $\overline{\mathcal{M}}_{g,n}$ from  \cref{cor:SewingAndFactorizationCorollary} (like those on $\overline{\mathcal{M}}_{0,n}$ from \cref{ex:HeisVB}) are elements of the
    tautological ring since we do not know if they satisfy factorization, and hence we do not that the Chern characters form
    a semisimple cohomological field theory as in \cite{moppz, dgt3}.
\end{remark}


\medskip

\section{Mode transition algebras for Virasoro VOAs}\label{eg:virasoro} For $c \in \CC$, by $\Vir_c = {M_{c,0}}/{<L_{-1}1>}$ we mean the (not necessarily simple) Virasoro VOA of central charge $c\in \CC$. 

\subsection{\texorpdfstring{$\Vir_c$}{Virc}}
By \cite{WWang}, when $c \neq c_{p,q}=1-\frac{6(p-q)^2}{pq}$, then $\Vir_c$ is a simple VOA, but it is not rational or $C_2$-cofinite. When $c=c_{p,q}$, the VOA $\Vir_c$ is not simple, but its simple quotient $L_c$ will be rational and $C_2$-cofinite, and therefore satisfy smoothing. We therefore only consider $\Vir_c$, for any values of $c$, and not $L_c$.

\begin{prop} Let  $\Vir_c$ be the Virasoro VOA.
\begin{enumerate}
    \item[(a)] The first mode transition algebra $\Ac_1(\Vir_c)$ is not unital, and so $\Vir_c$ does not satisfy smoothing.
    \item[(b)] The kernel of the canonical projection $\Aa_1(\Vir_c) \to \Aa_0(\Vir_c)$ is isomorphic to  $\Ac_1(\Vir_c)$.\end{enumerate}
\end{prop}

\begin{proof} We first prove \textit{(a)}. 
By \cite[Lemma 4.1]{WWang}, one has $\Aa_0(\Vir_c) \cong \CC[t]$, where  the class of $(L_{-2}1)_{[1]}$ is mapped to the generator $t$.

Here, as in Heisenberg case, $\LVf_{\pm 1}$ is a one dimensional vector space, with generators denoted $u_{\pm 1}$, so that $\Ac_1(\Vir_c) = u_1 \Aa_0(\Vir_c) u_{-1}$ . We can choose $u_1 = (L_{-2} 1)_{[0]}$ and $u_{-1} =  (L_{-2}1)_{[2]}$, and to understand the multiplicative structure of $\Ac_1(\Vir_c)$ we are only left to compute $[u_{-1}, u_1]$. Since $L_{-2}1$ is the conformal vector of $\Vir_c$, we can identify $(L_{-2}1)_{[n]}$ with the element $\Lc_{n-1}$ of the Virasoro algebra, and the bracket of $L(\Vir_c)$ coincides with the bracket in the Virasoro algebra. Hence we obtain 
\[ [u_{-1},u_1] = [(L_{-2}1)_{[2]},(L_{-2} 1)_{[0]}] =  [\Lc_{1}, \Lc_{-1}] = 2\Lc_{0} = 2(L_{-2} 1)_{[1]}. 
\]
We then have an identification of $\Ac_1(\Vir_c)$ with $(\CC[t],+,\star)$, where $+$ denotes the usual sum of polynomials, while $f(t) \star g(t) = 2 t f(t) g(t)$. In particular, this implies that $\Ac_1(\Vir_c)$ is not unital.

We now show \textit{(b)}. By \cite{WWang}, ${\Aa}_0(\Vir_c)$ is generated by $L_{-2}\textbf{1}+O_0(V)$ and $L_{-2}^2\textbf{1}+O_0(V)$ so that \[{\Aa}_0(\Vir_c)\cong \mathbb{C}[x,y]/(y-x^2-2x)\cong \mathbb{C}[x],\]
\[L_{-2}\textbf{1}+O_0(V) \mapsto x+(q_0(x,y)), \qquad  L_{-2}^2\textbf{1}+O_0(V) \mapsto y+(q_0(x,y)),\]
where $q_0(x,y)=y-x^2-2x$. By \cite[Theorem 4.7]{BarronVirL1}, 
$\Aa_1(\Vir_c)$ 
is generated by $L_{-2}\textbf{1}+O_1(V)$ and  $L_{-2}^2\textbf{1}+O_1(V)$, and by \cite[Theorem 4.11]{BarronVirL1} on has that 
\[{\Aa}_1(\Vir_c)\cong \mathbb{C}[x,y]/((y-x^2-2x)(y-x^2-6x+4)),\] 
\[L_{-2}\textbf{1}+O_1(V)\mapsto x+(q_0(x,y)q_1(x,y)), 
\qquad L_{-2}^2\textbf{1}+O_1(V)\mapsto y + (q_0(x,y)q_1(x,y)),\]
where $q_0(x,y)=y-x^2-2x$ and $q_1(x,y)=y-x^2-6x+4$ (see also \cite[\S 5]{BarronVirL1}). With the change of variables $X=y-x^2-6x+4$ and $Y=y-x^2-2x$, one has
\[\Aa_1(\Vir_c) = \dfrac{\CC[X,Y]}{XY} \qquad \text{ and} \qquad \Aa_0(\Vir_c) = \CC[X],
\] so that the kernel of the projection $\Aa_1(\Vir_c) \to \Aa_0(\Vir_c)$ is identified with the ideal $K_1$ generated by $Y$ inside $\Aa_1(\Vir_c)$. Since $XY=0$, the ideal $K_1$ is isomorphic to $(Y\CC[Y],+,\cdot)$. Furthermore, this algebra is isomorphic to the algebra $(\CC[t], +, \star)$ through the assignment $Yf(Y)\mapsto f(2t)$. This shows that, abstractly, $\Ac_1(\Vir_c)$ is identified with the kernel of $\Aa_1(\Vir_c) \to \Aa_0(\Vir_c)$.

We now see directly that this identification is provided by the natural map $\mu_1 \colon \Ac_1 \to \Aa_1(V)$, which is induced by $(L_{-2} 1)_{[0]} \otimes 1 \otimes   (L_{-2}1)_{[2]} \mapsto [(L_{-2} 1)_{[0]}(L_{-2} 1)_{[2]}]$ as in \cref{lem:right exact seq}. To check that indeed $\Ac_1(\Vir_c)$ naturally identifies with the kernel of  $\Aa_1(\Vir_c) \to \Aa_0(\Vir_c)$, it is enough to show that 
\[ \tilde{Y} -2(L_{-2} 1)_{[0]}(L_{-2} 1)_{[2]} \in N^2\UV_0, \]
where $\tilde{Y}$ is any lift of $Y$ to $\UV_0$. We choose
\[\tilde{Y} = (L_{-2}L_{-2}1)_{[3]}-(L_{-2} 1)_{[1]}(L_{-2} 1)_{[1]} - 2(L_{-2} 1)_{[1]}.
\]
To simplify the notation, we will now write  $\Lc_{n}$ to denote $(L_{-2}1)_{[n+1]}$. Using the Virasoro relations we obtain that this is the same as 
\begin{align*} \tilde{Y} &= 2 \sum_{n \geq 2} \Lc_{-n}\Lc_{n} + \Lc_{-1}\Lc_{1} + \Lc_{1}\Lc_{-1} + \Lc_0\Lc_0 - \Lc_0\Lc_0- 2\Lc_0\\
&=  2 \sum_{n \geq 2} \Lc_{-n}\Lc_{n} + \Lc_{-1}\Lc_{1} + \Lc_{1}\Lc_{-1}- 2\Lc_0\\
&= 2 \sum_{n \geq 2} \Lc_{-n}\Lc_{n} +2\Lc_{-1}\Lc_{1} + 2\Lc_0- 2 \Lc_0 \\
&= 2 \sum_{n \geq 2} \Lc_{-n}\Lc_{n} +2\Lc_{-1}\Lc_{1} \\
&= 2 \sum_{n \geq 2} \Lc_{-n}\Lc_{n} +2(L_{-2}1)_{[0]}(L_{-2}1)_{[2]}, 
\end{align*} and since $\sum_{n \geq 2} \Lc_{-n}\Lc_{n}  \in N^2\UV_0$, the proof is complete.
\end{proof}

\section{Questions}\label{sec:OtherQuestions}
Here we ask a few other questions that arise from this work.

\subsection{Not rational and strongly generated in higher degree}\label{sec:SecondQuestion}
Keeping in mind the example of the Virasoro VOA from \cref{eg:virasoro} and   \cref{thm:Big}, we ask the following:

\begin{question}\label{ques:CQ}
    For $V$ a $C_2$-cofinite and non-rational VOA, not generated in degree $1$, can one always find a pair $Z\subset W$ where $W=\PhiL(W_0)$ is  induced by an indecomposable $\Aa_0(V)$-module $W_0$, such that $L_0$ acts on $W_0$ as a scalar with eigenvalue $c_W \in \CC$,  and a proper submodule $Z \subset W$, with $c_Z - c_W >0$ for every eigenvalue $c_Z$ of $L_0$ on $Z$. 
\end{question} 

\medskip

In \cref{sec:EvidenceTriplet} we provide an example of such a pair of modules $Z \subset W$ for the triplet vertex operator algebra $\mathcal{W}(p)$.  This particular example was suggested to us in a communication with Thomas Creutzig. Simon Wood gave us a proof of \cref{claim:Wood}, a crucial detail for this example.   The features of such an example (and that it should exist for the triplet)  were described to us by Dra\v{z}en Adamovi\'{c}.

\subsubsection{Triplet VOAs}\label{sec:EvidenceTriplet} Let $\mathcal{W}(p)$ denote the triplet vertex operator algebra.  There are $2p$ simple $\mathcal{W}(p)$-modules 
$X_s^+$, and $X_s^-$, for $1\le s \le p$. Following \cite[Eq (2.39)]{TW}, we write $\ov{X}_s^{\pm}$ for the quotient $\Aa_0(X_s^\pm)= {X_s^{\pm}}/{I_0(X_s^{\pm})}$ which are simple modules over Zhu algebra $\Aa=\Aa_0(\mathcal W(p))$. And in this case, one also has in the notation of \cite{BVY}, that $\Omega(X_s^\pm)=(X_s^\pm)_0=\overline{X}_s^\pm$.
In particular, since $\ov X_s^{\pm}$ is an $\Aa$-module,  we may consider $\PhiL(X_s^{\pm})$. Moreover, using for instance \cite[Eq (3.8)]{TW},  the eigenvalues of the action of $L_0$ on the indecomposable modules $\overline{X}_s^\pm$, i.e. the conformal weights, satisfy $cw(X_{p - s}^-) > cw(X_s^+)$. The induced module $\PhiL(\ov X_s^+)$ can be identified with a quotient of the projective cover of $X_s^+$, as follows. By \cite[Proposition 4.5]{NTTrip} (see also \cite{TW}) the projective cover $P^+_s$  of $X^+_s$ has socle filtration of length three consisting of submodules $S_0 \subset S_1 \subset S_2 = P^+_s$ with $S_0 \cong X_s^+ \cong S_2/S_1$ and $S_1/S_0 \cong 2 X_{p - s}^-$.  

\begin{claim}\label{claim:Wood}
 $\PhiL(\ov X_s^+) \cong P^+_s / X_s^+$
\end{claim}
\begin{proof}
The $\Aa$-module $\overline{X}_s^+$ is indecomposable, and as $\PhiL$ takes indecomposable modules to indecomposable modules (eg.~\cite{DGK}), one has that $\PhiL(\ov X_s^+)$ is an indecomposable  admissible $\mathcal{W}(p)$ module. It follows that  $\ov X_s^+$ will be the weight space of least conformal weight in $\PhiL(\ov X_s^+)$, and as $X_s^+$ is generated by its lowest weight space $\ov X_s^+$, we get a canonical surjective map $\PhiL(\ov X_s^+) \to X_s^+$. By projectivity, the map from the projective cover $P_s^+ \to X_s^+$ lifts to a map $P_s^+ \to \PhiL(\ov X_s^+)$. As this map is surjective on the least weight space, the weight of $X_s^+$, and $\PhiL(\ov X_s^+)$ is generated by this subspace by construction, the map $P_s^+ \to \PhiL(\ov X_s^+)$ is surjective and so $\PhiL(\ov X_s^+)$ is a quotient of $P_s^+$.

The kernel of this quotient must contain the socle (which isomorphic to $X_s^+$), otherwise  $\PhiL(\ov X_s^+)$ would have two composition factors isomorphic to $X_s^+$ contradicting the size of its lowest weight space. The kernel cannot be larger, otherwise $\PhiL(\ov X_s^+)$ would admit a nontrivial extension by $X^-_{p-s}$ (which as noted, has greater conformal weight than $X_s^+$). Note that if we had such an extension $0 \to X^-_{p - s} \to E \to \PhiL(\ov X_s^+) \to 0$, the lowest weight spaces of $E$ and $\PhiL(\ov S_s^+)$ would be isomorphic, producing a universal map $\PhiL(\ov X_s^+) \to E$ whose composition with the map in the above sequence would be the universal map from $\PhiL(\ov X_s^+)$ to itself -- that is, the identity. Consequently we would have a splitting of our exact sequence and the extension would be trivial. Thus $\PhiL(\ov X_s^+)$ is isomorphic to ${P_s^+}/{X_s^+}$. \end{proof}
 
In particular, by \cref{claim:Wood}, $W=\PhiL(\ov X_s^+)=S_2/S_0$ would have a sub-module isomorphic to $Z=S_1/ S_0 \cong 2 X_{p - s}^-$, and the conformal weight $cw(Z)=cw(X_{p - s}^-)$ would then be strictly larger than the conformal weight $cw(W)=cw(\ov X_s^+)$.

\begin{prop}\label{prop:Trip}
    $\mathcal{W}(p)$ does not satisfy smoothing.
\end{prop}

\begin{proof}The example shows by \cref{thm:Sewing} and \cref{thm:Big}  that the triplet does not satisfy smoothing.
\end{proof}

\subsubsection{More general triplet vertex algebras}\label{sec:SF}
In \cite{GenTrip0, AMGenTrip1, AMGenTrip2, AMGenTrip3} the more general triplet vertex algebras $\mathcal{W}_{p_+,p_-}$ with $p_\pm \ge 2$ and $(p_+,p_-)=1$ are studied.  From their results, for $p_+=2$ and $p_-$ odd, the $\mathcal{W}_{p_+,p_-}$ are $C_2$-cofinite and not rational.  We would like to know the answer to \cref{ques:CQ} for this family of VOAs.

\subsubsection{Other $C_2$-cofinite, non-rational VOAs from extensions}\label{sec:Extensions}
In \cite{CKL} the authors discover three new series of $C_2$-cofinite and non-rational VOAs, via application of the vertex tensor category theory of  \cite{HLZ1, HLZ}, which are not directly related to the triplets.  They also list certain modules for these examples.  We would like to know the answer to \cref{ques:CQ} for these new families of VOAs.

\subsection{Local freeness in case \texorpdfstring{$V$}{V} does not satisfy smoothing}\label{sec:BundleQuestions}
\begin{question}\label{ques:VBQ}
    Are there particular choices of modules $W^\bullet$ over a  $V$ that do not satisfy smoothing, for which  sheaves $\VV(V; W^\bullet)$ form vector bundles on $\overline{\mathcal{M}}_{g,n}$?  
\end{question}

By \cref{cor:SewingAndFactorizationCorollary},  if $V$ satisfies smoothing, and if the sheaves of coinvariants are coherent,  they form vector bundles.  However, if $V$ does not satisfy smoothing, it is still an open question about whether these sheaves are locally free. For instance one could ask this for the triplet vertex algebras, which do not satisfy smoothing, but are $C_2$-cofinite, so their representations define coherent sheaves  on  $\overline{\mathcal{M}}_{g,n}$.  

\subsection{Generalized Constructions}\label{sec:GoodTriples}
In \cref{sec:Triples} the notion of triples of associative algebras is introduced, and to a good triple (see \cref{good definition}) we  associate many of the standard notions affiliated with a VOA from higher level Zhu algebras to mode transition algebras (see \cref{sec:GHZ} and \cref{sec:GMTA}).   Some of the results proved here apply in this more general context. For instance, as was already noted in the introduction,  
the exact sequence in (5), and Part \textit{(a)} of \cref{thm:Big} hold in this generality.  It would be interesting to further develop this theory, and it
is therefore natural to ask the following question:
\begin{question}\label{ques:GoodTriples}
    What are other examples of generalized (higher level) Zhu algebras and generalized mode transition algebras, beyond the context of VOAs?
\end{question}

\appendix

\section{Split filtrations}\label{sec:Appendix}
This appendix contains a number of details about graded and filtered completions, and their relationships to one another.  These serve to provide simple definitions of the building blocks of our constructions and uniform proofs of their properties. 

\subsection{Filtrations}

The purpose of this first section is to provide a framework in which we can simultaneously discuss and compare filtered and graded versions of certain constructions. In particular, this will give us a language appropriate for dealing simultaneously with both graded and filtered versions of the universal enveloping algebra of a vertex operator algebra which we recall in \cref{enveloping algebra defs}.

\begin{defn}[Left and right filtrations]
\label{filtration def}
    Let $X$ be an Abelian group. A left filtration on $X$ is a sequence of subgroups $X_{\leq n} \subset X_{\leq n+1} \subset X$ for $n \in \mathbb Z$. Similarly, a right filtration on $X$ is a sequence of subgroups $X_{\geq n} \subset X_{\geq n-1} \subset X$ for $n \in \mathbb Z$.
\end{defn}

\begin{remark} \label{ambidextrous filtered}
If $X$ has a left filtration of subgroups $X_{\leq n}$ we may produce a right filtration by setting $X_{\geq n} = X_{\leq -n}$. Hence the concepts of left and right filtrations are essentially equivalent. We will work in this section exclusively with left filtrations, but will have use for both left and right filtrations eventually. The reader should therefore keep in mind that the results in this section all have their right counterparts. If $X$ is a graded Abelian group, we can naturally regard it as filtered by setting $X_{\leq n} = \bigoplus_{i \leq n} X_i$.
\end{remark}
\begin{notation} \label{filtered notation}
If $X$ is a filtered Abelian group and $S \subset X$ is a subset, we write $S_{\leq n}$ to mean $S\cap X_{\leq n}$.
\end{notation}

We will now introduce some concepts which we will use throughout.

\begin{defn}[Exhaustive filtration]
Let $X$ be a (left) filtered Abelian group. We say that the filtration on $X$ is exhaustive if $\bigcup_n X_{\leq n} = X$ and separated if $\bigcap_n X_{\leq n} = 0$. 	
\end{defn}

\begin{defn}[Splittings of filtrations] \label{split def} 
Given a (left) filtered Abelian group $X$, we define the associated graded group to be $\gr X = \bigoplus_n \left(X_{\leq n}/X_{\leq n - 1}\right)$. 
A splitting of $X$ is defined to be a graded subgroup $X' = \bigoplus_n X'_n \subset X$ with $X'_n \subset X_{\leq n}$ such that for each $n$, the induced map $X'_n \to (\gr X)_n$ is an isomorphism.
\end{defn}

\begin{defn}[Split-filtered Abelian groups] \label{split filter def}
A split-filtered Abelian group is a filtered Abelian group $(X, \leq)$ together with a graded Abelian group $X' = \bigoplus_n X'_n$, and an inclusion $X' \subset X$ which defines a splitting as in \cref{split def}. 
\end{defn}

\begin{notation} \label{split parts notation}
For $X$ a split-filtered Abelian group, and $x \in X_{\leq n}$, we write $x_n \in X_n'$ and $x_{< n} \in X_{\leq n-1}$ for the unique elements such that $x = x_n + x_{< n}$.  
\end{notation}

\begin{ex}[Concentrated split-filtrations] \label{concentrated at zero}
If $X$ is an Abelian group, with no extra structure, we may define a split filtered structure on it, $X[d]$, which we refer to as ``concentrated in degree $d$,'' by:
\[ 
X[d]_{\leq p} = 
\begin{cases}
    0 & \text{if } p < d, \\
    X & \text{if } p \geq d.
\end{cases} \hspace{1cm} \text{ and } \hspace{1cm} 
X[d]'_p = 
\begin{cases}
    0 & \text{if } p \neq d,\\
    X & \text{if } p = d.
\end{cases}
\]
If $X$ is an Abelian group, with no extra structure, we may define the trivial split-filtration on $X$ to be $X[0]$.
\end{ex}

\begin{ex}\label{graded as split-filtered}
If $X = \bigoplus_n X_n$ is a graded Abelian group, we may also consider it as a split Abelian group with respect to the filtration $X_{\leq n} = \bigoplus_{p \leq n} X_p$. In this case the inclusion of $X$ into itself provides the splitting.
\end{ex}

\begin{defn}[Split-filtered maps] \label{split maps def}
If $X$ and $Y$ are split-filtered Abelian groups and $d \in \mathbb Z$, we say that a group homomorphism $f \colon Y \to X$ is a map of degree $d$ if $f(Y_{\leq p}) \subset X_{\leq p + d}$ and $f(Y'_p) \subset X'_{p + d}$ for all $p$.
\end{defn}

\begin{defn}[Split-filtered subgroups]
If $X$ and $Y$ are split filtered Abelian groups with $Y \subset X$, we say $Y$ is a split-filtered subgroup of $X$ is the inclusion is a degree $0$ map of split-filtered Abelian groups.
\end{defn}

\begin{lem} \label{split kernels}
Let $f \colon X \to Y$ be a degree $d$ homomorphism of split-filtered Abelian groups. Then $\ker f$ is a split-filtered subgroup of $X$.
\end{lem}
\begin{proof}
We verify that $(\ker f)_{\leq p} = (\ker f')_p + (\ker f)_{\leq p - 1}$, where $f' \colon X'_p \to Y'_{p + d}$ is the restriction of $f$. For this, we simply note that by definition, $f$ induces a map $X'_p \oplus X_{\leq p - 1} \to Y'_{p + d} \oplus Y_{\leq p + d - 1}$ which preserves the decomposition.
\end{proof}

The following lemma is straightforward to verify.
\begin{lem} \label{split images}
Suppose $f \colon Y \to X$ is a degree $d$ map of split-filtered Abelian groups. Then restricting the filtration on $X$ to the image of $f$, we find $(\im f)_{\leq p} = \im\left(f|_{Y_{\leq p - d}}\right)$. Further, $\im f' \subset \im f$ defines a splitting, giving $\im f$ the structure of a split-filtered Abelian group.
\end{lem}

\begin{lem} \label{cokernel splitting}
Suppose $f \colon Y \to X$ is a degree $d$ map of split-filtered Abelian groups. Then $\coker(f') \subset \coker(f)$ defines a splitting, giving $\coker(f)$ the structure of a split-filtered Abelian group.
\end{lem}
\begin{proof}
Via \cref{split images}, we know that 
$\im(f') \subset \im(f)$ defines a split-filtered structure on $\im(f)$. As $\coker(f) = \coker\left(\im(f) \to X\right)$, it therefore suffices to consider the case where $f$ is injective.
We have a diagram of (split) short exact sequences:
\[\xymatrix@R=.4cm @C=1.3cm{
0 \ar[r] & Y_p' \ar[d]  \ar[r] & Y_{\leq p} \ar[r] \ar[d] & Y_{\leq p - 1} \ar[d] \ar[r] & 0 \\
0 \ar[r] & X_p' \ar[r]  & X_{\leq p} \ar[r] & X_{\leq p - 1} \ar[r]& 0,
}\]
where the vertical maps are injections. By the snake lemma, this gives a split short exact sequence of cokernels $(X/Y)_{\leq p} = (X'/Y')_{p} \oplus (X/Y)_{\leq p}$. In particular, the inclusion $X_p \to X_{\leq p}$ induces an inclusion $(X'/Y')_p \subset (X/Y)_{\leq p}$ giving our desired splitting.
\end{proof}

\begin{prop} \label{split colimits}
The category of split-filtered Abelian groups is an Abelian category which is cocomplete, i.e. closed under colimits.
\end{prop}
\begin{proof}
The fact that we have an Abelian category is a consequence of \cref{cokernel splitting}, \cref{split images}, \cref{split kernels}. By \cite[Prop.~2.6.8]{Weibel:HA} cocompleteness follows from being closed under direct sums, which can be checked by noticing that 
$\bigoplus_{\lambda \in \Lambda} X^\lambda$ is 
split-filtered with respect to the graded subgroup $\bigoplus_{\lambda \in \Lambda} (X^\lambda)'$.
\end{proof}

\subsection{Modules and tensors} 

\begin{defn}
Let $R$ be a ring and $M$ a left (or right) $R$-module. We say that $M$ is a split-filtered $R$ module if it is a split filtered Abelian group such that $M_{\leq n}$, $M'$, and $M'_n$ are $R$-submodules of $M$ for all $n$.
\end{defn}

\begin{lem} \label{unflat injectivity}
    Let $R$ be a ring. Suppose $M$ is a split filtered right $R$-module and $N$ a left $R$-module. Then the natural maps $M_{\leq p} \otimes_R N \to M \otimes_R N$ and $M_{p}' \otimes_R N \to M \otimes_R N$ are injective.
\end{lem}
\begin{proof}
    As $M \otimes_R N$ is a directed limit of $M_{\leq i} \otimes_R N$ taken over all $i$, it follows that an element $x \in M_{\leq p} \otimes_R N$ maps to $0$ in $M \otimes_R N$ if and only if it maps to $0$ in $M_{\leq i} \otimes_R N$ for some $i > p$. Consequently, by induction, it suffices to show that the map $M_{\leq i} \otimes_R N \to M_{\leq i + 1} \otimes_R N$ is injective for all $i$. But note that (for any $i$, not just $i > p$):
    \[M_{\leq i + 1} \otimes_R N = (M_{\leq i} \oplus M_{i + 1}') \otimes_R N \cong (M_{\leq i} \otimes_R N) \oplus (M_{i + 1}' \otimes_R N),\]
    from which we see that $M_{\leq i} \otimes_R N \to M_{\leq i + 1} \otimes_R N$ is in fact split injective and therefore $M_{\leq p} \otimes_R N \to M \otimes_R N$ is injective.

    But by this same reasoning, we see  (taking $i + 1 = p$) that $M_p' \otimes_R N \to M_{\leq p} \otimes_R N$ is split injective. Since $M_{\leq p} \otimes_R N \to M \otimes_R N$ is injective by the previous paragraph, we find that $M_p' \otimes_R N \to M \otimes_R N$ is injective as well.
\end{proof}

\begin{deflem} \label{split tensor over unfiltered}
Let $R$ be a ring, $M$ a split-filtered right $R$-module and $N$ a split-filtered left $R$-module. Then $M \otimes_R N$ is naturally a split-filtered $R$-module by defining $(M \otimes_R N)_{\leq n}$ to be $\sum\limits_{p + q = n} M_{\leq p} \otimes_R N_{\leq q}$ and $(M \otimes_R N)'_n$ to be $\bigoplus\limits_{p + q = n} M_p' \otimes_R N_q'$.
\end{deflem}
Note that it follows from \cref{unflat injectivity} that these are in fact submodules of $M \otimes_R N$.
\begin{proof}
We have:
\begin{multline*}
M_{\leq p} \otimes_R N_{\leq q} = 
(M'_p \oplus  M_{\leq p - 1}) \otimes_R (N'_{q} \oplus N_{\leq q-1 }) \\= 
(M'_{p} \otimes_R N'_{q}) \oplus (M_{\leq p - 1} \otimes_R N'_{q}) \oplus (M'_{p} \otimes_R N_{\leq q - 1}) \oplus (M_{\leq p - 1} \otimes_R N_{\leq q - 1}) \\
\subseteq (M' \otimes_R N')_n \oplus (M \otimes_R N)_{\leq n - 1}.
\end{multline*}
This shows $(M \otimes_R N)_{\leq n }\subseteq (M' \otimes_R N')_n \oplus (M \otimes_R N)_{\leq n - 1}$. The other inclusion is straightforward.
\end{proof}

\subsection{Rings and ideals}

\begin{defn} \label{def:filtered ring}
If $U$ is a filtered Abelian group with a (not necessarily associative, not necessarily unital) ring structure, we say that it is a filtered ring if $U_{\leq p}U_{\leq q} \subset U_{\leq p + q}$. If $U$ is a filtered ring and we are given $U'$ a graded subring providing a splitting, we say $U$ is a split-filtered ring.
\end{defn}

\begin{defn} \label{split filtered module}
    Let $U$ be a split-filtered ring and $M$ a left $U$-module, split-filtered as an Abelian group. 
    We say that $M$ is a split-filtered $U$-module if we have $U_{\leq p} M_{\leq q} \subset M_{\leq p + q}$ and $U_{p}'M_q' \subset M_{p + q}'$.
    Equivalently, $M$ is a split filtered $U$-module if the multiplication map $U \otimes_{\ZZ} M \to M$    is a split filtered map (where the split-filtration on $U \otimes_{\ZZ} M$ is described in \cref{split tensor over unfiltered}).
\end{defn}

\begin{deflem} \label{split tensor}
Let $U$ be a split-filtered ring. $M$ a split-filtered right $U$-module and $N$ a split-filtered left $U$-module. 
Then $M \otimes_U N$ is naturally a split-filtered $U$-module by defining $(M \otimes_U N)_{\leq n}$ and $(M \otimes_U N)'_n$ to be the images in 
$M \otimes_U N$ of $\bigoplus\limits_{p + q = n} M_{\leq p} \otimes_\ZZ N_{\leq q}$ and 
$\bigoplus\limits_{p + q = n} M_p' \otimes_\ZZ N_q'$ respectively.
\end{deflem}
\begin{proof}
Consider the map
$f \colon M \otimes_{\mathbb Z} U \otimes_{\mathbb Z} N \to M \otimes_{\mathbb Z} N$
given by $f(x \otimes u \otimes y) = xu \otimes y - x \otimes uy$. By definition, $M \otimes_U N$ is defined as the cokernel of this map. Regarding the domain and codomain as split-filtered via \cref{split tensor over unfiltered}, we see that this is a degree $0$ map of split-filtered Abelian groups. So by \cref{cokernel splitting}, the cokernel is split filtered.
\end{proof}

\begin{lem} \label{split the products}
Let $U$ be a split-filtered ring and let $S, T \subset U$ be arbitrary split-filtered additive subgroups. Then $ST$ is split filtered with $(ST)_{\leq n} = \sum_{p + q = n} S_{\leq p} T_{\leq q}$ and $(ST)'_n = \sum_{p + q = n} S_p' T_q'$.
\end{lem}
\begin{proof}
If we consider the tensor product $S \otimes_{\mathbb Z} T$, with its split-filtered structure of \cref{split tensor}, we see that the multiplication map $S \otimes_{\mathbb Z} T \to ST \subset U$ is a degree $0$ map of split-filtered groups. The result now follows from \cref{split images}.
\end{proof}

\begin{lem} \label{split ideal}
Let $U$ be a split-filtered associative, unital ring, and let $X \subset U$ be a split-filtered additive subgroup. Then the ideal generated by $X$ in $U$ is also split-filtered with homogeneous part the ideal of $U'$ generated by $X'$.
\end{lem}
\begin{proof}
It follows from \cref{split tensor over unfiltered} that $U \otimes_{\mathbb Z} X \otimes_{\mathbb Z} U$ is split filtered with homogeneous part $U' \otimes_{\mathbb Z} X' \otimes_{\mathbb Z} U'$. As the multiplication map $U \otimes_{\mathbb Z} X \otimes_{\mathbb Z} U \to U$ is a map of degree $0$, it follows that its image, the ideal generated by $X$ is split-filtered.
\end{proof}

\begin{lem} \label{universal split filtration}
Suppose $L$ is a split-filtered Lie algebra over a commutative (associative and unital) ring $R$. Then the universal enveloping algebra $U(L)$ is a split-filtered algebra with respect to the graded subalgebra $U(L') \subset U(L)$.
\end{lem}
\begin{proof}
It follows from \cref{split tensor} and \cref{split colimits} that the tensor algebra $T(L)$ is split-filtered with respect to $T(L')$. Let $X \subset T(L)$ be the image of the map $L \otimes_{\mathbb Z} L \to T(L)$ defined by $x \otimes y \mapsto x \otimes y - yx - [x, y]$ (note the tensor in the preimage is over $\mathbb Z$ and in the image is over $R$). As this is a map of degree $0$, its image $X$ is split-filtered with homogeneous part spanned by the analogous expressions with homogeneous elements. 
By \cref{split images}, it follows that the ideal generated by $X$ is also split-filtered. Finally \cref{cokernel splitting} tells us that the quotient by this ideal, the universal enveloping algebra, is also split filtered as described.
\end{proof}

\subsection{Seminorms}

The algebraic structures which naturally arise in studying the universal enveloping algebras of a VOA come with additional topological structure in the form of a seminorm. In this section we will examine seminorms and their interactions with gradings, filtrations and split-filtrations.

\begin{defn} \label{seminorm def}
A system of neighborhoods of $0$ in an Abelian group $X$ is a collection of subgroups $\N^nX \subset X$, $n \in \mathbb Z$, with $\N^nX \subset \N^{n-1}X$ and $\bigcup_n \N^nX = X$.	
\end{defn}

\begin{defn}
    A (non-Archimedean) seminorm on an Abelian group $X$ is a function $X \to \mathbb R_{\geq 0}$, $x \mapsto |x|$ such that $|0| = 0$ and $|x + y| \leq \max\{|x|, |y|\}.$
\end{defn}

\begin{defn}
    A psuedometric on a set $X$ is a function $d: X \times X \to \RR_{\geq 0}$ such that $d(x, x) = 0$, $d(x, y) = d(y, x)$ and $d(x, z) \leq d(x, y) + d(y, z)$ for all $x, y, z \in X$.
\end{defn}
\begin{remark} \label{seminorm remark}
The notion of a system of neighborhoods is equivalent to the notion of an Abelian group seminorm (we always assume these to be non-Archimedean), where we would set $|x| = e^{-n}$ if $x \in \N^n X \setminus \N^{n+1}X$ or $|x| = 0$ if $x \in \bigcap_n \N^nX$. Such a seminorm also gives rise to a pseudometric by setting $d(x, y) = |x - y|$. Finally, these give rise to a topology on $X$ whose basis is given by open balls with respect to this pseudometric. We see that addition is continuous with respect to this topology.
\end{remark}
With this remark in mind, we will refer to systems of neighborhoods of $0$ and seminorms interchangeably, and will often refer to an Abelian group with a system of neighborhoods of $0$ as a seminormed Abelian group.
\begin{remark} \label{filtered is seminormed}
It follows from the definition that a system of neighborhoods (and hence a seminorm) is precisely the same as an exhaustive right filtration. We may therefore consider the seminorm associate to either a right or left exhaustive filtration (in view of \cref{ambidextrous filtered}).
\end{remark}

\begin{defn}[Restriction of seminorms] \label{seminorm restriction def}
If $X$ is a seminormed Abelian group and $Y\subset X$ is a subgroup, we will consider $Y$ a seminormed Abelian group via the restriction of the seminorm. That is, we set $\N^nY = \N^nX \cap Y$.	
\end{defn}

\begin{defn}[Seminormed rings and modules]
Let $U$ be a ring which is seminormed as an Abelian group. We say that $U$ is a seminormed ring if $|xy| \leq |x||y|$, or, equivalently, $(\N^pU)(\N^q U) \subset \N^{p + q}U$. If $M$ is a left $U$-module which is seminormed as an Abelian group, we say that it is a seminormed left module if $|xm| \leq |x||m|$ for all $x \in U$, $m \in M$.
\end{defn}
\begin{remark} \label{filtered continuous mult}
It follows immediately from \cref{def:filtered ring} that a filtered ring (not necessarily associative or unital) becomes a seminormed ring with respect to the seminorm induced by the filtration as in \cref{filtered is seminormed}, and that multiplication map is continuous with respect to the induced topology.
\end{remark}
\begin{warning}
We will often consider seminorms on rings which are not ring seminorms, but just Abelian group seminorms.
\end{warning}

\begin{defn}[Split-filtered seminorms]
Let $X$ be an split-filtered Abelian group. We say that a seminorm is split-filtered if each of its neighborhoods $\N^nX$ are split-filtered subgroups of $X$. In this case we will simply refer to $X$ as a split-filtered seminormed Abelian group. 
\end{defn}

The following notion captures a property that we will often seek: that smaller filtered parts of a given filtered Abelian group lie in progressively smaller neighborhoods.

\begin{defn}[Tight seminorms]
Suppose $X$ is a filtered seminormed Abelian group. We say that the $X$ is tightly seminormed if
for all $m, p$ there exists $d$ such that $X_{\leq -d} \subset \N^m X_{\leq p}$.
\end{defn}

\begin{lem} \label{tight density}
Suppose $X$ is a split-filtered seminormed Abelian group whose seminorm is tight. Then $X'$ is dense in $X$.
\end{lem}
\begin{proof}
Let $x \in X$. We can choose $n, p$ with $x \in \N^nX_{\leq p} \subset X$. For any $m$, we need to show that there exists $x' \in X'$ with $x - x' \in \N^m X$. We can write $\N^n X_{\leq p} = \N^n X'_p \oplus \N^n X_{\leq p - 1}$ and iterating this expression, we find $\N^n X_{\leq p} = \bigoplus_{i = p - d + 1}^p \N^n X'_i \oplus \N^n X_{\leq p - d}$ for any $d > 0$. But by the tightness of the seminorm, choosing $d >\!\!>  0$ we can ensure $\N^n X_{\leq p - d} \subset X_{\leq p - d} \subset \N^m X_{\leq p}$. In particular, we may write $x = x' + y$ with $x' \in \bigoplus_{i = p - d + 1}^p \N^n X'_i \subset X'$ and $y \in \N^m X_{\leq p}$ as desired.
\end{proof}

\subsection{Graded and filtered completions}
\begin{defn}[Graded-complete and filtered-complete Abelian groups]
Let $X$ be a normed Abelian group. If $X$ is graded, we say that it is graded-complete if each of the graded subspaces $X_n$ is complete. If $X$ is filtered, we say that is is filtered-complete if each subspace $X_{\leq n}$ is complete. 
\end{defn}
\begin{defn}[Short homomorphisms]
Let $X, Y$ be seminormed Abelian groups. A group homomomorphism $f \colon X \to Y$ is called a short (or metric) homomorphism if $|f(x)| \leq |x|$ for all $x \in X$.
\end{defn}

\begin{deflem}[Separated completions]
Let $X$ be a seminormed Abelian group. Then we may form the (separated) completion $\wh X$ of $X$, which a complete normed Abelian group equipped with a short map $\iota \colon X \to \wh X$ which is universal for short maps to complete normed Abelian groups. That is, for every complete normed Abelian group $Y$ and short homomorphism $X \to Y$, there is a unique factorization of this map as $X \overset{\iota}\to \wh X \to Y$.
\end{deflem}
This can be constructed in the usual way via equivalence classes of Cauchy sequences. The following Lemma is a consequence of the fact that a metric space maps injectively into its completion:
\begin{lem} \label{completion kernel}
Let $X$ be a seminormed Abelian group. Then the canonical map $X \to \wh X$ has kernel $\bigcap_{n \in \mathbb Z} \N^nX$. In particular, $X \to \wh X$ is injective exactly when the seminorm on $X$ is actually a norm.
\end{lem}

\begin{lem} \label{subquotient completions}
Let $W \subset Z \subset X$ be subgroups of a seminormed Abelian group $X$. Then in the induced seminorm on $Z/W$, the separated completion of $Z/W$ can be identified with $\wh Z/\wh W$, and $\wh Z$, $\wh W$ can be identified with the closures of the images of $Z$ and $W$ in $\wh X$ respectively.	
\end{lem}
\begin{proof}
The latter identification of completions and closures is straightforward to check. We note that there is a universal map $\wh Z \to \wh{\left(Z/W\right)}$ of separated completions with $W$ in the kernel. But as the image is Hausdorff, it follows that $\ov W$ must also be in the kernel. But now we see that the map $Z/W \to \wh Z/\wh W$ is therefore universal giving us $\wh Z/\wh W \cong \wh{\left(Z/W\right)}$ as desired.
\end{proof}

We have various closely related universal constructions as follows. 
\begin{deflem}[Filtered and graded completions] \label{fg completion def}
Let $X$ be a seminormed Abelian group. If $X$ is graded, then we can construct a short homomorphism $X \to \grcomp X$ which is universal for short homomorphisms to graded-complete Abelian groups. If $X$ is filtered, then we can construct a short homomorphism $X \to \filcomp X$ which is universal for short homomorphisms to filtered-complete Abelian groups.
\end{deflem}
\begin{proof}
We set $\grcomp X = \bigoplus_n \grcomp X_n$ where $\grcomp X_n = \wh{X_n}$, and $\filcomp X = \bigcup_n \filcomp X_{\leq n}$ where $\filcomp X_{\leq n} = \wh{X_{\leq n}}$.
\end{proof}
\begin{remark}
These are also described in \cite{MNT} as the degreewise completion and the filterwise completion respectively.
\end{remark}

\begin{lem}\label{split completions}
For $X$ a split-filtered tightly seminormed Abelian group,
 $\filcomp X$ is a split-filtered tightly seminormed Abelian group with respect to the graded subgroup $\grcomp{X'}$.
\end{lem}
\begin{proof}
Let us note that the natural morphism $\grcomp{X'} \to \filcomp X$ is an inclusion. For this, suppose that we have a pair of Cauchy sequences $(a_n), (b_n)$ in $X'_p$ which have the same image in $\filcomp X$. Without loss of generality, we may select subsequences and re-index (possibly after modifying our starting index so an appropriate integer), and assume $a_n - b_n \in \N^n X_{\leq p}$ for all $n$. But as $\N^n X'_p = \N^n X_{\leq p}\cap X'_p$ tells us that these Cauchy sequences have the same limit in $\grcomp{X'}$ as well, giving injectivity.

Next we check that $\filcomp X_{\leq p - 1} \cap \grcomp{X'}_p = 0$. Suppose we have an equality of classes of Cauchy sequences $(x_n) = (y_n)$ where $x_n \in X_{\leq p - 1}$, $y_n \in X'_p$ and $x_n - y_n \in \N^n X_{\leq p}$. 
We claim that we may replace $(x_n)$ by an equivalent Cauchy sequence $(x_n')$ with $x_n' \in X_{\leq p - d}$ for all $d > 0$. 
To see this, we argue by induction on $d$. Suppose $x_n \in X_{\leq p - (d - 1)}$ we use the fact that our seminorm is split filtered to write $x_n = x'_n + x''_n$ with $x'_n\in X_{\leq p - d}$ and $x''_n \in X_{p - (d - 1)}$. As $y_n - x_n = y_n - x'_n - x''_n \in \N^n X_{\leq p}$ and 
\[\N^n X_{\leq p} = \N^n X_p + \N^n X_{\leq p - 1} = \cdots = \N^n X_p \oplus \N^n X_{p - 1} \oplus \cdots \oplus \N^n X_{p - d + 1} \oplus \N^n X_{\leq p - d},\]
we see that an element of $X_{\leq p}$ lies in the $n$'th neighborhood $\N^n X_{\leq p}$ if and only if each of its factors with respect to the decomposition \[X_{\leq p} = X_p \oplus X_{p - 1} \oplus \cdots \oplus X_{p - d + 1} \oplus X_{\leq p - d},\]
lies in their corresponding $n$'th neighborhood.
We therefore find $x''_n \in \N^n X_{p - d + 1}$ for all $n$. Consequently $\lim\limits_{n \to \infty} x''_n = 0$, which says the Cauchy sequences $(x_n)$ and $(x'_n)$ are equivalent. This verifies our claim.

Now, we claim that $(y_n) = 0$. By definition of the completion, this amounts to $(y_n) \in \N^m X_{\leq p - 1}$ for all $m$. By our hypothesis, for any $m$ there exists $d'$ such that $X_{\leq -d'} \subset \N^m X_{\leq p - 1}$. In particular, choosing $d = d' + p$ in the prior argument, we find that we may choose a Cauchy sequence $(y'_n)$ equivalent to the first, with $y'_n \in X_{\leq -d'} \subset \N^m X_{\leq p - 1}$, showing that $(y'_n) \in \N^m \filcomp X_{\leq p - 1}$ for all $m$, verifying our claim.

Now we check $\N^n \filcomp X_{\leq p} \subset \N^n \filcomp X_{\leq p - 1} + \N^n \grcomp{X'}_p$. For this, let $\sum x_i \in \N^n \filcomp X_{\leq p }$ be a convergent infinite series. Without loss of generality, we may assume $x_i \in \N^{n+i} X_{\leq p}$ for all $i$. As our seminorm is split-filtered, we can write $x_i = (x_i)_p + (x_i)_{< p}$ as in \cref{split parts notation}. But now we see that the sums $\sum_i (x_i)_p$ and $\sum_i (x_i)_{< p}$ both converge in $\N^n \grcomp{X'}_p$ 
and $\N^n \filcomp X_{\leq p - 1}$ respectively, showing that 
$\sum x_i \in \N^n \filcomp X_{\leq p - 1} + \N^n \grcomp{X'}_p$ as desired.

As $\filcomp X = \bigcup_n \N^n \filcomp X$ and 
$\grcomp{X'} = \bigcup_n \N^n \grcomp{X'}$ 
we conclude that $X$ is split-filtered seminormed by \cref{split colimits}.

To check that it is tightly seminormed, We notice that whenever $X_{\leq -d} \subset \N^mX_{\leq p}$ we find that, upon taking closures in $\filcomp X_{\leq p}$, that 
$\filcomp X_{\leq -d} \subset \N^m \filcomp X_{\leq p}$, showing that $\filcomp X$ is also tightly seminormed (with the same choice of $d$ for a given $m, p$).
\end{proof}

\begin{remark}
In light of this result, it makes sense to refer to $\filcomp X$ as the completion of $X$, when $X$ is split-filtered, with the understanding that the graded subgroup is given by $\grcomp{X'}$. In the case $X = \filcomp X$, we say $X$ is complete.
\end{remark}

\begin{deflem} \label{closure split}
Suppose $X$ is a split-filtered, tightly seminormed, complete Abelian group, and suppose $Y \subset X$ is a split-filtered subgroup. Define the closure $\ov Y$ of $Y$ in $X$ to be the filtered subgroup with $\ov Y_{\leq p}$ the closure of the image of $Y_{\leq p}$ in $X_{\leq p}$ and $\ov{Y'}_p$ the closure of the image of $Y'_p$ in $X'_p$. Then $\ov Y$ is split-filtered with respect to the graded subgroup $\ov Y'$. 
\end{deflem}
\begin{proof}
It follows from the definition that the restriction of a tight seminorm is again a tight seminorm. As we can identify the closures with the completions by \cref{subquotient completions}, the result follows from \cref{split completions}.
\end{proof}

\subsection{Canonical seminorms}

The seminorms used in studying universal enveloping algebras of VOAs arise in a very specific way, as described in \cite{tuy, FrenkelZhu, bzf, FrenkelLanglands, NT, MNT}. We will recall an generalized definition of these seminorms, as in \cite{MNT}, and then examine some abstract features in the context of split-filtrations, which will allow us to relate the filtered and graded versions.

\begin{defn}[The canonical seminorm]\label{canonical def}
Let $U$ be a filtered ring. The canonical system of neighborhoods on $U$ is defined by $\cN^n U = U U_{\leq -n}$ (a left ideal of $U$ if $U$ is associative). We will write $^c|\cdot|$ for the corresponding canonical seminorm.
\end{defn}

\begin{lem} \label{canonical lemma}
Suppose $U' \subset U$ is a split-filtered ring. Then the canonical seminorm is split-filtered and tight.
\end{lem}
\begin{proof}
Suppose $u \in \cN^n U_{\leq p}$. 
By \cref{split the products}, we can write $u$ as a sum of elements of the form 
$\alpha\beta$ with $\alpha \in U_{\leq a}$ and $\beta \in U_{\leq b}$ with 
$a + b = p$ and $b \leq -n$.   

Using our splitting we may write $\alpha = \ov \alpha + \alpha'$ and $\beta = \ov \beta + \beta'$ with $\ov \alpha \in U'_a$, $\alpha' \in U_{\leq a-1}$, $\ov \beta \in U'_b$, $\beta' \in U_{\leq b - 1}$, and so we have $\alpha' \ov \beta \in \cN^nU_{\leq p - 1}$ giving us:
\[\alpha \beta = \ov \alpha \ov \beta + \ov \alpha \beta' + \alpha' \ov \beta + \alpha' \beta' \in   \cN^nU'_p + \cN^nU_{\leq p - 1} +\cN^{n+1}U_{\leq p-1} = \cN^nU'_p + \cN^nU_{\leq p - 1}.\]
It follows that $\cN^n U_{\leq p} \subset \cN^nU_{\leq p - 1} + \cN^nU'_p$, and hence $\cN^n U_{\leq p} = \cN^nU_{\leq p - 1} + \cN^nU'_p$, showing the canonical seminorm is split filtered.

To check that it is tight, we simply notice that for any $m, p$, we have for $d \geq \max\{m, -p\}$, $U_{\leq -d} \subset U_{\leq p} \cap \cN^m U = \cN^m U_{\leq p}$.
\end{proof}

The following Lemmas are easily verified. 
\begin{lem} \label{canonical identity}
Let $U$ be a filtered associative ring. Then for any $p, q, n \in \ZZ$, we have
\[ \left(\cN^nU_{\leq p}\right) U_{\leq q} \subset \cN^{n-q} U_{\leq p + q} \quad \text{ and } \quad U_{\leq p}\left(\cN^nU_{\leq q}\right) \subset \cN^n U_{\leq p + q}.	\]
\end{lem}

\begin{lem} \label{canonical image}
Let $f \colon  X \to Y$ be a surjective filtered homomorphism of filtered associative rings. Then $f(\cN^nX) = \cN^nY$.
\end{lem}

By \cref{canonically continuous}, a useful property of the canonical topology is that multiplication is continuous with respect to it, at least when restricted to the various filtered parts.

\begin{lem} \label{canonically continuous}
Let $U$ be a filtered associative ring equipped with a seminorm such that
\[ \left(\N^nU_{\leq p}\right) U_{\leq q} \subset \N^{n-q} U_{\leq p + q} \quad \text{ and } \quad U_{\leq p}\left(\N^nU_{\leq q}\right) \subset \N^n U_{\leq p + q}.	\]
Then for any $p, q$, the multiplication map $U_{\leq p} \times U_{\leq q} \to U_{\leq p + q}$ is continuous with respect to the seminorm in both variables. Consequently, the completion $\filcomp U$ naturally has the structure of an associative ring.
\end{lem}
\begin{remark}
It follows that under these hypotheses, if $U$ and its seminorm is split-filtered, then the multiplication map $U'_{p} \times U'_{q} \to U'_{p + q}$ is also continuous (being the restriction of a continuous map). Consequently, in this case, the completion $\filcomp U$ is also a split-filtered associative ring, which is tightly split-filtered if $U$ is (by \cref{split completions}).
\end{remark}

\begin{proof}
Let $u_1 \in U_{\leq p}, u_2 \in U_{\leq q}$. 
Then we must show that multiplication is continuous with respect to both variables at $(u_1, u_2)$. That is, given $d \in \mathbb Z$, we must show there exist $n_1, n_2$ such that $(u_1 + \N^{n_1} U_{\leq p})u_2 \subset u_1 u_2 + \N^d U_{\leq p + q}$ and $u_1(u_2 + \N^{n_2} U_{\leq q}) \subset u_1 u_2 + \N^d U_{\leq p + q}$. 
By our hypotheses, for $n_2 \geq d$, we have $u_1(u_2 + \N^{n_2} U_{\leq q}) \subset u_1u_2 + \N^{n_2} U_{\leq p + q} \subset u_1u_2 + \N^{d} U_{\leq p + q}$. 
On the other hand, for $n_1 \geq q + d$, we find 
$(\N^{n_1} U_{\leq p}) u_2 \subset \left(\N^{n_1} U_{\leq p} \right) U_{\leq q} \subset \N^{n_1 - q} U_{\leq p + q} \subset \N^d U_{\leq p + q},$
as desired.
\end{proof}

\begin{remark} \label{good and tight}
If $U$ is a split-filtered seminormed ring with $\cN^n U_{\leq p} \subset \N^n U_{\leq p}$ then by \cref{canonical lemma}, it is tightly seminormed.
\end{remark}

The canonical seminorm on a split-filtered associative ring has a number of useful properties which we would like to axiomatize. As we have seen, it is tight and split-filtered (\cref{canonical lemma}) and verifies the identities of \cref{canonical identity}. 
\begin{defn} \label{alm can def}
Let $U$ be a split-filtered seminormed associative ring. We say the seminorm is almost canonical if it verifies the following conditions:
\begin{enumerate}
\item \label{alm can 1} the seminorm is split filtered,
\item \label{alm can 2} $\N^nU_{\leq p} = \cN^nU_{\leq p} + \N^{n+1}U_{\leq p}$ for all $n, p$,
\item \label{alm can 3} $\left(\N^nU_{\leq p}\right) U_{\leq q} \subset \N^{n-q} U_{\leq p + q}$ and $U_{\leq p}\left(\N^nU_{\leq q}\right) \subset \N^n U_{\leq p + q}$ for all $p, q, n$.
\end{enumerate}

\end{defn}

\begin{lem} \label{alm can graded lem}
Let $U$ be a split-filtered almost canonically seminormed associative ring. Then 
$\N^nU'_{p} = \cN^nU'_{p} + \N^{n+1}U'_{p}$ for all $n, p$.
\end{lem}
\begin{proof}
Using the fact that the seminorm is split filtered and \cref{alm can def}~\eqref{alm can 2}, we have
$ \N^nU'_{p} + \N^nU_{\leq p - 1} = \cN^nU'_{p} + \cN^nU_{\leq p - 1} + \N^{n+1}U'_{p} + \N^{n+1}U_{\leq p - 1}$
from which the result follows looking modulo $U_{\leq p - 1}$. 
\end{proof}

\begin{lem} \label{almost canonical conditions}
Let $U$ be a split-filtered seminormed associative ring. Then the following are equivalent:
\begin{enumerate}
	\item \label{alm can condition} $\N^nU_{\leq p} = \cN^nU_{\leq p} + \N^{n+1}U_{\leq p}$ for all $n, p$,	\item $\N^nU_{\leq p} = \cN^nU_{\leq p} + \N^{n+d}U_{\leq p}$ for all $n, p$ and $d > 0$,
	\item \label{alm can con 3} $\cN^nU_{\leq p}$ is contained in and dense in $\N^n U_{\leq p}$.
\end{enumerate}	
\end{lem}
\begin{proof}
This follows by iterating the expression in part~\eqref{alm can condition}.	
\end{proof}

\begin{remark} \label{almost canonical is tight}
    In particular, if $U$ carries an almost canonical seminorm, then by \cref{alm can def}\eqref{alm can 2}, it satisfies the equivalent conditions of \cref{almost canonical conditions}, and from \cref{almost canonical conditions}\eqref{alm can con 3}, it follows from \cref{good and tight} that its seminorm is also tight. 
\end{remark}

\begin{lem} \label{almost canonical image}
Let $f \colon X \to Y$ be a surjective map of split-filtered associative rings, and suppose $X$ is endowed with an almost canonical seminorm. Then the system of neighborhoods $\N^n Y_{\leq p} = f(\N^n X_{\leq p})$ defines an almost canonical seminorm on $Y$.
\end{lem}
\begin{proof}

By \cref{canonical image} the image of the canonical neighborhoods in $X$ are canonical neighborhoods in $Y$, and by definition the images of neighborhoods in $X$ are neighborhoods in $Y$. The result then follows directly by applying the homomorphism $f$ to the properties of \cref{alm can def}~\eqref{alm can 2} and~\eqref{alm can 3}.
\end{proof}

\begin{lem} \label{almost canonical completion}
Suppose $U$ is a split-filtered associative ring with an almost canonical seminorm. Then the induced norm on the filtered completion $\filcomp{U}$ is also almost canonical.
\end{lem}
\begin{proof}
By \cref{good and tight} and \cref{split completions}, $\filcomp U$ is split-filtered and tightly seminormed, implying that $\filcomp U$ satisfies \cref{alm can def}~\eqref{alm can 1}. We proceed to \cref{alm can def}~\eqref{alm can 2} using the equivalent conditions of \cref{almost canonical conditions}.
As the neighborhoods $\N^n\filcomp U_{\leq p}$ can be identified as the closure of the image of $\N^n U_{\leq p}$ and $\cN^nU_{\leq p}$ is dense in $\N^n U_{\leq p}$, it follows that the image of $\cN^nU_{\leq p}$ is dense in $\N^n \filcomp{U}_{\leq p}$. But as the image of $\cN^n U_{\leq p}$ is contained in $\cN^n \filcomp{U}_{\leq p}$, it follows that $\cN^n \filcomp{U}_{\leq p}$ is also dense in $\N^n\filcomp{U}_{\leq p}$, verifying \cref{alm can def}~\eqref{alm can 2}.

We verify the first part of \cref{alm can def}~\eqref{alm can 3} (the second part is analogous). The multiplication map $\N^n U_{\leq p} \times U_{\leq q} \to U_{\leq p + q}$  is continuous by \cref{canonically continuous} and it factors through  $\N^{n-q} U_{\leq p + q}$. By continuity, taking closures (of the images) in the completions $\filcomp U_{\leq p}, \filcomp U_{\leq q}, \filcomp U_{\leq p + q}$ of $U_{\leq p}, U_{\leq q}, U_{\leq p + q}$ respectively, we find that our map extends to a continuous map $\ov{\N^n U_{\leq p}} \times \ov{U_{\leq q}} \to \ov{U_{\leq p + q}}$ which factors through $\ov{\N^{n - q} U_{\leq p + q}}$. Since the closure of the image in a completion can be identified with the completion itself, and $\ov{\N^n U_{\leq p}} = \N^n\filcomp U_{\leq p}$, $\ov{\N^{n - q} U_{\leq p + q}} = \N^{n - q}\filcomp U_{\leq p + q}$, we interpret our multiplication as a continuous map $\N^n \filcomp U_{\leq p} \times {\filcomp U_{\leq q}} \to {\filcomp U_{\leq p + q}}$ which factors through $\N^{n - q} \filcomp U_{\leq p + q}$, as desired.
\end{proof}

\subsection{Completed tensors}
Completed tensors, introduced here in \cref{def:CT}, make a number of arguments  more natural.

\begin{defn}[Seminorm on tensors] \label{tensor neighborhoods}
Let $R$ be a seminormed ring, $M$ a right seminormed $R$-module and $N$ a left seminormed $R$-module. We define a seminorm on $M \otimes_R N$ by the following neighborhoods of $0$:
\[
\Nn{n} (M \otimes_R N) = 
\sum_{p + q = n} \im\big((\Nn{p}M \otimes_{\Nn{0} R} \Nn{q} N) \to M \otimes_R N\big).
\]
\end{defn}
\begin{defn}\label{def:CT}(Complete tensors)
Let $R$ be a seminormed ring, $M$ a right seminormed $R$-module, and $N$ a left seminormed $R$-module. The complete tensor product $M \wh \otimes_R N$ is defined to be the completion of the seminormed Abelian group $M \otimes_R N$ with seminorm as described in \cref{tensor neighborhoods}.
\end{defn}

\begin{deflem}[Complete tensors, filtered and graded] \label{complete ft def}
Let $R$ be a seminormed ring, $M$ a right seminormed $R$-module and $N$ a left seminormed $R$-module. If $R, M, N$ are graded then we can construct a short homomorphism $M \times N \to M \gctensor_R N$ which is universal for $R$-bilinear maps to graded-complete Abelian groups. If $R, M, N$ are filtered then we can construct a short homomorphism $M \times N \to M \fctensor_R N$ which is universal for $R$-bilinear maps to filtered-complete Abelian groups.
\end{deflem}

\begin{proof}
These are  $M \gctensor_R N = \grcomp{M \otimes_R N}$ and $M \fctensor_R N = \filcomp{M \otimes_R N}$ respectively.
\end{proof}

\subsection{Discrete quotients}

This section will be particularly useful in construction of generalized Verma modules and new algebraic structures (the mode transition algebras of \cref{sec:ModeAlgebra}) which will play an important role for us.

If $U$ is a filtered ring, then $U_{\leq 0}$ is always a subring and $U_{\leq -n}$ for $n > 0$ is a two-sided ideal of $U_{\leq 0}$. Moreover, for $n > 0$, we have 
$U \otimes_{U_{\leq 0}} U_{\leq 0}/U_{\leq -n} = U/UU_{\leq -n} = U/\cN^nU$.

\begin{lem} \label{discrete lemma}
Suppose $U$ is a split-filtered almost canonically seminormed ring. 
Then
\[U \fctensor_{U_{\leq 0}} U_{\leq 0}/U_{\leq -n} \cong 
\filcomp U/[\N^n \filcomp U \cong 
\grcomp{U'}/\N^n \grcomp{U'} \cong U' \gctensor_{U'_{\leq 0}} U'_{\leq 0}/U'_{\leq -n}\]
with isomorphism induced by the continuous map 
$U \otimes_{U_{\leq 0}} U_{\leq 0}/U_{\leq -n} \to U/\N^n U$ via $u \otimes \ov{a} \mapsto \ov{ua}$.

In particular, as a topological space, these have the discrete topology.
\end{lem}
\begin{proof}
It is immediate that, assuming the claimed equalities hold, the natural quotient topology on $\grcomp{U'}/\N^n \grcomp{U'}$ is discrete.

As we have noticed, $U \otimes_{U_{\leq 0}} U_{\leq 0}/U_{\leq -n} \cong U/\cN^nU$. We can therefore identify the separated completion of $U_{\leq p} \otimes_{U_{\leq 0}} U_{\leq 0}/U_{\leq -n}$ with the separated completion of $U_{\leq p}/\cN^nU_{\leq p}$. But since $\ov{\cN^nU_{\leq p}} = \N^n U_{\leq p}$ by \cref{almost canonical conditions}, the isomorphism $U \fctensor_{U_{\leq 0}} U_{\leq 0}/U_{\leq -n} \cong \filcomp U/\N^n \filcomp U$ follows by \cref{subquotient completions}.

Next, we note that the natural map $\grcomp{U'} \to \filcomp U/\N^n \filcomp U$ has kernel $\N^n \grcomp{U'}$. As $U'_{\leq -m} \subset \cN^n U \subset \N^n U$ for $m \gg0$ it follows that our map $U' \to \filcomp U/\N^n \filcomp U$, which has dense image by \cref{tight density} factors through the surjection $U' / U'_{\leq -m} \to U' / \N^n U'$. 
In particular, the restriction of this map to $U'_{\leq p} / U'_{\leq -m}$ factors through $\wh{U'_{\leq p}/U'_{\leq -m}}$ and hence the image of this part coincides with the image of $\grcomp{U'_{\leq p}}$. But $\wh{U'_{\leq p}/U'_{\leq -m}} = \bigoplus\limits_{-m < i \leq p} \wh U'_i = \bigoplus\limits_{-m < i \leq p} \grcomp{U'}_i$. We therefore find that the map $\grcomp{U'}_{\leq p} \to \filcomp U_{\leq p}/\N^n \filcomp U_{\leq p}$ factors through $\bigoplus\limits_{-m < i \leq p} \grcomp{U'}_i$ which is a complete space. As this map has dense image, it is surjective and from our prior description of the kernel, we see $\grcomp{U'}_{\leq p}/\N^n \grcomp{U'}_{\leq p} \cong \filcomp U_{\leq p}/\N^n \filcomp U_{\leq p}$. Taking a union over all $p$ gives the identification $\filcomp U \cong \grcomp{U'}/\N^n \grcomp{U'}$. 

Making the same observations as in the beginning of the proof with $U'$ instead of $U$, we may identify the separated completion of $U'_{\leq p} \otimes_{U'_{\leq 0}} U'_{\leq 0}/U'_{\leq -n}$ with the separated completion of $U'_{\leq p}/\cN^nU'_{\leq p}$. Choosing $m$ as in the previous paragraph, we find that we have a surjective map $U' / U'_{\leq -m} \to U' / \N^n U'$ which allows us to identify the separated completion of $U'_{\leq p}/\cN^nU'_{\leq p}$ with $\grcomp{U'}_{\leq p}/\N^n \grcomp{U'}_{\leq p}$ as desired.
\end{proof}

\subsection{Triples of associative algebras}\label{sec:Triples}

In this section we collect some of our previous facts which will be useful for the construction of our universal enveloping algebras of a VOA. As we simultaneously construct and relate three versions of the enveloping algebra (left, right and finite \cref{enveloping algebra defs}), we will therefore introduce notions for working with triples of associative algebras here.

\begin{defn} \label{good definition}
A good triple of associative algebras $(U^{\mathsf{L}}, U', U^{\mathsf{R}})$ consists of the data of a left split-filtered associative algebra $U^{\mathsf{L}}$,  a right split-filtered associative algebra $U^{\mathsf{R}}$, and a graded subalgebra $U'$ of both $U^{\mathsf{L}}$ and $U^{\mathsf{R}}$.
\end{defn}
\begin{defn}
A morphism of good triples $(X^{\mathsf{L}}, X', X^{\mathsf{R}}) \to 	(Y^{\mathsf{L}}, Y', Y^{\mathsf{R}})$ is a pair of degree $0$ maps of split-filtered associative algebras $X^{\mathsf{L}} \to Y^{\mathsf{L}}$ and $X^{\mathsf{R}} \to Y^{\mathsf{R}}$ which agree on $X' \to Y'$.
\end{defn}
\begin{defn} \label{good seminorm def}
A good seminorm on a good triple of associative algebras $(U^{\mathsf{L}}, U', U^{\mathsf{R}})$ consists of almost canonical split-filtered seminorms on $U^{\mathsf{L}}$ and $U^{\mathsf{R}}$ defined by neighborhoods $\N_\mathsf{L}^n U^{\mathsf{L}}$ and $\N_\mathsf{R}^n U^{\mathsf{R}}$ respectively such that $\N_\mathsf{L}^n U'_p = \N_\mathsf{R}^{n+p} U'_p$.
\end{defn}
\begin{remark} \label{zero neighborhoods}
We note that in the case $p = 0$ we have $\N_\mathsf{L}^n U'_0 = \N_\mathsf{R}^n U'_0$, and in this case we can unambiguously write $\N^n U'_0$ for each. Also in this case, it follows from \cref{alm can def}~\eqref{alm can 3} that $\N^n U'_0$ is a two sided ideal of $U'_0$.
\end{remark}

\begin{lem} \label{good ideals from homogeneous}
Suppose $(U^{\mathsf{L}}, U', U^{\mathsf{R}})$ is a good triple of associative unital algebras and $I \triangleleft U'$ is a homogeneous ideal. Let $I^{\mathsf{L}} = U^{\mathsf{L}} I U^{\mathsf{L}}$ and $I^{\mathsf{R}} = U^{\mathsf{R}} I U^{\mathsf{R}}$ be the ideals of $U^{\mathsf{L}}$ and $U^{\mathsf{R}}$ generated by $I$. Then $(I^{\mathsf{L}}, I, I^{\mathsf{R}})$ is a good triple (of ideals). 
\end{lem}
\begin{proof}
We note that the triple $(I, I, I)$ is good, where we regard $I$ itself as left and right filtered as in \cref{graded as split-filtered}. The result now follows from \cref{split ideal}, in light of the observation that the ideal generated by $I$ in $U'$ is $I$ itself.
\end{proof}

The following Lemma is an immediate consequence of \cref{closure split} (note that good seminorms give rise to almost canonical ones which are tight by \cref{almost canonical is tight}).
\begin{lem} \label{good ideal closures}
Let $(U^{\mathsf{L}}, U', U^{\mathsf{R}})$ be a good triple of associative unital algebras and $(I^{\mathsf{L}}, I, I^{\mathsf{R}})$ a good triple of ideals. Then the closures $(\ov I^{\mathsf{L}}, \ov I, \ov I^{\mathsf{R}})$ is a good triple of ideals.
\end{lem}

\begin{remark} \label{good seminorm remark}
If our seminorms on a triple $(U^{\mathsf{L}}, U', U^{\mathsf{R}})$ are canonical, they are easily verified to be good: it is split-filtered by \cref{canonical lemma} and satisfies the other conditions of \cref{alm can def} by definition of the canonical seminorm and by \cref{canonical identity}.
\end{remark}

\begin{defn}
If $(U^{\mathsf{L}}, U', U^{\mathsf{R}})$ is a good triple with a good seminorm, it's completion is the triple $\left(\filcomp{U^{\mathsf{L}}}, \grcomp{U'}, \filcomp{U^{\mathsf{R}}}\right)$. We say that triple is complete if these are the same seminormed triple under the canonical map.
\end{defn}

The following two results show that, in the appropriate sense, the class of good triples with good seminorms are closed under completions and homomorphic images.

\begin{cor}\label{good quotients}
If $(X^{\mathsf{L}}, X', X^{\mathsf{R}}) \to 	(Y^{\mathsf{L}}, Y', Y^{\mathsf{R}})$ is a surjective map of good triples, and $(X^{\mathsf{L}}, X', X^{\mathsf{R}})$ has a good seminorm, the induced seminorm on $(Y^{\mathsf{L}}, Y', Y^{\mathsf{R}})$ is good.
\end{cor} 

\begin{proof}
This is an immediate consequence of \cref{cokernel splitting} and \cref{almost canonical image}.	
\end{proof}

\begin{cor} \label{good completions}
Good triples with good seminorms are closed under the operation of completion.
\end{cor}
\begin{proof}
This is an immediate consequence of \cref{almost canonical completion}.	
\end{proof}

\section{Generalized Verma modules and mode algebras}\label{sec:Generalized}

In this section, our basic object will be a graded seminormed algebra. While such an algebra may come as part of a triple as described in the previous section, the graded structure will play the decisive role here. We will, however, occasionally regard our graded algebra as also (split-)filtered as in \cref{graded as split-filtered}.

\subsection{Generalized higher Zhu algebras and Verma modules}\label{sec:GHZ}

\begin{defn}
\label{def:GHZ}
For a graded, seminormed unital algebra $U$, we define the \textit{generalized n-th Zhu algebra} as $\Aa_n(U) = U_0/\Nn {n+1} U_0$.
\end{defn}

For $\alpha \in U_0$, we write $[\alpha]_n$ to denote the image of $\alpha$ in $\Aa_n(U)$, and write $[\alpha]$ if $n$ is understood. Observe that $\Aa_n(U)=0$ if $n \leq -1$ since $\Nn{i}U_0=U_0$ whenever $i \leq 0$.

\begin{defn}
If $U$ is a graded algebra with an almost canonical seminorm and $W_0$ is a left $\Aa_n(U)$-module, we define a $U$-module $\PhiL_n(W_0)$ by
\[\PhiL_n(W_0) = \left(U/\NL {n+1} U\right) \otimes_{U_0} W_0 = \left(U/\NL {n+1} U\right) \otimes_{\Aa_n(U)} W_0.\]
\end{defn}

We will generally write $\PhiL(W_0)$ for $\PhiL_0(W_0)$. Following \cite{FrenkelZhu, DongLiMason:HigherZhu} we define:

\begin{defn}
If $U$ is a graded algebra with an almost canonical seminorm and $W$ is a left $U$-module, we define an $\Aa_n(U)$-module $\Omega_n(W)$ by
\[\Omega_n(W) = \{w \in W \mid (\NL {n+1} U)w = 0\}.\]
\end{defn}

We can show that the functors $\PhiL$ have the following universal property:
\begin{prop} \label{universal property}
Let $M$ be a $U$-module and $W_0$ an $\Aa_n(U)$-module. Then there is a natural isomorphism of bifunctors:
\[\Hom_{\Aa_n(U)}(W_0, \Omega_n(M)) = \Hom_{U}(\PhiL_n(W_0), M).\]
\end{prop}
In \cref{sec:Quotient} we use \cref{universal property} (there given by \cref{prop:UnivPaper}) to conclude that Zhu's original induction functor is naturally isomorphic to $\PhiL$.

\begin{proof} We describe the equivalence as follows. For $f \colon  W_0 \to \Omega_n(M)$ we define a map $g \colon \PhiL_n(W_0) \to M$ by $g(u \otimes m) = uf(m)$. Note that if $u \in \NL {n+1} U$ then $uf(m) = 0$ as $f(m) \in \Omega_n(M)$. In the other direction, if we are given $g\colon \PhiL_n(W_0) \to M$, we note that the natural map $W_0 \to \PhiL_n(W_0)$ defined by $w \mapsto 1 \otimes w$ is injective and by definition of the $U$-module structure of $\PhiL_n(W_0)$, has image lying inside $\Omega_n(\PhiL_n(W_0))$. But as the map $g$ is a $U$-module map, it follows that $g(W_0)\subset g(\Omega_n(\PhiL_n(W_0))) \subset \Omega_n(M)$. Consequently we obtain a map $f\colon W_0 \to \Omega_n(M)$ which is easily checked to be an $\Aa_n(U)$-module map and to give an inverse correspondence to the prior prescription.
\end{proof}

Of course, we can also do a right handed version of this construction for a right $\Aa_n(U)$ module $Z_0$ and obtain in way a right $U$-module $\PhiR_n(Z_0)$. We will describe the properties of $\PhiL$ and leave the analogue statements about $\PhiR$ to the reader. 

\begin{lem}
Suppose $U$ is a split-filtered algebra with graded subalgebra $U'$ and with an almost canonical seminorm . Then
\[\PhiL(W_0) = \left(U'/\NL 1 U'\right) \otimes_{U'_0} W_0 \cong \left(U/\NL 1 U\right) \otimes_{U_0} W_0.\] 
\end{lem}
\begin{proof}
This is an immediate consequence of \cref{discrete lemma}.
\end{proof}

Note that $\NL 1 U$ is a left $U$ module and a right $U_{\leq 0}$ module which is annihilated on the right by $U_{\leq -1}$ in particular, we can also write the above expression as:
\[\left(U/\NL 1 U\right) \otimes_{U_0} W_0 \cong \left(U/\NL 1 U\right) \otimes_{U_{\leq 0}} W_0,\]
with respect to the truncation  quotient map $U_{\leq 0} \to U_0$ with kernel $U_{\leq -1}$. This is because the additional relations in the tensor product on the right are of the form $\alpha \beta \otimes w - \alpha \otimes \ov{\beta} w$ with $\beta \in U_{\leq -1}$ . But $\alpha \beta \in U U_{\leq -1} \in \NL 1U$ represents $0$ as does $\ov{\beta}$. Hence these extra relations all vanish.

\begin{remark} \label{induced grading}
We see that $\PhiL(W_0)$ is naturally a graded module, with grading inherited from $U/\NL 1 U$:
\[\PhiL (W_0) = \bigoplus_{p = 0}^\infty \left(U/\NL 1 U\right)_p \otimes_{U_0} W_0 = \bigoplus_{p = 0}^\infty \left(U_p/\NL 1 U_p\right)\otimes_{U_0} W_0.\]
Notice here that $U_{-m} \subset \NL m U$ and so $\left(U/\NL 1 U\right)_{p} =0$ for $p < 0$. 	
\end{remark}

\begin{lem} \label{higher zhu action}
The action of $U_0$ on $\PhiL(W_0)$ via its left module structure induces an $\Aa_d(U)$ module structure on $\PhiL(W_0)_{\leq d} = \bigoplus_{p = 0}^{d} \PhiL(W_0)_p$.
\end{lem}
\begin{proof}
We have $U_{\leq -d - 1} \PhiL(W_0)_{\leq d} = 0$ from degree considerations. It follows that 
\[(\cN^{d + 1} U_0) \PhiL(W_0)_{\leq d} = 0.\]
But by \cref{almost canonical conditions} $\cN^{d + 1} U_0$ is dense in $\Nn {d + 1} U_0$ and by \cref{canonically continuous}, the multiplication action of $U_0$ on $U_{\leq d}$ is continuous, and hence so is the multiplication of $U_0$ on $U_{\leq d}/\NL 1 U_{\leq d}$ and hence of $U_0$ on $\PhiL(W_0)$. But as $U/\NL 1 U$ has a discrete topology, so does $\PhiL(W_0)$. Since a dense subset of $\Nn{d + 1} U_0$ acts as zero, it follows that it acts as zero, making the action of the algebra $\Aa_d(U)$ well defined.
\end{proof}

\subsection{Generalized mode transition algebras}\label{sec:GMTA}

\cref{innerstar lemma} is the main technical tool used to define algebraic structures and their actions on generalized Verma modules:

\begin{lem}\label{innerstar lemma}
Suppose $U$ is a graded algebra with an almost canonical seminorm. Then we have a natural isomorphism
\begin{equation*} \left(U/\NR 1 U\right) \otimes_{U}
\left(U/\NL 1 U\right) \to
\Aa_0(U), \qquad
\ov{\alpha} \otimes \ov{\beta} \mapsto \alpha \innerstar \beta
\end{equation*}
where for $\alpha, \beta \in U$ homogeneous, we define $\alpha \innerstar \beta$ as follows:
\[\alpha \innerstar \beta = \begin{cases}
0 & \text{if } \deg(\alpha) + \deg(\beta) \neq 0 \\
[\alpha \beta] & \text{if } \deg(\alpha) + \deg(\beta) = 0
 \end{cases}
\]
and then extend the definition to general products by linearity.

\end{lem}

\begin{proof}
As our seminorm is almost canonical, the map $U_0 \to {{U}}/{\left(\NL{1} U + \NR{1} U\right)}$ factors through $U \to {{U}}/{\left(UU_{\leq -1} + U_{\geq 1}U\right)}$. But for this map, we see that both $U_{\leq -1}$ and $U_{\geq 1}$ are in the kernel, which implies that the restriction to $U_0$ is surjective. The kernel of this map $U_0 \to {{U}}/{\left(\NL{1} U + \NR{1} U\right)}$ consists of $\NL{1} U_0 \cap \NR{1} U_0 = \Nn{1} U_0$ (see \cref{zero neighborhoods}).
\end{proof}

As an application of the above result, we obtain the following.
\begin{cor}\label{cor:Compare}Let $W_0$ be a left $\Aa_0(U)$-module and $Z_0$ be a right $\Aa_0(U)$-module. Then the map defined in \cref{innerstar lemma} induces an isomorphism 
\[\PhiR(Z_0) \otimes_{U} \PhiL(W_0) \to Z_0 \otimes_{\Aa_0(U)} W_0.\]
\end{cor}

\begin{defn} \label{bimod induce def}
For a graded algebra with almost canonical seminorm $U$, and an $\Aa_0(U)$-bimodule $B$, we define a bigraded group: 
\begin{align*}
\Phi(B) &= \PhiR(\PhiL(B)) = \PhiL(\PhiR(B))= \left(U/\NL 1 U\right) \otimes_{U_{ 0}} B \otimes_{U_{ 0}} \left(U/\NR 1 U\right)	\\
&=\bigoplus_{d_1 \geq 0} \ \bigoplus_{d_2 \leq 0} \left(U/\NL 1 U\right)_{d_1} \otimes_{U_{ 0}} B \otimes_{U_{ 0}} \left(U/\NR 1 U\right)_{d_2}.
\end{align*}
\end{defn}

We now introduce the space $\Phi(B)$ and the operation $\star$ arising from $\ostar$ which, as we show below, defines an algebra structure on $\Phi(B)$ whenever $B$ is an associative ring admitting a homomorphism $f \colon  \Aa_0(U) \to B$.

\begin{defn}\label{star product} 
Let $B$ be an associative ring admitting a homomorphism $f \colon  \Aa_0(U) \to B$ and let $W_0$ be a left $B$-module. Then we can define a map $\Phi(B) \times \PhiL(W_0) \to \PhiL(W_0)$ as follows. For $x = \alpha \otimes a \otimes \alpha' \in \Phi(B)$ and $\beta \otimes w \in \PhiL(W_0)$ we set
\[x \star (\beta\otimes w) = \alpha \otimes a f(\alpha' \innerstar \beta) w.\]
\end{defn}

\begin{prop}\label{lem:abstract-module-associativity}  The map defined in \cref{star product} defines an associative algebra structure on $\Phi(B)$ such that the above action of $\Phi(B)$ on $\PhiL(W_0)$  defines a left module structure. Moreover, $\gamma \cdot (x \star \beta) = (\gamma \cdot x) \star y$ for every $x \in \Phi(B)$, $y \in \PhiL(W_0)$, and  $\gamma \in U$. Analogously,  $(x \star y) \cdot \gamma = x \star (y \cdot \gamma)$ for every $x, y \in \Phi(B)$ and $\gamma \in U$. Finally, with respect to the bigrading of \cref{bimod induce def}, we have
\begin{align*}
    &\Phi(U)_{d_1, d_2} \star \Phi(U)_{d_3, d_4} \subseteq \Phi(U)_{d_1, d_4} \qquad \text{ and }\\
    &\Phi(U)_{d_1, d_2} \star \Phi(U)_{d_3, d_4} = 0 \qquad \text{ whenever $d_2 + d_3 \neq 0$}.
\end{align*}
\end{prop}

\begin{proof}
We check first that this satisfies the standard associativity relationship for a module action. Let $\alpha \otimes a \otimes \alpha', \beta \otimes b \otimes \beta' \in \Phi(B)$ and $\gamma \otimes c \in \PhiL(W_0)$, then 
\begin{align*}
(\alpha \otimes a \otimes \alpha') \star \big((\beta \otimes b \otimes \beta') \star (\gamma \otimes c)\big) &=
(\alpha \otimes a \otimes \alpha') \star (\beta \otimes b f(\beta' \innerstar \gamma) c) \\
&= (\alpha \otimes a f(\alpha' \innerstar \beta) b f(\beta' \innerstar \gamma) c) \\
&= (\alpha \otimes a f(\alpha' \innerstar \beta) b \otimes \beta' ) \star (\gamma \otimes c) \\
&= \big((\alpha \otimes a \otimes \alpha') \star (\beta \otimes b \otimes \beta')\big) \star (\gamma \otimes c),
\end{align*}
as desired. The associativity of the algebra structure now follows, taking $W_0 = \PhiR(B)$. 

We now check the compatibility with the $U$-module structure on the left. Set $x=\alpha \otimes a \otimes \alpha'$ and $y= \beta \otimes b$. Then  
\begin{align*}
	\gamma \cdot \big((\alpha \otimes a \otimes \alpha') \star (\beta \otimes b)\big) &= \gamma \cdot (\alpha \otimes a f(\alpha' \innerstar \beta) b ) \\
	&= (\gamma\alpha) \otimes a f(\alpha' \innerstar \beta) b 
	= ((\gamma\alpha) \otimes a \otimes \alpha') \star (\beta \otimes b).
\end{align*}
The other-handed version of the above argument gives us the compatibility with the $U$-module structure on the right when $W_0 = \PhiR(B)$.

The last assertion follows from the product $\innerstar$ as described in \cref{innerstar lemma}.
\end{proof}

\begin{defn} \label{def:GoodTripleMTA notriple}
For a graded algebra with almost canonical seminorm $U$, we define $\Ac(U) = \Phi(\Aa_0(U))$ to be the \textit{(generalized) mode transition algebra}, and we write $\Ac(U)_d$ for the $d$th mode transition subalgebra $\Ac(U)_{d, -d}$. 
\end{defn}

In this case, it turns out that the action of $\Ac(U)$ extends to much more general modules than simply the Verma-type modules, as we now observe.

\begin{lem} \label{who kills continuous}
Suppose $U$ is a graded algebra with a good seminorm, and let $W$ be an $\NN$-graded $U$-module endowed with the discrete topology, such that the action of $U \times W \to W$ is continuous. Then for all $d \in \NN$, $\NR {m+1}  U_{m - d} W_d = 0 = \NL {d + 1} U_{m - d} W_d$.
\end{lem}
\begin{proof}
As a good seminorm is almost canonical, we may iteratively apply \cref{alm can graded lem}, to see that $^c\!\NR {m+1} U_{m - d}$ is dense in $\NR {m+1} U_{m - d}$. Consequently, if $\alpha \in \NR {m+1} U_{m - d}$ and $w \in W_d$, then by continuity, we can find $\alpha' \in \mbox{}^c\!\NR {m+1} U_{m - d}$ such that $\alpha' w = \alpha w$. But therefore $\alpha' w \in U U_{-d-1}w = 0$.
As the seminorm is good, we find $\NL {d + 1} U_{m - d} = \NR {d + 1} U_{m - d}$, giving the final claim.
\end{proof}

\begin{lem}\label{well defined multiplication}
Suppose $U$ is a graded algebra with a good seminorm, and let $W$ be an $\NN$-graded $U$-module endowed with the discrete topology, such that the action of $U \times W \to W$ is continuous. Then the left action of $U$ on $W$ induces well-defined maps for every $p$:
\begin{align*}
(U_p / \NR {p + d + 1} U_p) \times W_d &\to W_{p + d}, \\
([u], w) & \mapsto uw,
\end{align*}
which we will simply write as $[u]w$.
\end{lem}
\begin{proof}
This follows from \cref{who kills continuous}, which shows $(\NR {p + d + 1} U_p) W_d = 0$.
\end{proof}

\begin{prop} \label{general module action}
Let $U$ be a graded algebra with an almost canonical seminorm, and let $W$ be an $\NN$-graded $U$-module endowed with the discrete topology, such that the action of $U \times W \to W$ is continuous. Then we obtain an action of $\Ac(U)$ on $W$ 
such that $\Ac(U)_{d_1, -d_2} W_{d_2} \subset W_{d_1}$ and $\Ac(U)_{d_1, -d_2} W_{d_3} = 0$ when $d_2 \neq d_3$ which, on simple tensors, is given by the products defined in \cref{well defined multiplication}:
\[(\alpha \otimes a \otimes \alpha') \star w = \begin{cases}
\alpha (a (\alpha' w)) & \text{if $d_2 = d_3$} \\
0 & \text{otherwise.}
\end{cases} 
\]
\end{prop}
\begin{proof}
Explicitly, if we write $\alpha = [u], a = [x], \alpha' = [v]$ for $u \in U_{d_1}, v \in U_{-d_2}, x \in U_0$, we have by definition (\cref{well defined multiplication}), $\alpha (a (\alpha' w)) = u(x(v(w)))$. It follows from \cref{well defined multiplication} that this is a well defined map. Further, associativity of these products follows from the associativity of the action of $U$ on $W$.
\end{proof}

\subsection{Relationship with higher generalized Zhu algebras}

Throughout this section, let us fix a graded algebra $U$ complete with respect to an almost canonical seminorm. We will write $\Aa_n$ in place of $\Aa_n(U)$ and $\Ac_n$ in place of $\Ac(U)_n$. 

The action of $U_0$ on $\Ac_n$ is continuous where $\Ac_n$ is defined as a discrete module. 

\begin{lem}\label{lem:right exact seq}
For each $d \geq 0$, there is an exact sequence
\begin{equation}\label{right exact seq} 
\Ac_d \overset{\mu_d}{\longrightarrow} \Aa_d \overset{\pi_d}{\longrightarrow} \Aa_{d-1} \longrightarrow 0, \end{equation}  where $\mu_d(\overline{\alpha} \otimes [u]_0 \otimes \overline{\beta})= [\alpha u \beta]_d$, for all $\alpha \in U_d$ (respectively $\beta \in U_{-d}$) and where  $\ov{\alpha}$ (respectively $\ov{\beta}$) denotes its class in $U/\NL{1}U$ (respectively in $U/\NR{1}U$).
\end{lem}
\begin{proof}
We first check that the map $\mu_d$ is well defined. Notice  $\mu_d$ is independent on the lifts of $\ov{\alpha}$ and $\ov{\beta}$ to $U$, since $\NL{1}U \cdot U_0 \cdot U_d \subseteq \Nn{1}U_0$ and similarly $U_d \cdot U_0 \cdot \NR{1}U_{-d} \subseteq  \Nn{1}U_0$. Analogously, since $U_d \cdot \Nn{1}U_0 \cdot U_{-d} \subseteq \Nn{d}U_0$, the map $\mu_d$ is independent of the lift of $[u]_0$ to $u \in U_0$. Finally, we need to show that it respects the tensor products over $U_0$. For this we need to check that $\mu_d(\ov{\alpha v} \otimes [u]_0 \otimes \ov{\beta}) = \mu_d(\ov{\alpha} \otimes [vu]_0 \otimes \ov{\beta})$ for every $v \in U_0$. But by definition both are the class of the element $\alpha vu \beta$ in $\Aa_d$. 

We have identifications $\Aa_d = U_0/\Nn{d+1}U_0$ and $\Aa_{d - 1} = U_0/\Nn{d}U_0$. Consequently the kernel of the canonical projection $\pi_d$ can be written as $\Nn{d}U_0/\Nn{d+1}U_0$. It follows from the definition of $\mu_d$ and $\Ac_d$ that the image of the $\mu_d$ consists exactly of sums of element of the form $[\alpha \beta]_d$ with $\deg \alpha = d$ and  $\deg \beta = - d$. Hence the image of $\mu_d$ consists of the image of $\cN^{d}U_0$ in $\Aa_d$. Since we have an almost canonical filtration, we have by \cref{alm can graded lem}, 
\[ \Nn d U_0 = \cN^d U_0 + \Nn {d + 1} U_0,\]
which shows that $\mu_d$ is therefore surjective onto $\Nn d U_0 / \Nn {d + 1} U_0 = \ker \pi_d$, showing right exactness.
\end{proof}

The following result is immediate from the definitions and from the associativity of the actions.

\begin{lem} \label{action lemma}
Let $W_0$ be an $\Aa_0$-module. Then the action of $\Ac_d$ on $\PhiL(W_0)_d$ factors through the action of $\Aa_d$ described in \cref{higher zhu action} via the map $\mu_d$. 
\end{lem}

We now are ready to state the principal result of this section.

\begin{thm} \label{abstract splitting}
If $\Ac_d$ admits an identity element, 
then the map $\mu_d$ in \eqref{right exact seq} is injective and the sequence splits, giving a ring product
$ \Aa_d \cong \Ac_d \times \Aa_{d - 1}.$
\end{thm}
\begin{proof}

We first check that the map $\mu_d$ is injective. Suppose  we are given an element $\xx\in \Ac_d$ which is in the kernel of this map. By \cref{action lemma}, the action of $\xx$ on $\PhiL(M)_d$ for any $M$ factors through the action of $\Aa_d$ via $\mu_d$, so it should be $0$. If we consider the case of $M = \PhiR(\Aa_0)$, this says that, in particular, the action of $\xx$ on $\Ac_d \subset \PhiL(\PhiR(\Aa_0))_d$ is $0$. This action is identified with the algebra product via \cref{star product}. It follows that since $\Ac_d$ has an identity element  $\Is_d$, we have $\xx = \xx \star \Is_d = 0$ as claimed.

Since $\mu_d$ is injective we will omit it in the remainder of the proof and see $\Ac_d$ as naturally sitting inside $\Aa_d$. Denote the unity in the higher level Zhu algebras $\Aa_d$ by $1$, and write $e = \Is_d$. Let $f = 1 - e$ so that $e$ and $f$ are orthogonal idempotents. Note that $e$ generates the 2-sided ideal $\Ac_d \triangleleft \Aa_d$. Furthermore, since e is the unity of $\Ac_d$, for every $a \in \Aa_d$, we have $ae = eae = ea$ and so $e \in Z(\Aa_d)$. Consequently $f = 1 - e \in Z(\Aa_d)$ as well. It follows that $f$ and $e$ are orthogonal central idempotents, and therefore $\Aa_d = \Aa_d e \times \Aa_d f$ as rings. But $\Aa_d e = \Ac_d$ and $\Aa_d f \cong \Aa_d / \Ac_d \cong \Aa_{d - 1}$ completing our proof.
\end{proof}

\bibliographystyle{alpha}
\newcommand{\etalchar}[1]{$^{#1}$}

\end{document}